%% file: HMS-II-v4-05-2024.tex
\documentclass[12pt,final]{amsart}
\usepackage{euscript,nicefrac,fancybox,xcolor}
\usepackage{xspace,mathrsfs}
\usepackage{inputenc}
\usepackage{amscd,tikz}
\usepackage{tikz-cd}
\usepackage{latexsym}
\usepackage{xr,xr-hyper}
\usepackage{epsfig,pxfonts, yfonts}
\usepackage{ifpdf}
\usepackage{euscript,nicefrac,fancybox,xcolor}
\usepackage[nomargin]{fixme}
\usepackage{xspace,mathrsfs}
\usepackage{inputenc}
\usepackage{amscd,tikz}
\usepackage{latexsym}
\usepackage{xr,xr-hyper}
\usepackage{epsfig,pxfonts, yfonts}
\usepackage{ifpdf}
\input xy
\xyoption{all}
\usepackage[stable]{footmisc}

  \usepackage{hyperref}
  \hypersetup{
final,
linktoc=all
}
\usepackage[margin=3cm,footskip=25pt,headheight=21pt]{geometry}
\usepackage{fancyhdr}
\newtheorem{theorem}{Theorem}[section]
\newtheorem*{theorem*}{Theorem}

\newtheorem{lemma}[theorem]{Lemma}
\newtheorem{proposition}[theorem]{Proposition}
\newtheorem{corollary}[theorem]{Corollary}

\newtheorem{conjecture}[theorem]{Conjecture}

\newtheorem*{conjecture*}{Conjecture}
\newtheorem*{proposition*}{Proposition}
\newtheorem{itheorem}{Theorem}

{
\theoremstyle{plain}
\newtheorem{definition}[theorem]{Definition}
\newtheorem{example}[theorem]{Example}

\newtheorem{remark}[theorem]{Remark}

\newtheorem{warning}[theorem]{Warning}
}

\renewcommand{\theequation}{\arabic{section}.\arabic{equation}}
\newcounter{subeqn}

\makeatletter
\@addtoreset{subeqn}{equation}

\newcommand{\nc}{\newcommand}

\nc{\PR}{D^i\!R}
\nc{\rola}{X}
\nc{\wela}{Y}
\nc{\Lcat}{\mathsf{DQ}}
\nc{\lcat}{\mathsf{dq}}

\nc{\cat}{\mathcal{V}}
\nc{\func}{\EuScript{T}}
\nc{\res}{\operatorname{res}}
\newcommand{\Spec}{\operatorname{Spec}}
\nc{\sL}{\EuScript{L}}
\nc{\tsL}{\tilde{\sL}}

\newcommand{\Pic}{\operatorname{Pic}}

\renewcommand{\Re}{\operatorname{Re}}

\newcommand{\id}{\operatorname{id}}

\renewcommand{\dim}{\operatorname{dim}}

\newcommand{\Sym}{\operatorname{Sym}}
\nc{\bQ}{\mathbb{N}}
\nc{\op}{\operatorname{op}}
\newcommand{\rk}{\operatorname{rk}}

\nc{\sEnd}{\mathcal{E}\mathit{nd}}

\newcommand{\chamber}{\Delta}

\newcommand{\Ba}{\mathbf{a}}

\newcommand{\bP}{\mathbf{P}}

\newcommand{\Bx}{\mathbf{x}}
\newcommand{\By}{\mathbf{y}}

\nc{\sz}{\mathsf{z}}
\nc{\sw}{\mathsf{w}}
\nc{\sx}{\mathsf{x}}
\nc{\sy}{\mathsf{y}}
\nc{\su}{\mathsf{u}}
\nc{\sv}{\mathsf{v}}
\nc{\dF}{\mathsf{F}}
\nc{\dE}{\mathsf{E}}
\nc{\lle}{\Gamma}
\nc{\Dc}{d}
\nc{\pS}{S}
\nc{\ps}{\mathsf{s}}

\nc{\pt}{\mathsf{t}}

\nc{\slehat}{\mathfrak{\widehat{sl}}_e}
\nc{\sllhat}{\mathfrak{\widehat{sl}}_\ell}
\nc{\glehat}{\mathfrak{\widehat{gl}}_e}
\nc{\slnhat}{\mathfrak{\widehat{sl}}_n}
\nc{\glnhat}{\mathfrak{\widehat{gl}}_n}
\nc{\eE}{\EuScript{E}}
\newcommand{\arxiv}[1]{\href{http://arxiv.org/abs/#1}{\tt arXiv:\nolinkurl{#1}}}
\nc{\ep}{\epsilon}
\nc{\RHom}{\mathbb{R}\operatorname{Hom}}
\nc{\sHom}{\mathscr{H}\!\mathit{om}}
\nc{\eF}{\EuScript{F}}
\nc{\fF}{\mathfrak{F}}
\nc{\fE}{\mathfrak{E}}
\nc{\fI}{\mathfrak{I}}
\nc{\fp}{\mathfrak{p}}

\newcommand{\Z}{\mathbb{Z}}
\newcommand{\Q}{\mathbb{Q}}
\nc{\Qlb}{\mathbb{\bar \Q}_\ell}
\nc{\Fq}{\mathbb{F}_q}
\nc{\Fqb}{\mathbb{\bar F}_q}

\nc{\walg}{W}

\newcommand{\R}{\mathbb{R}}

\newcommand{\C}{\mathbb{C}}

\nc{\KZ}{\mathsf{KZ}}

\newcommand{\DQ}{\mathsf{DQ}}

\nc{\Bv}{\mathbf{v}}
  \nc{\Bw}{\mathbf{w}}
\nc{\perf}{\operatorname{-perf}}
  \nc{\Bp}{\mathbf{p}}
 \nc{\tU}{\mathcal{U}}
\nc{\Bu}{\mathbf{u}}
 \nc{\Fl}{\mathscr{F}\!\ell}
\nc{\Tr}{\operatorname{Tr}}
\nc{\cata}{\mathfrak{V}}
\nc{\tcat}{\tilde{\mathcal{V}}}
\nc{\tcata}{\tilde{\mathfrak{V}}}
\nc{\sheafK}{\EuScript{K}}
\nc{\bmu}{\boldsymbol{\mu}}
\nc{\bpi}{\boldsymbol{\pi}}
\nc{\dwalg}{\mathbb{W}}
\nc{\dalg}{\mathbb{T}}
\nc{\aalg}{\mathbb{A}}
\nc{\alm}{\mathscr{A}}
\nc{\Hbb}{\mathbb{H}}
\nc{\bra}{\mathscr{B}}
\nc{\bO}{\mathbb{O}}

\nc{\Kos}{\EuScript{K}}
\nc{\tilt}{\mathscr{T}}

\newcommand{\al}{\mathcal{L}^{(+)}}
\newcommand{\ml}{\mathcal{L}^{(\times)}}

\newcommand{\Hom}{\operatorname{Hom}}

\newcommand{\fd}{\mathfrak {d}}
\newcommand{\ft}{\mathfrak {t}}

\newcommand{\cP}{\mathcal{P}}

\newcommand{\cO}{\mathcal{O}}
\newcommand{\Dolfiber}{\mathcal{L}}

\newcommand{\LL}{\mathbb{L}}

\newcommand{\becircled}{\mathaccent "7017}

\newcommand{\Ext}{\operatorname{Ext}}

\newcommand{\cS}{\mathcal{S}}
\newcommand{\cF}{\mathcal{F}}

\newcommand{\cQ}{\mathscr{Q}}

\newcommand{\cC}{\mathscr{C}}

\newcommand{\Loc}{\operatorname{Loc}}

\newcommand{\excise}[1]{}

\newcommand{\End}{\operatorname{End}}

\newcommand{\fg}{\mathfrak{g}}

\newcommand{\mmod}{\operatorname{-mod}}
\newcommand{\dgmod}{\operatorname{-mod}_{\operatorname{dg}}}
\newcommand{\dgperf}{\operatorname{-perf}_{\operatorname{dg}}}

\newcommand{\Coh}{\mathsf{Coh}}
\newcommand{\Cohdg}{\mathsf{Coh}_{\operatorname{dg}}}

\newcommand{\AC}{\mathbb{C}}
\newcommand{\ACn}{\mathbb{C}^n}
\newcommand{\tm}{\tilt^{(\times)}}
\newcommand{\ta}{\tilt^{(+)}}
\newcommand{\TAno}{(T^*\C^n)^\circ}
\newcommand{\TAo}{(T^*\C)^\circ}
\newcommand{\Lotimes}{\overset{L}{\otimes}}
\newcommand{\Bfield}{\gamma}
\newcommand{\logB}{\tilde{\gamma}}
\newcommand{\GIT}{\delta}
\newcommand{\tone}{1}
\newcommand{\node}{\operatorname{node}}

\newcommand{\cG}{\mathcal{G}}
\newcommand{\cR}{\mathcal{R}}

\FXRegisterAuthor{w}{w}{Ben}
\FXRegisterAuthor{g}{g}{Benjamin}
\FXRegisterAuthor{m}{m}{Michael}
\FXRegisterAuthor{v}{v}{Vivek}
\FXRegisterAuthor{h}{h}{Justin}
\makeatletter
\renewcommand*\FXLayoutInline[3]{%
  {\@fxuseface{inline}
  
  \ignorespaces\noindent \ovalbox{\hspace{.03\textwidth} \begin{minipage}{.9\textwidth}
  	#3 \fxnotename{#1}: #2
  \end{minipage}\hspace{.03\textwidth}}}}
\makeatother

\newcommand{\mtodo}[1]{\fxnote[inline,author=Michael]{\color{magenta!50!black}#1}

}

\newcommand{\thetitle}{Homological mirror symmetry for hypertoric varieties II}
\newcommand{\theshorttitle}{Homological mirror symmetry for hypertoric varieties II}
\nc{\iwedge}[1]{\bigwedge\nolimits^{\! #1}}
\nc{\wedgep}[1]{\iwedge{#1}\C^n}
\nc{\fsl}{\mathfrak{sl}}
\nc{\sln}{\mathfrak{sl}_n}
\nc{\divisor}{\mathscr{D}}

\newcommand{\Fuk}{\mathsf{Fuk}}
\newcommand{\Sh}{\mathsf{Sh}}
\newcommand{\Shc}{\mathsf{Sh}^c}
\newcommand{\Perv}{\mathsf{Perv}}
\newcommand{\D}{\mathbb{D}}

\newcommand{\mmm}{\becircled{\mu}}
\newcommand{\mht}{\mathfrak{U}}
\newcommand{\aht}{\mathfrak{M}}
\newcommand{\cH}{\mathcal{H}}
\newcommand{\cL}{\mathcal{L}}

\newcommand{\Arg}{\operatorname{Arg}}

\newcommand{\Bper}{\textgoth{A}^{\operatorname{per}}}
\newcommand{\Btor}{\textgoth{A}^{\operatorname{tor}}}
\newcommand{\Dtor}{\textgoth{D}^{\operatorname{tor}}}
\newcommand{\Dper}{\textgoth{D}^{\operatorname{per}}}

\newcommand{\tilLL}{\widetilde{\mathbb{L}}}
\newcommand{\twist}{\mathfrak{Z}}
\newcommand{\fOV}{\twist}

\newcommand{\parone}{\beta}
\newcommand{\partwo}{\alpha}
\newcommand{\parr}{\gamma}
\newcommand{\scrE}{\mathscr{E}}

\newcommand{\scrH}{\mathscr{H}}
\newcommand{\scrL}{\mathscr{L}}
\newcommand{\bsL}{\overline{\mathscr{L}}}
\newcommand{\scrD}{\mathscr{D}}
\newcommand{\scrQ}{\mathscr{Q}}
\newcommand{\qsch}{\mathscr{Q}^{\text{sch}}}
\newcommand{\bqsch}{\overline{\mathscr{Q}}^{\text{sch}}}
\newcommand{\sfam}{\mathfrak{X}}
\newcommand{\frX}{\mathfrak{X}}
\newcommand{\sfamq}{\overline{\sfam}}

\newcommand{\circuit}{\sigma}

\newcommand{\mhta}{\mht_{(\parone,\tone)}}

\newcommand{\LGr}{\operatorname{LGr}}
\newcommand{\LCGr}{\operatorname{L^\C Gr}}
\newcommand{\muPerv}{\mu\mathsf{Perv}}

\newcommand{\tDolfiber}{\widetilde{\Dolfiber}}

\newcommand{\scale}{\mathbb{C}^\times}
\newcommand{\Log}{\operatorname{Log}}

\newcommand{\naiveliouville}{\mathscr{X}}
\newcommand{\tX}{\widetilde{\naiveliouville}}
\newcommand{\Dolb}{\mathfrak{D}}
\newcommand{\tDolb}{\widetilde{\Dolb}}

\newcommand{\Lieder}{\mathcal{L}}
\newcommand{\tfOV}{\widetilde{\fOV}}
\newcommand{\tfOValg}{\tfOV^{\operatorname{alg}}}
\newcommand{\momi}{\mu_I}
\newcommand{\momj}{\mu_J}
\newcommand{\momk}{\mu_K}

\newcommand{\trihamvec}{\mathscr{Y}}
\newcommand{\hamvec}{\mathscr{W}}
\newcommand{\liouvillevec}{\mathscr{Z}}

\newcommand{\Dolparone}{\mathfrak{D}_\parone}
\newcommand{\tDolparone}{\widetilde{\mathfrak{D}}_\parone}
\newcommand{\tDolalgparone}{\tDolb^{\operatorname{alg}}_\parone}

\AtBeginDocument{%
   \def\MR#1{}
}

\begin{document}

\renewcommand{\theitheorem}{\Alph{itheorem}}

\usetikzlibrary{decorations.pathreplacing,backgrounds,decorations.markings,calc,
shapes.geometric}
\tikzset{wei/.style={draw=red,double=red!40!white,double distance=1.5pt,thin}}
\tikzset{bdot/.style={fill,circle,color=blue,inner sep=3pt,outer sep=0}}
\tikzset{dir/.style={postaction={decorate,decoration={markings,
    mark=at position .8 with {\arrow[scale=1.3]{<}}}}}}
\tikzset{rdir/.style={postaction={decorate,decoration={markings,
    mark=at position .8 with {\arrow[scale=1.3]{>}}}}}}
\tikzset{edir/.style={postaction={decorate,decoration={markings,
    mark=at position .2 with {\arrow[scale=1.3]{<}}}}}}

\noindent {\Large \bf 
\thetitle}\\
{\bf \Small with an Appendix written jointly with Laurent C\^ot\'e and Justin Hilburn}
\bigskip\\
{\bf Benjamin Gammage} \\ Department of Mathematics, University of California, Berkeley \\
{\bf Michael McBreen} \\ Department of Mathematics, Chinese University of Hong Kong \\
{\bf Ben Webster}\\
Department of Pure Mathematics, University of Waterloo \& \\
Perimeter Institute for Theoretical Physics
\bigskip\\
{\small
\begin{quote}
\noindent {\em Abstract.}
In this paper, we prove a homological mirror symmetry equivalence for pairs of multiplicative hypertoric varieties, and we calculate monodromy autoequivalences of these categories by promoting our result to an equivalence of perverse schobers.  We prove our equivalence by matching holomorphic Lagrangian skeleta, on the A-model side, with noncommutative resolutions on the B-model side.  The hyperk\"ahler geometry of these spaces provides each category with a natural t-structure, which helps clarify SYZ duality in a hyperk\"ahler context. Our results are a prototype for mirror symmetry statements relating pairs of K-theoretic Coulomb branches.
\end{quote}
}
\bigskip

\listoffixmes

\section{Introduction}
\label{sec:introduction}
In this paper, a sequel to \cite{McBW}, we realize a homological mirror symmetry equivalence for multiplicative hypertoric varieties. 

Multiplicative hypertoric varieties are
variants of the more familiar toric hyperk\"ahler, or ``additive hypertoric,'' varieties defined in \cite{Goto,BD,Ko00}. 
Both multiplicative and additive hypertoric varieties appear as the simplest examples of, and hence an excellent testing ground for, a class of holomorphic symplectic spaces arising from supersymmetric gauge theory and of great interest in geometric representation theory.
 Additive hypertoric varieties appear as Coulomb branches of 3d $\mathcal{N}=4$ gauge theories, as first constructed in \cite{BFN}, with abelian gauge group. Their multiplicative cousins appear as Seiberg-Witten systems governing 4d $\mathcal{N}=2$ theories, and after compactification of the 4-dimensional theory on a circle they appear in the K-theoretic Coulomb branch construction from \cite{BFN} (see also \cite{teleman2018role,FT} for further descriptions of these spaces).

We should emphasize that while the appearances of quantum field theory above are in dimensions above 2, in this paper, we only consider ``mirror symmetry'' in the sense
of the duality of 2d $\mathcal{N}=(2,2)$ theories familiar to mathematicians from, for example, the work of Kontsevich \cite{kontsevich1995homological}.  It would be interesting to understand our results from the perspective of 3- or 4-dimensional constructions in field theory.

\subsection{SYZ mirror symmetry}
From the perspective of mirror symmetry,
the salient feature of
multiplicative Coulomb branches is the presence of a Lagrangian torus fibration: indeed, to a first approximation, these spaces are holomorphic symplectic manifolds admitting the structure of an integrable system whose generic fiber is a  holomorphic Lagrangian torus; by \cite[Theorem 2]{teleman2018role},  it is Poisson birational to $T\times T^\vee$, which is the spectrum of the K-theory of the affine Grassmannian. (See also \cite[\S 3]{teleman2018role} for a more detailed discussion of the integrable system structure.)
The mirror, obtained by ``dualizing the torus fibration'' \emph{\`a la} \cite{SYZ} (a procedure which must be corrected in general for singular fibers), will be a space with the same structure. 
For each such space $\mht$, it is expected that the symplectic geometry of $\mht,$ or at least those aspects visible to the Fukaya category of $\mht,$ admits a description in terms of the algebraic category of coherent sheaves on the mirror space $\mht^\vee.$

The simplest multiplicative hypertoric variety is the affine variety
\[\TAo := \C^2 \setminus \{\sz\sw +1 =0\}.\]
As discussed in \cite[Section 5.1]{Auroux07}, this space admits a fibration by 2-tori with a single nodal torus fiber, and it is self-dual in the sense of the SYZ mirror construction of \cite{AAK}.
The space $\TAo$ is holomorphic symplectic, and the $S^1$ action $e^{i\theta}\cdot (\sz,\sw) = (e^{i\theta}\sz,e^{-i\theta}\sw)$ is quasi-hyperhamiltonian ({\em i.e.,} quasi-Hamiltonian with respect to the holomorphic symplectic form and Hamiltonian with respect to the real K\"ahler form -- see Section \ref{subsec:notation} for our conventions on Hamiltonian and quasi-Hamiltonian torus actions), so that it has a complex group-valued moment map and a real moment map, jointly valued in $\C^\times\times \R$. As described in Section \ref{sec:mht}, this can be lifted to a hyperhamiltonian ({\em i.e.}, Hamiltonian for all three K\"ahler forms) action, with moment map valued in $\R^3 = \C\times \R$, on the universal cover of $\TAo$.

All multiplicative hypertoric varieties are given as complex quasi-Hamiltonian reductions of $\TAno:=\left( \TAo \right)^n$ by a subtorus $T_{\R}$ of $(S^1)^n.$ To ensure that this reduction is smooth for generic choices of parameters, we require that this subtorus is {\bf unimodular}: if $J\subset \{1,\dots, n\}$ is a subset such that the projection map $T_{\R}\to (S^1)^J$ is an isogeny (a finite cover), it is an isomorphism.  
Hence, a multiplicative hypertoric variety $\mht$ is determined by two pieces of
data: \begin{enumerate}
    \item a unimodular subtorus $T\subset (\C^\times)^n$ complexifiying $T_{\R}$; and 
    \item a hyperhamiltonian
reduction parameter, which we can split into a complex moment map parameter $\parone\in T^\vee$ and a GIT stability parameter $
\GIT\in i\mathfrak{t}_\R^\vee.$
\end{enumerate}  
We will introduce an additional parameter $\Bfield$ valued in the compact torus $T^\vee_\R,$ which physicists might call a (periodicized) B-field. 
As we explain in Appendix \ref{sec:kirwan}, the
parameter $\Bfield$ determines a class in $H^2(\mht;\R/\Z)$ by the Kirwan map $H^2(BT; \R/\Z)\to H^2(\mht;\R/\Z)$, which we think of as the reduction mod $\Z$ of the imaginary part of a complexified K\"ahler form on $\mht$. 

We can combine $\GIT$ and $\Bfield$ into an element $\partwo = \Bfield \cdot \exp(\GIT)$ of $T^{\vee}$. As we'll see below, we expect the parameters $\parone$ and $\partwo$ to be exchanged under mirror symmetry.

Consider a generic pair  $(\parone,\partwo)\in (T^\vee)^2$ with the unique factorization  $\partwo=\Bfield\cdot \exp(\GIT)$ for $\Bfield\in T^\vee_\R$ 
and $\GIT\in i\ft_\R^\vee$. (Warning: in our convention, $i\ft_\R$ is the Lie algebra of the split real torus of $T,$ since we have privileged the compact real form with the notation $\ft_\R.$)
Then we will write $\mht_{(\parone,\partwo)}$ for the multiplicative hypertoric variety associated as above to the pair of parameters $(\parone,\partwo)$.
The variety $\mht_{(\tone,\tone)}$ is singular; the varieties $\mht_{(\parone,\tone)}$ and $\mht_{(\tone,\partwo)}$ are a smoothing and a resolution, respectively.

 If $(\parone,\GIT)=(\tone,0)$ but $\Bfield$ is generic, the underlying variety $\mht_{(\parone,\partwo)}$ is singular but we will define in Section \ref{sec:Bside} a noncommutative resolution (in the sense of \cite{vdB04}) depending on the choice of B-field $\Bfield$.

The diffeomorphism type of the underlying manifold $\mht_{(\parone,\partwo)}$ does not depend on $(\parone,\partwo)$, so long as they are suitably generic, but its isomorphism type as an algebraic variety depends on $\parone$ (and not $\partwo$), while its complexified symplectic form depends only on $\partwo$ (with the real part of the symplectic form depending only on $\GIT$).  

As mentioned above (and as we will recall in Appendix \ref{sec:syz-appendix}), the space $\TAo$ is shown in \cite{Auroux07,AAK} to be equal to its SYZ dual space, defined as a moduli space of Lagrangian torus objects in the Fukaya category $\Fuk(\TAo),$ and the same is true for the $n$-fold product $\TAno$. Moreover, this space has two structures which are related by mirror symmetry: namely, a complex group-valued moment map $f:\TAno\to (\C^\times)^n,$ and a Hamiltonian action of $(S^1)^n.$ Standard conjectures in mirror symmetry (see, for instance, \cite{Teleman14,TO1}) predict that mirror symmetry exchanges the operations of imposing an equation $\{f=a\}$ and taking a real Hamiltonian reduction. It is therefore natural to expect that the space $\mht_{(\parone,\partwo)},$ which is obtained by performing both operations, should be mirror to another space $\mht_{(\partwo',\parone')},$ where the pairs $\parone,\parone'$ and $\partwo,\partwo'$ are related by an (a priori complicated and unknown) mirror map. 

We conjecture that, unlike in the traditional setting of mirror symmetry for real symplectic manifolds, where the mirror map is often transcendental, in our setting the mirror map is actually the {\em identity}:
\begin{conjecture}\label{conj:syz-intro}
    The manifolds $\mht_{(\parone,\partwo)}$ and $\mht_{(\partwo,\parone)}$ are an SYZ mirror pair in the sense of \cite{AAK}.
\end{conjecture}
We refer to Appendix \ref{sec:syz-appendix} for further details on SYZ mirror symmetry. In addition, we speculate in Section \ref{subsec:relation} about the relation of Conjecture \ref{conj:syz-intro} to the SYZ mirror symmetry result proven in \cite{LZhypertoric}, which relates a multiplicative hypertoric variety equipped with superpotential to an additive hypertoric variety.

The homological mirror symmetry statement we prove in this paper will be in the setting
where the B-side variety $\mht_{(\parone,\partwo)}$ has
$\parone=\tone$, and the A-side has $\partwo=\tone.$ (This implies, in particular, that the A-side variety is affine with exact symplectic form.) Moreover, we will focus on the case where the nontrivial parameter lives inside the compact real torus $T^\vee_\R,$ which will allow us to take advantage of the holomorphic symplectic geometry of these varieties.
%
\begin{remark}
In the setting $\parone=1$ in which we prove homological mirror symmetry in this paper, it is possible to derive a more direct SYZ-type relation between these two varieties, representing the B-side variety as a moduli of torus-like objects in the Fukaya category of the A-side variety, and we discuss this relation in Section \ref{subsec:SYZ2}. However, the statement is complicated by the presence of a B-field term, which makes our B-side space not a variety but a noncommutative resolution.
\end{remark}


\subsection{Fukaya categories from Lagrangian skeleta}
As we shall see, $\mhta$ carries a natural Liouville structure. Computation of the Fukaya category of Liouville manifolds has recently become much more tractable, thanks to work of Ganatra, Pardon, and Shende on Fukaya categories of Liouville sectors \cite{GPS1,GPS2,GPS3}, following a conjecture by Kontsevich \cite{SGoHA} and subsequent work by many other mathematicians on constructible-sheaf approaches to mirror symmetry. 

In this approach, the key object is the {\bf skeleton} $\LL$, defined as the locus of points which do not flow off to infinity under the Liouville flow. Following \cite{GPS2}, we say that a Liouville manifold is {\em weakly Weinstein} if $\LL$ is mostly isotropic and admits homological cocores.\footnote{As the name suggests, any Weinstein manifold is weakly Weinstein. Remark \ref{rem:weak} touches on the role of Weinstein structures in our context.} It will ensure good behavior of the Fukaya category.

Locally, near a point on $\LL$ where it looks like a conic Lagrangian in a cotangent bundle, one can compute a category of microlocal sheaves along $\LL,$ as defined in \cite{KS94}. 
When the manifold $X$ admits a stable normal polarization --- a section over the Lagrangian Grassmannian of its stabilized symplectic normal bundle --- these are the stalks of a cosheaf of categories on $\LL.$

\begin{theorem}[\cite{Sh-hprinciple}]
There is a cosheaf of dg categories $\mu\Sh_\LL^c$ on the skeleton $\LL$ of a stably polarized weakly Weinstein manifold $X$ whose costalks may be computed locally in terms of microlocal sheaves on $\LL$.
\end{theorem}

Moreover, the results of \cite{NS20}, combined with those of \cite{GPS2,GPS3}, are sufficient to produce a comparison between the above cosheaf of categories and the Fukaya category.

\begin{theorem}[{\cite[Theorem 1.4]{GPS3}}]\label{thm:gps-main-intro}
For $X$ a stably polarized weakly Weinstein manifold with Lagrangian skeleton $\LL,$ there is an equivalence
$$
\Fuk(X) \cong \mu\Sh^c_\LL(\LL)
$$
between the wrapped Fukaya category of $X$ and the global sections of the cosheaf $\mu\Sh^c_\LL.$
\end{theorem}

To apply these results in our case, we first define a Liouville structure on $\mht_{(\parone,\tone)}$ and calculate its corresponding skeleton:
\begin{proposition}[Proposition~\ref{prop:mht-skel}, Proposition~\ref{prop:toricskel} \& Lemma \ref{lem:weaklyweinstein}]\label{prop:intro-skel}
	Let $\parone\in T^\vee$ be contained in the compact real torus $T_\R^\vee.$ Then there exists a weakly Weinstein Liouville structure on $\mht_{(\parone,\tone)}$ with associated Lagrangian skeleton $\LL_\parone$ a union of toric varieties $\mathfrak{X}_i$, whose moment polytopes are chambers of a certain toric hyperplane arrangement $\Btor_\parone$.
\end{proposition}


It may seem strange to impose the requirement $\parone\in T^\vee_\R$. One way to frame this choice is via the language of stability conditions:

\begin{conjecture}\label{conj:stabconds}
    The parameter $\parone$ determines a stability condition (in particular a t-structure) on the Fukaya category $\Fuk(\mhta),$ and for $\parone\in T^\vee_\R,$ this t-structure is a microlocal perverse t-structure described below.
\end{conjecture}
In Section \ref{sec:mon}, we give some evidence for Conjecture \ref{conj:stabconds} by describing wall-crossing equivalences relating the microlocal perverse t-structures associated to generic parameters $\parone\in T^\vee_\R$ when the parameter crosses a real wall in the compact torus $T^\vee_\R.$

A union of toric varieties as described in Proposition \ref{prop:intro-skel} actually exists as a holomorphic subvariety of $\mhta$ in a certain complex structure (defined only on an open subset of $\mhta$), called the {\bf Dolbeault complex structure}. We may treat $\LL_\parone$ not just as an ordinary Lagrangian but as a holomorphic Lagrangian variety. (Recall that a holomorphic Lagrangian (or ``(B,A,A)'') brane in a hyperk\"ahler manifold is a subspace which is Lagrangian with respect to K\"ahler forms $\omega_J$ and $\omega_K,$ and holomorphic with respect to the complex structure $I$.)

Having a holomorphic Lagrangian skeleton rigidifies many local calculations and helps to simplify the determination of the Fukaya category. As we explain in Section~\ref{subsec:microlocal} and see in detail in Section~\ref{subsec:perv-calc}, the costalks of $\mu\Sh_{\LL_\parone}^c$ in this case can be coherently equipped with a t-structure, which looks like the perverse t-structure on the category of constructible sheaves. In other words, the global Fukaya category of $\mhta$ can be computed by gluing together categories of perverse sheaves (constructible with respect to toric stratifications) on toric varieties.

These local categories of perverse sheaves can be described by purely linear-algebraic data, 
and the global category obtained by gluing these together can be described by combining this data into 
a quiver $\cQ_\parone$ together with relations in its path algebra.  
This quiver has one vertex for each chamber of $\Btor_\parone$ (with self-loops for the monodromies of the open torus of $\mathfrak{X}_i$), a pair of opposite edges for each facet shared by a pair of chambers, and the relations that the composition of one of these edge pairs is equal to monodromy around the corresponding toric divisor.
\begin{itheorem}\label{ith:A}
For $\parone\in T_\R^\vee$ generic, there is an equivalence of dg categories
$$
\mu\Sh_{\LL_\parone}^c(\LL_\parone)\cong \cQ_\parone\dgperf
$$
between a category of microlocal sheaves on the Lagrangian skeleton $\LL_\parone$ and the dg category of modules for the quiver with relations $\cQ_\parone$ defined above (or described in more detail in Definitions \ref{def:perv-tor} and \ref{def:abstract-quiver}).
Combining this with Theorem~\ref{thm:gps-main-intro}, we get an equivalence 
$$
\Fuk(\mht_{(\parone,\tone)})\cong \cQ_\parone\dgperf
$$
with the wrapped Fukaya category of $\mhta.$
\end{itheorem}

\begin{remark}
The space $\mhta$ is a Stein manifold and thus admits a natural family of Weinstein structures \cite{Eli-Ciel}. However, we will endow $\mhta$ with a distinguished Liouville vector field which does not---to our knowledge---arise in this manner. Its construction will rely on the hyperk\"ahler geometry of $\mhta$. We expect that the resulting Fukaya categories are equivalent, but we do attempt to prove this.
\end{remark}

\mtodo{I shuffled the Weinstein vs weakly-Weinstein discussion some more, and added the above remark.}

\begin{remark}
The assumption that $T_\R\subset (S^1)^n$ is unimodular implies that the components of $\LL_\parone$ are toric varieties rather than toric orbifolds/DM stacks, since no theory of microlocal sheaves on orbifolds currently exists in the literature. However, it should not be difficult to define such a theory by descent (using functoriality of $\mu\Sh$ under contactomorphisms). As we discuss at the end of Section \ref{subsec:perv-calc}, given such a theory, the first equivalence of Theorem \ref{ith:A} holds even without the unimodularity assumption.
It is expected that the category computed in this way still models the Fukaya category of $\mhta,$ thought of as an orbifold or smooth DM stack, but the foundational work necessary to define equivariant Fukaya categories does not yet exist, so there is not yet any version of Theorem \ref{thm:gps-main-intro} in this case.
\end{remark}

\subsection{Homological mirror symmetry}\label{sec:ncrintro}
On the B-side, we need to compute the category of coherent sheaves $\Cohdg(\mht_{(1,\parone)}).$
If we chose $\parone\in T^\vee$ to be a generic parameter in the split real form of $T^\vee,$ then the variety $\mht_{1, \parone}$ would be a resolution of the singular space $\mht_{(\tone,\tone)}$. However, to match the prescriptions of SYZ mirror symmetry, 
we would like to choose the parameter $\parone$ inside the compact real torus $T^\vee_\R$ instead.
This ``purely imaginary GIT parameter'' does not specify a resolution $\mht_{(1,\parone)}$ of $\mht_{(1,1)}$ in the traditional sense. Rather, for each generic choice of $\parone\in T^\vee_\R,$
we will define in Section \ref{sec:construction-tilting} a sheaf of algebras on $\mht_{(\tone,\tone)}$ which can be understood as a noncommutative crepant resolution in the sense of \cite{vdB04}. By definition, we will understand the category $\Coh(\mht_{(1,\parone)})$ as the category of perfect modules for this noncommutative resolution.

This notation is justified by the fact that the noncommutative resolution is the endomorphism algebra of a tilting bundle on any smooth resolution $\mht_{(1,\parone')}$ for a generic $\parone'$, which gives an equivalence to the derived category of coherent sheaves on this resolution.  However, this equivalence is not canonical, so we prefer to think about the derived category of modules over the noncommutative resolution.  

We construct this resolution by comparison with a similar noncommutative resolution for the additive hypertoric variety in \cite{McBW}, which was independently constructed by Spenko and van den Bergh \cite{Spenko-van-den-Bergh}.  Related results for multiplicative hypertoric varieties were obtained by Cooney and Ganev \cite{Ganev18,Cooney16}, but we will not use these directly.  We can view this noncommutative resolution as arising from the approach of Kaledin and others \cite{KalDEQ}, using quantization in characteristic p, where $\parone$ now plays the r\^ole of a noncommutative moment map parameter for a quantum Hamiltonian reduction. The resulting noncommutative resolution is a special case of constructions of line operators in gauge theory, described mathematically through the extended BFN category of \cite{WebSD}. 

The categories $\Coh(\mht_{(1,\parone)})$ so constructed are all equivalent to each other (for generic $\parone$), in accord with the expectation that the topological B-model is not sensitive to the (complexified) K\"ahler parameter $\parone.$ In analogy with (in fact, mirror to) Conjecture \ref{conj:stabconds}, this behavior can be best explained in the language of stability conditions:

\begin{conjecture}
    The parameter $\parone\in T^\vee$ determines a stability condition (in particular a t-structure) on the dg category of coherent sheaves on any symplectic resolution of $\mht_{(1,1)},$ and for $\parone\in T^\vee_\R,$ this t-structure is the ``noncommutative crepant resolution'' t-structure we construct.
\end{conjecture}


The category $\Coh(\mht_{(1,\parone)})$ is explicitly described as a category of modules over a certain quiver with relations. This quiver with relations can be expressed in terms of the combinatorics of the hyperplane arrangement $\Btor$ in exactly the same way as the quiver describing the Fukaya category of $\mht_{(\parone,\tone)}$. By comparing these two descriptions, we reach our main theorem (Theorem~\ref{thm:main-thm}):
\begin{itheorem}\label{ith:B}
There is an equivalence of dg categories
\[
   \Fuk(\mht_{(\parone,\tone)})\cong \Cohdg(\mht_{(\tone,\parone)})
\]
between the Fukaya category of the affine MHT $\mht_{(\parone,\tone)}$ and the category of coherent sheaves on a crepant resolution of $\mht_{(1,1)}.$ 
For $\parone\in T^\vee_\R\subset T^\vee,$ this equivalence is induced from an equivalence of abelian categories, the respective hearts of the perverse t-structure and the noncommutative resolution t-structure determined by $\parone$.
\end{itheorem}
%
\subsection{Perverse schobers and monodromy}
The mirror symmetry equivalences proved in Theorem~\ref{thm:main-thm} live above the components of $T^\vee_\R$ where $\parone$ is generic. These components are separated from each other by walls of (real) codimension 1. However, inside the full moduli space $T^\vee$ of $\beta,$ these walls are of complex codimension 1, so it is natural to want to follow these mirror symmetry equivalences around the walls, to produce derived equivalences between the categories located at various choices of generic $\beta\in T^\vee_\R$ or in other words a local system of categories over the complement $T^\vee_{\operatorname{gen}}$ of the discriminant locus. 

In this paper, we prefer to study only those categories living above $T^\vee_\R,$ since this is where the Lagrangian skeleton and hence also the t-structure are most tractable. Luckily, there is a way to compute the aforementioned local system of categories using only categories living over $T^\vee_\R,$ if we understand the relation of the categories computed above to the ``singular Fukaya categories'' living above parameters $\parone$ in the discriminant locus of $T^\vee.$ This method uses the technology of perverse schobers \cite{KShyperplane,KSschobers,BKS} developed by Kapranov and Schechtman. Previous uses of perverse schobers to study families of varieties include \cite{HL-S,BKS,Don-Ku}.

Expectations for the B-side schober controlling families of resolved Coulomb branches (including additive hypertoric varieties) were sketched in \cite{WebcohI}; in Section~\ref{sec:mon}, we compute this schober explicitly in the case of multiplicative hypertoric varieties. We then define an A-side perverse schober from Fukaya-categorical data associated to the family $\mht_{(\parone,\tone)}$, and finally we prove that these two schobers agree (Theorem~\ref{thm:schobers}):
\begin{itheorem}
	The A-side perverse schober defined in Section~\ref{subsec:a-schober} is equivalent to the B-side perverse schober defined in \cite{WebcohI}, and both of them can be described explicitly in terms of representations of a certain quiver.
\end{itheorem}

\subsection{Relation with part I}\label{subsec:relation}
The part I of this series \cite{McBW} concerned mirror symmetry for additive hypertoric varieties. The SYZ geometry of this situation was described in \cite{LZhypertoric}:
\begin{theorem}[{\cite{LZhypertoric}}]\label{thm:syz-additive}
Let $\aht$ be an additive hypertoric variety. Then $\aht$ is SYZ mirror\footnote{The notion of SYZ mirror used here is more precise than that in \cite{AAK} but in particular satisfies the definitions of SYZ mirror from \cite{Auroux07,AAK}.}
to an LG model $(\mht^\vee,W:\mht\to \C)$ whose underlying space $\mht^\vee$ is a multiplicative hypertoric variety. In particular, the complement in $\aht$ of a divisor $\divisor$ is SYZ mirror to the multiplicative hypertoric variety $\mht^\vee.$
\end{theorem}

In this paper, we explain that the multiplicative hypertoric variety $\mht^\vee$ is mirror to another multiplicative hypertoric variety $\mht.$ The combination of Theorem~\ref{thm:syz-additive} and Conjecture \ref{conj:syz-intro} suggests the following conjecture:
\begin{conjecture} \label{conj:additivetomult}
    Let $\mht$ be a multiplicative hypertoric variety. Then there exists an additive hypertoric variety $\aht,$ a divisor $\divisor\subset \aht$, and an isomorphism $\mht \cong \aht\setminus \divisor.$
\end{conjecture}
More precisely, the pair $(\aht,\divisor)$ should be as in Theorem \ref{thm:syz-additive}; we refer to Section \ref{sec:syz-appendix} for more detailed discussion.

\begin{remark}
Unlike the remainder of the results described in this paper, which are expected to hold in some form for multiplicative Coulomb branches more generally, the embedding $\mht\to \aht$ given by the above conjecture is special to hypertoric varieties and is not expected to exist in general. The best relation we can hope for between general additive and multiplicative varieties is the formal comparison isomorphism described in Theorem~\ref{thm:complete-iso}. Using the description of additive/multiplicative Coulomb branches as Spec of the homology/K-theory of a space, this map can be understood as a manifestation of the Chern character map from K-theory to homology. (See the proof of \cite[Proposition 3.23]{BFNplus} for another application of this philosophy.)
\end{remark}

The above conjecture suggests a concise explanation of the relationship between the results of this paper and those of \cite{McBW}.

%
Fix $\parone\in T^\vee$, which as usual we factor as
\[
\parone=\Bfield'\cdot \exp(\GIT').
\]
As explained earlier, from the choice of $\beta$ we can define $\mht=\mht_{(1,\parone)}$ as a (possibly noncommutative, if $\GIT'\neq 0$) resolution of $\mht_{(1,1)}.$
Similarly, we can consider the additive hypertoric variety $\aht=\aht_{(0,\parone)}$ defined as the hyperk\"ahler reduction of $T^*\C^n$ with complex moment map parameter 0, real moment map parameter $\GIT'$ and B-field $\Bfield'$.
(Once again, this means that if $\GIT'$ is zero but $\Bfield'$ is nonzero, we think of $\aht_{(0,\parone)}$ as a noncommutative resolution of $\aht_{(0,0)}$.)
Write $\Cohdg(\aht)_0$ for the category of
coherent sheaves on $\aht$ which are set-theoretically supported on the zero fiber of the moment map to $\mathfrak{d}^\vee$ (or equivalently, scheme-theoretically supported on the completion of the zero fiber inside $\aht)$. This is one of the categories of coherent sheaves studied in \cite{McBW}.

Since the zero moment map fiber does not intersect the divisor $\divisor$ described in \cite{LZhypertoric},
this category is equivalent to $\Cohdg(\aht \setminus \divisor)_0$ and hence equivalent to a category
$\Cohdg(\mht)_1$ of coherent sheaves on a completion of $\mht$ on the fiber over $1\in D^\vee$.

Although we do not prove Conjecture \ref{conj:additivetomult}, we are able to directly construct the desired equivalence on formal completions; this is Theorem \ref{thm:complete-iso} below. Thus, for a general value of the parameter $\parone\in T^\vee,$ we obtain an isomorphism
\begin{equation}
    \label{eq:coh-match}
    \Cohdg({\mht}_{(\tone,\parone)})_1\cong \Cohdg({\aht}_{(0,\parone)})_0,  
\end{equation}
where the parameter $\parone = \Bfield'\cdot\exp(\GIT')$ specifies multiplicative and additive hypertoric varieties as described above.
We highlight two special cases of \eqref{eq:coh-match}. When $\Bfield'=\tone$ but $\GIT'$ is generic, the varieties $\aht$ and $\mht$ are ``purely commutative'' resolutions of the singular affine additive/multiplicative hypertoric varieties, and \eqref{eq:coh-match} is an equivalence of categories of coherent sheaves. However, when $\GIT'=0,$ then \eqref{eq:coh-match} is an equivalence of module categories for the \emph{noncommutative resolutions} $\aht$ and $\mht.$
In fact, this theorem is a key technical tool and motivating observation for the noncommutative resolution we construct.  

Under the homological mirror symmetry
equivalence of Theorem~\ref{thm:main-thm}, we see that the passage from $\Cohdg(\mht)$ to $\Cohdg(\mht)_1$ is mirrored by passing from the category $\Fuk(\mht^\vee)$ to the category $\Fuk(\mht^\vee)^{\operatorname{uni}}$ where we require the monodromies on each component of $\LL$ to be unipotent.

A different incarnation of this A-brane category was produced in \cite{McBW} by studying a category
$\DQ$ of DQ-modules supported on the union of toric
varieties $\LL$ with unipotent monodromies on the components of $\LL.$
In order to specialize the category $\DQ$ at a particular quantization parameter $h=1,$
the authors of \cite{McBW} use a grading on $\DQ,$ mirror to the conic $\C^\times$-action on $\aht,$ coming from the
Hodge theory of DQ-modules.
(It would be interesting to understand better how
this sort of $\C^{\times}$-action interacts with mirror symmetry in general.)

By \cite[Prop. 4.35]{McBW}, the category $\DQ$ is equivalent to the derived category of  representations which are continuous in the discrete topology of a topological algebra $\widehat{H}_{\C}^{\parone}$ (see Definition \ref{def:hats}).  On the other hand, by Theorem \ref{ith:A}, the Fukaya category $\Fuk(\mht^\vee)^{\operatorname{uni}}$ is the equivalent category for an algebra $\widehat{A}_{\parone}$ obtained by completion of the algebra with relations appearing in this theorem.  In Proposition \ref{prop:completed isomorphism}, we explicitly construct an isomorphism $\widehat{H}_{\C}^{\parone}\cong \widehat{A}_{\parone}$, which shows that we have an equivalence of categories $\Fuk(\mht^\vee)^{\operatorname{uni}}\cong \DQ$.  On simples, this functor sends the quantization of the structure sheaf of a Lagrangian to the constant local system on the same Lagrangian. Thus, the resulting equivalence behaves like a solutions functor, sending a DQ-module to the constructible sheaf on each component obtained by considering it as a D-module on the component, and applying the Riemann-Hilbert correspondence. At the moment, we have not constructed this functor in a local geometric way, though this would be an interesting research direction.



To produce a full statement of homological mirror symmetry for the
additive hypertoric variety $\aht,$ we would have to take into account
the superpotential $W$. In the language of Liouville sectors and
skeleta, this would entail adding to the Lagrangian skeleton $\LL$
discussed above a noncompact piece asymptotic to the fiber of
$W$. 

\subsection{Notation and conventions}\label{subsec:notation}
The basic combinatorial datum in this paper is a short exact sequence $1 \to T \to D \to G \to 1$ of complex tori. We write $G_\R$ (and similarly for $D$ and $T$) for the compact real torus inside $G$; $G^\vee$ for the dual torus, containing compact torus $G^\vee_\R$; and we write $\mathfrak{g},\mathfrak{g}_\R,\mathfrak{g}^\vee,\mathfrak{g}_\R^\vee$ for their respective Lie algebras; hence we also write $i  \mathfrak{g}_\R,i \mathfrak{g}_\R^\vee$
for the Lie algebras of the split real tori in $G,G^\vee.$
(Note that since we write $G_\R$ for the compact real form of $G$ instead of the split, our conventions on which part of the torus is ``real'' and which ``imaginary'' are reversed from \cite{Teleman14}.)

The pair of parameters $(\parone,\partwo)$ on which the definition of multiplicative hypertoric varieties depends lives in $T^\vee\times T^\vee.$ We will sometimes find it useful to factor
\[
\partwo=\Bfield\cdot\exp(\GIT),
\]
where $\Bfield\in T^\vee_\R$ and $\GIT\in i\ft^\vee_\R.$ For instance, if $T^\vee$ is identified with $\C^\times,$ then $\Bfield = e^{2\pi i\Arg(\partwo)},$ and $\GIT=\log|\partwo|.$

We will also write
\[
\mathfrak{g}_\Hbb:= \mathfrak{g}_\R\otimes_\R\R^3 = \mathfrak{g}\oplus \mathfrak{g}_\R
\]
for the ``quaternionified Lie algebra'' of $G,$ and $\mathfrak{g}_\mathbb{Z}^\vee := X^\bullet(G)$ for the character lattice of $G$.

In this paper, we use both dg categories and abelian categories; we will always be explicit about which are which. All limits, colimits, sheaves, \emph{etc.} of categories are always taken in the appropriate homotopical sense. Here we make clear our notation for various flavors of categories we consider:
\begin{itemize}
    \item For $X$ an algebraic variety, we write $\Coh(X)$ for the abelian category of coherent sheaves on $X,$ and $\Cohdg(X)$ for the dg-category of complexes in this category.
    \item For $A$ a (possibly dg) algebra, we write $A\dgmod$ for the dg category of complexes of $A$-modules. In the case where $A$ is an underived (\emph{i.e.}, concentrated in degree 0) algebra, we write $A\mmod$ for the abelian category of $A$-modules. Similarly, we write $A\perf$ and $A\dgperf$ for the respective abelian and dg categories of perfect $A$-modules.
    \item For $M$ a Weinstein manifold or sector, we write $\Fuk(M)$ for (an idempotent-complete, dg model of) the Fukaya category of $M$ as defined in \cite{GPS1}.
    \item For $X$ a manifold, we write $\Sh(X)$ for the dg category of possibly infinite-rank complexes of sheaves on $X$ with constructible cohomology. If $X$ is a complex manifold, we write $\Perv(X)$ for the abelian category of possibly infinite-rank perverse sheaves on $X$.
    \item In general, for $\cC$ a dg category equipped with a t-structure, we write $\cC^\heartsuit$ for the abelian category which is its heart.
\end{itemize}

We also study various classes of closely-related spaces, whose notation we list here:
\begin{itemize}
    \item The letter $\aht$ always denotes a toric hyperk\"ahler variety, which we refer to in this paper as an ``additive hypertoric variety,'' to distinguish it from its multiplicative analogue.
    \item Multiplicative hypertoric varieties will be denoted by $\mht.$ We will sometimes call this a ``Betti MHT'' to distinguish it from its Dolbeault version. We write $\widetilde{\mht}$ for the $\mathfrak{g}^\vee_\Z$ cover of $\mht$ determined by pulling back the complex group-valued moment map $\mht\to G^\vee$ along the exponential map $\fg^\vee\to G^\vee.$
    \item The ``Dolbeault version'' of a multiplicative hypertoric variety, which is a hyperk\"ahler rotation of (an open retract of) $\mht$, is denoted by $\Dolb$.
\end{itemize}

Finally, we discuss some nonstandard or possibly unfamiliar usages in symplectic geometry.
\begin{itemize}
    \item Let $X$ be a symplectic manifold with an action of a compact torus $T_\R.$ We say that this action is {\bf quasihamiltonian} if there is a map $X\to T_\R^\vee$ such that the pullback $X\times_{T_\R^\vee}\mathfrak{t}_\R^\vee\to\mathfrak{t}_\R^\vee$ along the exponential map $\mathfrak{t}_\R^\vee\to T_\R^\vee$ is a moment map for the $T_\R$-action on the pullback. 
    The map $X\to T_\R^\vee$ is the {\bf group-valued moment map} associated to this quasihamiltonian action.
    (We do not consider quasi-Hamiltonian actions of non-abelian groups, whose definition is more involved and for which we refer readers to \cite{AMM98}.)
    \item Let $X$ be a space equipped with a K\"ahler form $\omega_J$ and a holomorphic symplectic form $\Omega_J = \omega_K+i\omega_I.$ We say that an action of a torus is {\bf hyperhamiltonian} if it is Hamiltonian for all three forms $\omega_J,\omega_K,\omega_I$ and {\bf quasihyperhamiltonian} if it is Hamiltonian for two of them and quasihamiltonian for the third.
    \item The space $\mht$ we study in this paper is equipped with a quasihyperhamiltonian action of a torus $G_\R,$ for which we denote the $\omega_J$ moment map by $\mmm_\R:\mht\to \fg^\vee_\R,$ and we combine the other two moment maps into a map $\mmm_{\C^\times}:\mht\to G^\vee.$
    \item In the above situation, we will sometimes want instead to combine the two additive moment maps to a map to the complex Lie algebra $\fg^\vee.$ To emphasize that this map is compatible with an unusual complex structure on $\mht,$ we denote it $\mmm_{I,\C}:\mht\to \fg^\vee.$
    \item Let $T_\R$ be a torus with a Hamiltonian action on a K\"ahler manifold $X.$ Traditionally, a symplectic reduction $X/\!\!/_\GIT T_\R$ is specified by a parameter $\GIT\in \ft^\vee_\R,$ which determines the symplectic form on the reduction. However, to match up mirror complex-structure and K\"ahler moduli, we would actually like to specify a {\em complexified} K\"ahler form, understood as a K\"ahler form plus an ``imaginary-valued'' 2-form, on such a Hamiltonian reduction, which requires a choice of parameter $\partwo\in T^\vee.$ We will often write $\partwo=\Bfield\cdot \exp(\GIT),$ so that it consists of the data of the traditional moment-map (or GIT) parameter $\GIT$ together with a $T^\vee_\R$-valued parameter $\Bfield.$ This latter parameter determines a class in $H^2(X;\R/\Z)$ by the Kirwan map, which we call the {\bf B-field}. It is the reduction modulo $\Z$ of the imaginary part of a complexified K\"ahler form on the Hamiltonian reduction $X/\!\!/_\GIT T_\R.$ (We choose to remember this class only modulo $\Z$ because the B-field contributes to Floer-theory calculations via its exponential, which is $\Z$-periodic.) We write $X/\!\!/_\partwo T_\R$ for the quotient $X/\!\!/_\GIT T_\R$ equipped with this complexified K\"ahler form.
    For more details of this construction, see Appendix \ref{sec:kirwan}.
\end{itemize}

\subsection{Acknowledgements}
The authors are grateful to Davide Gaiotto and Andy Neitzke for explanations about the Ooguri-Vafa metric, and to Daniel Pomerleano for discussions about SYZ mirror symmetry. B.G. would also like to thank Ben Wormleighton for many helpful discussions about preprojective algebras and Denis Auroux for advice about K\"ahler geometry. M.M. would like to thank the hospitality of Kavli IPMU.

B.G. was supported by an NSF Graduate Research Fellowship and by the Berkeley RTG on Number Theory and Arithmetic Geometry (DMS-1646385).
B.W. was supported by NSERC under a Discovery Grant.

M.M. was supported by an RGC Early Career Scheme grant (project number : 24307121).

All the authors wish to thank Perimeter Institute for Theoretical Physics, where this work was substantially completed during B.G.'s term as a visiting student.  
Research at Perimeter Institute is supported in part by the Government of Canada through the Department of Innovation, Science and Economic Development Canada and by the Province of Ontario through the Ministry of Colleges and Universities.

\section{Multiplicative hypertoric varieties}\label{sec:mht}
In this section, we describe the construction of hypertoric, or toric hyperk\"ahler, varieties and their multiplicative versions. The most basic hypertoric variety is the space $T^*\C^n$, which is holomorphic symplectic and in fact hyperk\"ahler, which is obvious from its identification with $\mathbb{H}^n.$ All other additive hypertoric varieties will be hyperhamiltonian reductions
of this space.

Hence we begin with the data of a split $k$-dimensional algebraic torus $T$ over $\C$ and a faithful linear action
of $T$ on the affine space  $\ACn$, which we may assume is diagonal in the
usual basis.  As discussed in the introduction, we assume that $T$ is unimodular: \begin{definition}
    For each subset $J\subset \{1,\dots, n\}$, we have a coordinate subtorus $(S^1)^J$, and a projection $\pi_J\colon T\to (S^1)^J$.  We call $T$ {\bf unimodular} if whenever $\pi_J$ is an isogeny, it is an isomorphism.  
\end{definition}
Furthermore, we assume that no coordinate subtorus lies in the image of $T$.
Letting $D\cong \mathbb{G}_m^n$ be the group of diagonal  matrices in this basis, we have a short exact sequence of tori
\begin{equation} \label{eq:basictoriseq}
1\longrightarrow T\longrightarrow D\longrightarrow G\longrightarrow 1.
\end{equation}
%



The action of $D$ on $\AC^n$ induces a hyperhamiltonian action of $D_\R$ on $T^*\AC^n,$ on which an element $d\in D$ acts on $(\vec{x},\vec{y})\in \AC^n$ by $(\vec{x},\vec{y})\mapsto (d\cdot \vec{x},d^{-1}\cdot \vec{y}).$ We factor the hyperk\"ahler moment map $\mu_{\Hbb}:T^*\AC^n\to \mathfrak{d}^*\otimes\R^3$ into algebraic and real parts as
\begin{equation}\label{eq:hmoment}
(\mu_\C,\mu_\R):T^*\AC^n\longrightarrow\mathfrak{d}^\vee\oplus \mathfrak{d}_\R^\vee,\qquad 
(z_i,w_i)_{i=1}^n\longmapsto\left((z_iw_i),(|z_i|^2-|w_i|^2)\right)_{i=1}^n.
\end{equation}

By restricting along the inclusion $T\hookrightarrow D,$ we also have an action of $T$ and hence a hyperk\"ahler moment map to $\ft^\vee \times \ft^\vee_\R.$ This moment map is obtained by composing the moment map \eqref{eq:hmoment} with the the maps obtained by dualizing the inclusion $\mathfrak{t}\to\mathfrak{d}.$ From now on, we will use $\mu_\Hbb = (\mu_\C,\mu_\R)$ to denote this moment map.

Fix $(\parone,\GIT)\in \mathfrak{t}^\vee \oplus \mathfrak{t}^\vee_\R = \mathfrak{t}^\vee_\Hbb.$
\begin{definition}
The {\bf hypertoric variety $\mathfrak{Y}$} associated to the data of the short exact sequence \eqref{eq:basictoriseq} and the choice of $(\parone,\GIT)$ is the hyperhamiltonian reduction $T^*\AC^n/\!\!/\!\!/_{(\parone,\GIT)}T_\R$ of $T^*\AC^n$ by $T_{\R}$ with parameters $(\parone,\GIT).$ In other words, $\mathfrak{Y}$ is the GIT quotient of $\mu_\C^{-1}(\parone)$ by $T$ at character $\GIT.$
\end{definition}

There are many natural ways (for instance, passing through multiplicative preprojective algebras \cite{CBS,Yamakawa}, K-theoretic Coulomb branches \cite{BFN,FT}, or moduli of microlocal sheaves on nodal curves \cite{BK16} ) to produce a ``multiplicative'' or ``loop-group'' version of the above construction. The construction we describe below produces a multiplicative hypertoric variety $\mht.$  We refer to the complex structure that arises naturally in our construction as the {\bf Betti complex structure}. (Later, we will introduce a {\bf Dolbeault hypertoric variety}, which contains a holomorphic Lagrangian subvariety isomorphic (as a Lagrangian) to our preferred Lagrangian skeleton of $\mht$; the expected relation of the Dolbeault variety to $\mht$ via hyperk\"ahler rotation serves as inspiration for our constructions, although none of our results logically depends on it.)

\subsection{The Betti MHT}

Our construction now proceeds exactly as before, except that we replace the space $T^*\AC^n$ with its multiplicative version
  \[\TAno:=\{(\sz_1,\dots,\sz_n,\sw_1,\dots,\sw_n) \mid \sz_i\sw_i\neq -1\}\subset T^*\ACn,\]
   equipped with a holomorphic symplectic form $\Omega_{K} := \sum_{i=1}^n \frac{ d \sz_i d \sw_i}{1 + \sz_i \sw_i},$ as well as a Liouville structure whose definition we postpone to Proposition \ref{prop:mht-skel}. 
  We will write $\omega_I$ and $\omega_J$ for the real and imaginary parts of $\Omega_{K},$ respectively.

  As in the additive case, $D$ acts on $\TAno,$ although with respect to $\Omega_{K}$ the action is now quasi-Hamiltonian in the sense of \cite{AMM98} (see also \cite{Saf16} for a modern perspective): to the $D$-action is associated a moment map $\mmm_{\C^\times}$ valued in the \emph{group} $D^\vee$.
This $D$-action is quasi-Hamiltonian for the holomorphic symplectic form $\Omega_{IJ},$ and the $D_{\R}$ action is Hamiltonian for the third symplectic form $\omega_K,$ so that we have a quasi-hyperhamiltonian moment map
 \begin{equation*}\label{eq:hmoment2}
(\mmm_{\C^\times},\mmm_\R):\TAno\longrightarrow D^\vee\times \mathfrak{d}_\R^\vee,\qquad 
(\sz_i,\sw_i)_{i=1}^n\longmapsto\left((\sz_i\sw_i+1),(|\sz_i|^2-|\sw_i|^2)\right)_{i=1}^n.
\end{equation*}
From now on we use $\mmm_{\C^\times},\mmm_R$ to denote the composition of the above maps with the pullbacks $D^\vee\to T^\vee, \mathfrak{d}^\vee_\R\to\mathfrak{t}^\vee_\R.$
  
  \begin{definition}\label{def:mht}
	  The {\bf multiplicative hypertoric variety} $\mht_{(\parone,\partwo)}$ associated to the data of the short exact sequence \eqref{eq:basictoriseq} and the respective complex-structure and moment-map parameters $(\parone,\partwo)\in T^\vee\times T^\vee,$ where we write $\partwo=\Bfield\cdot \exp(\GIT),$ 
	  is the Hamiltonian reduction \[
	  \mht_{(\parone,\partwo)}:=\mmm_{\C^\times}^{-1}(\parone)/\!\!/_{\partwo}T_\R,
	  \]as defined in \ref{def:bfield-reduction}: its underlying symplectic manifold is $\mmm_{\C^\times}^{-1}(\parone)/\!\!/_\GIT T_\R,$
	  with periodicized B-field pulled back from $\Bfield\in H^2(BT_\R;\R/\Z)$ along the Kirwan map. 
	  Where the parameters $(\parone,\partwo)$ are either already clear or else irrelevant, we will sometimes refer to this variety just as $\mht.$ 
%
  \end{definition}
  
  As in the case of additive hypertoric varieties, the variety associated to parameters $(\parone,\GIT)=(\tone,0)$ is singular and affine; the first parameter controls a smoothing, and the second a resolution, of the singular variety $\mht_{(\tone,\tone)}$. In particular, for suitably generic $\parone$, the variety $\mht_{(\parone,\tone)}$ is smooth and affine. In this paper, we will be interested in the symplectic geometry of this affine variety;
  we will see that the mirror to this space is a (possibly noncommutative) resolution of $\mht_{(\tone,\tone)}.$
  
We can be more explicit about the space of generic parameters $\parone$ for which the variety $\mht_{(\parone,\GIT)}$ is smooth:
\begin{definition}\label{def:circuits}
    The {\bf support} of an element $x \in \fd_\Z = \Z^n$ is the subset of $\{1,...,n\}$ consisting of coordinates with nonzero coefficient.

     A {\bf signed circuit} of $\fd_\Z$ is a primitive element in the image of $\ft_\Z$ with minimal support. Note that if $\circuit$ is a signed circuit, then so is $-\circuit$. A {\bf circuit} is the support of a signed circuit.
\end{definition}
\begin{definition} \label{def:paronegeneric}
     The orthogonal complement to a circuit $\circuit \in \ft_\Z$ determines a codimension one subtorus $T^{\vee}_{\circuit} \subset T^{\vee}$.  The parameter $\parone \in T^{\vee}$ is said to be {\bf generic} if it does not lie in $T^{\vee}_{\circuit}$ for any circuit.
\end{definition}
\begin{lemma}
Let $\parone\in T^\vee$ be generic. Then the variety $\mht_{(\parone,0)}$ is smooth.
\end{lemma}
\begin{proof}
First, we show that the action of $T$ on $\mmm_{\C}^{-1}(\parone)$ is free.  Assume not; in this case, the Lie algebra of the stabilizer of some point $(\sz,\sw)$ contains a signed circuit $\circuit \in \ft_\Z$.  This means that $\sz_i=\sw_i=0$ for all $i$ in the corresponding circuit and in turn implies that $\mmm_{\C}(\sz,\sw)\in T^{\vee}_{\circuit}$, contradicting the assumption of genericity.  

Thus, the action of $T$ on $\mmm_{\C}^{-1}(\parone)$ is free.  This implies that all the orbits of $T$ are closed, since if an orbit isn't closed, it will have an orbit of lower dimension in its closure, which is thus necessarily non-free.  Since all orbits are closed, $\mht_{(\parone,0)}$ is the set-theoretic quotient, and this is smooth since it has a free and proper action of $T$.  
\end{proof}

From now on, unless otherwise specified, we assume that $\parone$ is generic in this sense.
Note that if we restrict our parameter space to the compact torus $T^\vee_\R,$ the generic parameters form an open dense subset -- in fact, they are the complement of a toric hyperplane arrangement, which will be studied in more detail in Section \ref{sec:mon}.
  
We will need one further piece of structure on $\mht_{(\parone, \partwo)}$. The action of $D$ on $\TAno$ descends to an action on $\mht_{(\parone, \partwo)}$ factoring through the torus $G = D/T$. The latter has moment map 
\begin{equation} 
(\mmm_{\C^\times},\mmm_\R): \mht_{(\parone, \partwo)} \longrightarrow D^\vee\times \mathfrak{d}_\R^\vee
\end{equation}
obtained by restricting the $D$-moment map, and its image is the translate of $G^{\vee} \times \mathfrak{g}_\R^\vee$ given by the preimage of $(\parone, \delta)$ under the quotient map $D^\vee\times \mathfrak{d}_\R^\vee \to T^\vee\times \mathfrak{t}_\R^\vee.$ 

\section{Liouville structures and Dolbeault spaces} \label{subsec:dolbeault}

A {\bf Liouville structure} on a symplectic manifold $(M,\omega)$ is a primitive $\omega=d\lambda$ for the symplectic form, or equivalently a vector field $X$ on $M$ satisfying $\mathcal{L}_X\omega=\omega,$ so that the Liouville form is $\lambda=X\lrcorner\omega.$ We say that a Liouville manifold $(M,\omega,\lambda)$ is {\bf Weinstein} if it can be equipped in addition with an exhausting Morse-Bott function $f:M\to \R$ such that the Liouville vector field $X$ is gradient-like for $f.$ A standard reference for the theory of Liouville, Stein, and Weinstein manifolds is \cite{Eli-Ciel}.

\begin{definition}
    Let $(M, \omega = d\lambda)$ be a Liouville manifold with Liouville vector field $V=\omega^{-1}(\lambda),$ whose time-$t$ flow we denote by $\phi_t.$ The {\bf skeleton} of $X$ is the subset $\LL\subset M$ of points in $M$ which do not escape to infinity under this flow.
\end{definition}

Thanks to results of \cite{GPS3}, the Fukaya category of a Weinstein manifold may be calculated locally on its skeleton $\LL,$ so our task is to determine a skeleton for the multiplicative hypertoric variety $\mht_{(\parone,\tone)}.$ One natural way to define a Liouville structure is via a choice of Morse-Bott K\"ahler potential on the Stein manifold $\mht_{(\parone,\tone)}.$ However, this approach leads to some sensitive calculations in Morse-Bott theory. 

We resolve this issue by introducing Dolbeault spaces to our story. In addition to the complex structure so far, the smooth manifold underlying $\mht_{(\parone, 1)}$ admits another complex structure, which makes it into a complex analytic manifold $\Dolparone,$ the {\bf Dolbeault hypertoric space} (or {\bf Dolbeault space} for short). These spaces were first defined by Hausel and Proudfoot in an unpublished note, and have been recently studied in \cite{McBW} and \cite{DMS19}. We will briefly recall their construction below, referring to \cite[Section 4]{McBW} or \cite[Section 7]{DMS19} for more details.

The relevance of these spaces for us is that the most convenient Liouville form on the space $\mht_{(\parone, 1)}$ is naturally defined in terms of the geometry of the Dolbeault space $\Dolparone.$ This is analogous to the case of conic symplectic resolutions, where a Liouville structure on the deformation of a symplectic singularity may be defined by means of a $\C^\times$ action on the resolution (which is diffeomorphic to the deformation). As we will see, in our case we will make use of a $\C^\times$ action not on the Dolbeault space $\Dolparone$ itself, but rather on a related space, the algebraization of its universal cover.

The resulting Liouville structure will be weakly Weinstein rather than Weinstein. This is enough to apply the results of \cite{GPS3}.

\begin{warning} \label{warn:weakly}
    The above paragraphs mention two ways of producing a Liouville structure on the manifold $\mht_{(\parone,1)}$: either by using a Stein K\"ahler potential on the affine algebraic variety $\mht_{(\parone,1)}$, or by defining a Liouville structure on the diffeomorphic Dolbeault space $\Dolparone.$ We opt for the latter, and we expect that the two approaches are equivalent, but we do not prove this.
\end{warning}

We begin by introducing the basic building block $\fOV$ of the Dolbeault space, and constructing a Liouville domain $\fOV_\epsilon \subset \fOV$ in Lemma \ref{lem:basicliouvilledomain}. We then define more general Dolbeault spaces, and associate a Liouville domain $(\Dolparone)_\epsilon \subset \Dolparone$ to each in Lemma \ref{lem:generaldolbeaultLiouvilledomain}. Finally, we identify the Liouville completion of $(\Dolparone)_\epsilon$ with $\mht_{(\parone, 1)}$ in Proposition \ref{prop:mht-skel}.



Throughout this section, we take the parameter $\parone\in T^\vee$ to live inside the compact torus $T^\vee_\R.$

\subsection{Construction of the basic spaces \texorpdfstring{$\fOV, \tfOV$ and $\tfOValg$}{Z,Z~ and Z~alg}} \label{sec:basicdol-constr}
In this section we construct the analogue $\fOV$ of the basic Betti space $\TAo$. It is an elliptic fibration over the disk. We begin by defining an infinite type toric surface $\tfOValg$, before introducing $\fOV$ in Definition \ref{def:fOV} as a $\Z$ quotient of an open set of $\tfOValg$. This construction is in the spirit of Mumford's construction of degenerating abelian varieties in \cite{mumford1972analytic}; we learned of its relevance to hypertoric geometry from unpublished work of Hausel and Proudfoot. 




\begin{definition} Let $\tfOValg$ be the infinite type toric surface obtained by identifying $\C^2_n \setminus \{x_n = 0\}$ with $\C^2_{n+1} \setminus \{y_{n+1}=0\}$ via 
\[ y_{n+1} = x_{n}^{-1} \] and 
\[ x_n y_n = x_{n+1}y_{n+1}. \]
\end{definition}
We write $q : \tfOValg \to \C$ for the function which equals $x_ny_n$ on $\C^2_n$. It is a $\C^\times$ fibration over $\C$, degenerating over the origin to a union of $\Z$ copies of $\mathbb{P}^1$ glued end-to-end. 
\begin{remark} One may view $\tfOValg$ as the toric variety whose fan of cones in $\R^2$ is generated by $\{(n,1)\mid n\in \Z\}$.
\end{remark}

There is an action of $\C^{\times} \times \C^{\times}$ satisfying
\begin{equation} \label{eq:actionofbigtorus}  (z,w)^*x_n = w^{n+1} z x_n, \ \ \ (z,w)^* y_n = w^{-n} z^{-1} y_n. \end{equation}
We write $\C^{\times}_{z} = \{ (z,1) \}$ and $\C^{\times}_w = \{ (1,w) \}$ for the complementary subtori. The variety $\tfOValg$ carries an algebraic symplectic form $\Omega_I^{\operatorname{alg}}$ which restricts to $dx_n \wedge dy_n$ on $\C_n^2$. The subtorus $\C^{\times}_{z}$ preserves $\Omega_I^{\operatorname{alg}}$, with moment map $q : \tfOValg \to \C$. The subtorus $\C^{\times}_w$ scales both $\Omega_I^{\operatorname{alg}}$ and $q$ with weight one:
\begin{equation} \label{eq:algebraicdilation}
(z,w)^* \Omega_I^{\operatorname{alg}} = w\Omega_I^{\operatorname{alg}} \ \ \ (z,w)^*q = w q.
\end{equation}

We have an action of $\Z$ on $\tfOValg$, whose generator $\tau$ is characterized by $$\tau^*x_n = x_{n-1}, \tau^* y_n = y_{n-1}.$$ From \eqref{eq:actionofbigtorus}, we have the following relation between the $\Z$-action and the $\C^{\times}_z\times \C^{\times}_w$-action (expressed in terms of the pullback map on functions) :
\begin{equation} \label{eq:basictauanddilations}
\tau^* (z,w)^* = (wz, w)^* \tau^*.
\end{equation}

The $\Z$-action on $\tfOValg$ is discrete along the locus where $|q| < 1$. 
\begin{definition} \label{def:fOV}
    Let $\tfOV := \{ |q| < 1 \} \subset \tfOValg$. Let $\fOV := \tfOV / \Z$.
\end{definition}
Note that the action of $\C^{\times}_w$ does not preserve $\tfOV \subset \tfOValg$, since it scales $q$ with weight one.

The space $\fOV$ is an elliptic fibration over the disk with a nodal fiber 
\begin{equation} \label{eq:tatenodalfiber} \mathcal{L} := q^{-1}(0) \end{equation} over the origin. Fix a small $\epsilon > 0$ and let \[ \tfOV_{\epsilon} = \{ |q| \leq 1 - \epsilon \} \subset \tfOV. \] We similarly define $\fOV_{\epsilon} = \tfOV_{\epsilon} / \Z$. It is a manifold with boundary. It admits a hyperk\"ahler metric, first described by Ooguri-Vafa \cite{OV96} and studied at length in Section 3 of \cite{GrWi}. The construction proceeds via the Gibbons-Hawking ansatz. This starts with a four-manifold $X$ with a circle-action and an open embedding $X/S^1 \to \R^3$. and a harmonic function on the image with prescribed singularities at images of the $S^1$ fixed points. It returns a hyperk\"ahler metric on $X$, such that the map $X \to \R^3$ is the hyperk\"ahler moment map for the $S^1$ action. Starting from an elliptic fibration $X$ over a disk with a nodal fiber at the origin, Gross and Wilson show how to pick the harmonic function on (a subset of) $\R^3$ so as to ensure that one of the hyperk\"ahler complex structures coincides with the given complex structure on $X$. In general this may require shrinking the radius of the disk by an amount which depends on the periods of the fibration. In our setting, the radius can be taken arbitrarily close to $1$, and we choose $1- \epsilon$ to be slightly smaller than the maximum possible radius. None of our constructions will depend on the precise choice of $\epsilon$. 

For more details on the application of Gross and Wilson's results to our setting, we refer the reader to \cite[Section 7]{DMS19}. We summarise the properties of the resulting metric as follows. 

\begin{theorem} \label{thm:OVmetric}
The space $\tfOV_{\epsilon}$ admits a $\Z$-invariant (non-complete) hyperk\"ahler metric $g$. We denote the associated triple of complex structures $(I,J,K)$, and the triple of associated symplectic forms 
\[ \omega_I = g(I-,-), \omega_J = g(J-, -), \omega_K = g(K-,-). \]
We combine two of these into the holomorphic symplectic form
\[ \Omega_I := \omega_J + I \omega_K. \]
The complex structure $I$ (situated at $\infty$ on the twistor sphere) is the one obtained by restriction from $\tfOValg$. The algebraic form $\Omega_I^{\operatorname{alg}}$ restricts to $\Omega_I$.
\end{theorem}
\begin{theorem} \label{thm:momentmaps}
The action of the compact subtorus $S^1 \subset \C^{\times}_z$ on $\tfOV_{\epsilon}$ is hyperhamiltonian, with moment map $(\momi,\momj,\momk) : \tfOV_{\epsilon} \to \R \times \C$ satisfying 
\begin{equation} \label{eq:taudeck} \tau^* (\momi, \momj, \momk) = (\momi + 1, \momj, \momk) \end{equation}
where $\tau \in \Z$ is the generator. The map $q$ restricts to $\momj + I\momk : \tfOV_{\epsilon} \to \D$. 
\end{theorem}

In addition, the following result, comparing the Dolbeault space $\fOV_\epsilon$ to the Betti space $\TAo,$ can be found in \cite[Section 1.3.2]{li2019syz}:
\begin{theorem}
    In complex structure $J$, the space $\fOV_\epsilon$ admits a complex-analytic embedding into $\TAo.$
\end{theorem}


We see from \eqref{eq:taudeck} that the action of $S^1 \subset \C^{\times}_z$ on $\fOV_{\epsilon}$ is  quasi-hyperhamiltonian with moment map valued in $\C\times S^1$: in other words, the $S^1$-action is honestly Hamiltonian for the holomorphic symplectic form but quasi-Hamiltonian for the real K\"ahler form on $\fOV_{\epsilon}.$

\begin{definition} 
Let $\trihamvec$ be the hyperhamiltonian vector field $\trihamvec$ generated by the $S^1$-action on $\tfOV_{\epsilon}$.
\end{definition}
By definition,
\[ d\momi = \omega_I(\trihamvec, -), \ d\momj = \omega_J(\trihamvec, -), \ d\momk = \omega_K(\trihamvec, -). \ \] 


\subsection{Liouville vector fields in \texorpdfstring{$\tfOV_{\epsilon}$}{Z~epsilon}}
Let $\mathscr{Y}_z, \mathscr{Y}_w$ denote the vector fields on $\tfOValg$ generating the action of $S^1_z \subset \C^{\times}_z$ and $S^1_w \subset \C^{\times}_w$ respectively. Then $\mathscr{Y}_z$ restricts to the hyperhamiltonian vector field $\trihamvec$ on $\tfOV_{\epsilon}$, whereas $\mathscr{Y}_w$ restricts to a non-hyperhamiltonian vector field. 

\begin{definition}
Let $\tX := I \cdot \mathscr{Y}_w$. 
\end{definition}
This is the vector field on $\tfOValg$ generating the action of $\R^* \subset \C_w^*$. We will from now on view as a vector field on the open subset $\tfOV_{\epsilon}$. By \eqref{eq:algebraicdilation}, on we have
        \begin{equation} \label{eq:basicdilatingproperty} \Lieder_{\tX} (\omega_J + I \omega_K) = \omega_J + I \omega_K  \end{equation}
and
\begin{equation} \label{eq:dilatesthebasetoo} \Lieder_{\tX} (\momj + I \momk) = \momj + I \momk. \end{equation}

\begin{lemma} \label{lem:basicliouvilleproperty}
The vector field $\tX$ dilates the symplectic form $\omega_J$ and is outward-pointing along the boundary $q^{-1}(S^1_{1-\epsilon}).$
\end{lemma}
\begin{proof}
The first statement is \eqref{eq:basicdilatingproperty}. On the other hand, \eqref{eq:dilatesthebasetoo} shows that $\tX$ scales the function $q=\mu_J+I\mu_K$ and thus points outward along the level set $\{|q|=1-\epsilon\}.$
\end{proof}

We would like to obtain a Liouville domain by passing to $\fOV_\epsilon = \tfOV_\epsilon / \Z$.
Unfortunately, $\tX$ is not $\Z$-invariant: 
from \eqref{eq:basictauanddilations}, we have
\begin{equation} \label{eq:tXdeck} \tau^* \tX = \tX + I \cdot \trihamvec. \end{equation}
Hence $\tX$ does not descend to a Liouville vector field on $\fOV_{\epsilon}$. We will obtain a $\Z$-invariant Liouville vector field on $\tfOV_{\epsilon}$ by adding a Hamiltonian correction term to $\tX$. Given $f: \tfOV_{\epsilon} \to \R$, let $\hamvec_f$ denote the Hamiltonian vector field associated to the function $f$ and the form $\omega_J$. 
\begin{definition}
\begin{equation} \label{eq:hamcorrection} \liouvillevec := \tX + \hamvec_{\momi \momk }. \end{equation}
\end{definition}
\begin{remark} In \eqref{eq:hamcorrection}, the moment maps $\momi, \momk$ make an appearance, whereas the moment map $\momj$ does not. Had we chosen $\omega_K$ instead of $\omega_J$ as our real symplectic form, we would see the moment maps $\momi, \momj$ appear at this stage.
\end{remark}

To prove that this is $\Z$-invariant, we need to do a few preliminary calculations.
\begin{lemma} \label{lem:hamvecofmomiandmomk}
\[ \hamvec_{\momk} = -I \cdot \trihamvec \ \ {\text and } \ \ \hamvec_{\momi} = K \cdot \trihamvec.\]
\end{lemma}
\begin{proof}
\[ \iota_{I \cdot \trihamvec}(\omega_J) = g(JI \cdot \trihamvec, -) = g(- K \cdot \trihamvec, -) = - \iota_\trihamvec(\omega_K) = - d \momk. \]
This proves the first equality. The proof of the second equality is identical.
\end{proof}

\begin{lemma}
We have 
\begin{equation} \label{eq:hamvecofmomimomk} \hamvec_{\momi \momk} = - \momi I \cdot \trihamvec + \momk K \cdot \trihamvec. \end{equation}
\end{lemma}
\begin{proof}
The claim follows from Lemma \ref{lem:hamvecofmomiandmomk} and the general equality $\hamvec_{fg} = f \hamvec_{g} + g \hamvec_{f}$. 
\end{proof}

\begin{lemma}
$\liouvillevec$ is a $\Z$-invariant vector field on $\tfOV_{\epsilon}$.
\end{lemma}
\begin{proof}
We have $\tau^* \hamvec_{\momi \momk} = \hamvec_{(\momi+1)\momk} = \hamvec_{\momi \momk} + \hamvec_{\momk} = \hamvec_{\momi \momk} - I \cdot \trihamvec$. This cancels the term appearing in \eqref{eq:tXdeck}, leaving us with a $\Z$-invariant vector field. 
\end{proof}

\begin{lemma} \label{lem:Liouvilleproperty}
\begin{equation} \label{eq:Zliouville} \Lieder_\liouvillevec (\omega_J) = \omega_J. \end{equation}
\end{lemma}
\begin{proof}
Since $\liouvillevec$ differs from $\tX$ by a hamiltonian vector field, this follows from  \eqref{eq:basicdilatingproperty}.
\end{proof}

\subsection{Making \texorpdfstring{$\fOV_\epsilon$}{Z~epsilon} into a Liouville domain}
We now show that $\liouvillevec$ dilates the base $\D$ of the elliptic fibration, in the following sense. Recall that $\D$ is a disk with coordinates $\momj, \momk$. 
\begin{lemma} \label{lem:basedilationproperty}
$\liouvillevec(\momj) = \momj$ and $\liouvillevec(\momk) = \momk(1 + g(\trihamvec, \trihamvec))$.
\end{lemma}

\begin{proof}
We have $\tX( \momj) = \momj$ by \eqref{eq:basicdilatingproperty}, and  $\hamvec_{\momi \momk}(\momj) = \{ \momi \momk, \momj \} = -\trihamvec(\momi \momk) = 0$. Hence $\liouvillevec(\momj) = \momj$. 

We have $\tX( \momk) = \momk$ by \eqref{eq:basicdilatingproperty}, and $\hamvec_{\momi \momk}(\momk) = \{ \momi \momk , \momk \}$. Applying the Leibnitz rule, this equals
\begin{equation} \momk \{ \momi, \momk \} = \momk \omega_J(-I \cdot \trihamvec, K \cdot \trihamvec) = \momk g(\trihamvec,\trihamvec). \end{equation}
Hence $\liouvillevec(\momk) = \momk(1 + g(\trihamvec, \trihamvec))$.
\end{proof}

\begin{lemma} \label{lem:basicliouvilledomain}
The symplectic form $\omega_J$ and the vector field $\liouvillevec$ make $\fOV_{\epsilon}$ into a Liouville domain.
\end{lemma}
\begin{proof}
By Lemma \ref{lem:Liouvilleproperty}, $\liouvillevec$ dilates the symplectic form. On the other hand Lemma \ref{lem:basedilationproperty} shows that $\liouvillevec$ points outwards along the boundary.
\end{proof}

\begin{lemma} \label{lem:basicskeleton}
The skeleton of $\fOV_{\epsilon}$ is the zero fiber $(\momj + I\momk)^{-1}(0) \subset \fOV_{\epsilon}$.
\end{lemma}
\begin{proof}
By Lemma \ref{lem:basedilationproperty}, any point $x \in \fOV_\epsilon$ away from the zero fiber $\{ \momj = \momk = 0 \}$ will escape to the boundary under the Liouville flow. On the other hand, the same Lemma shows that the Liouville flow preserves $\{ \momj = \momk = 0 \}$. Since the latter space is compact, it must equal the non-escaping set. 
\end{proof}

\subsection{Construction and properties of the general Dolbeault space}
Recalling the exact sequence of tori \eqref{eq:basictoriseq}, we observe that $\fOV_{\epsilon}^n$ admits a quasi-hyperhamiltonian action of $T_\R.$ Since in this section we require $\parone\in T^\vee$ to live inside the compact torus $T^\vee_\R,$ it is sensible to use it as a group-valued parameter for the quasi-Hamiltonian part of this action. We will also fix a lift $\Log(\parone) \in \fg^{\vee}_\R$ of $\parone$.
%
\begin{definition}
	The \textbf{Dolbeault multiplicative hypertoric space} $(\Dolparone)_{\epsilon}$ associated to the exact sequence \eqref{eq:basictoriseq} and the data $\parone$ is the quasi-hyperhamiltonian reduction $\fOV_{\epsilon}^n/\!\!/\!\!/_{(\parone,1)} T_\R.$
	\label{defn:dolbeault-var}

\end{definition}

\begin{remark}
For a pair of general parameters $(\parone,\partwo)\in T^\vee\times T^\vee,$ the quasi-hyperhamiltonian reduction $\fOV_{\epsilon}^n/\!\!/\!\!/_{(\parone,\partwo)}T_\R$ is still reasonable to consider, where we take the compact part $\Arg(\partwo)$ of $\partwo$ as a B-field parameter and combine the noncompact parts $\Log(\parone),\Log(\partwo)$ into a $\ft^\vee$-valued parameter for the complex moment map. We will not consider such spaces in this paper.
\end{remark}
We similarly define the hyperhamiltonian reduction $(\tDolparone)_{\epsilon} := \tfOV_{\epsilon}^n/\!\!/\!\!/_{(\Log(\parone),1)} T_\R.$ The lattice $\fg^{\vee}_\Z$ acts by isometries, and $(\Dolparone)_{\epsilon} = (\tDolparone)_{\epsilon} / \fg^{\vee}_\Z$.

The compact subtorus $G_\R \subset G$ acts by hyperhamiltonian transformation on $(\tDolparone)_{\epsilon}$ and quasi-hyperhamiltonian transformations on $(\Dolparone)_{\epsilon}$, with a triplet of moment maps 
\[ (\momi, \momj, \momk) : (\tDolparone)_{\epsilon} \to \mathfrak{g}_\R^{\vee} \times \mathfrak{g}_{\C}^{\vee}. \]
descending to
\[ (\momi, \momj, \momk) : (\Dolparone)_{\epsilon} \to (\mathfrak{g}_\R^{\vee}/ \mathfrak{g}_\Z^{\vee}) \times \mathfrak{g}_{\C}^{\vee}. \]

As in the basic case, we will also need an algebraisation of $(\tDolparone)_{\epsilon}$, denoted $\tDolalgparone$. The latter space was studied in detail in \cite{groechenighypertoric}. Its key property is the  existence of the $\C^\times$-action in Lemma \ref{lem:scalingandlattice}. In fact, we will only use the derivative $\tX$ of this $\C^{\times}$-action, restricted to the open subset $(\tDolparone)_{\epsilon}$. The reader willing to take the existence of a vector field satisfying satisfying \eqref{eq:generaldilatingproperty}, \eqref{eq:generaldilatesthebasetoo} and \eqref{eq:tXdeckgeneral} on faith may skip ahead to Section \ref{sec:generalliouvillefield}.


To construct $\tDolalgparone$, we start with the action of $T \subset D$ on $(\tfOValg)^n$, with its holomorphic moment map 
\[ \mu_\C : (\tfOValg)^n \to \ft^{\vee}_\C. \] The moment map can be understood explicitly as follows. The action of $D$ on $(\tfOValg)^n$ preserves the fibration $(q_i)_{i=1}^n : (\tfOValg)^n \to \C^n = \fd^{\vee}_\C$, and the moment map is obtained by composing this fibration with the dual $\fd^{\vee}_\C \to \ft^{\vee}_\C$ of the inclusion $\ft_\C \to \fd_\C$. Its restriction to $\tfOV_{\epsilon}^n$ equals the complex part of the hyperk\"ahler moment map for $T_\R$.

The paper \cite{groechenighypertoric} explains how the choice of $\Log(\parone)$ determines a $\Hom(G, \C^{\times})$-invariant $T$-equivariant ample line bundle on $(\tfOValg)^n$. The ample line bundle determines a semistable locus $(\tfOValg)^n_{\Log(\parone)-ss}$ of the $T$-action on $(\tfOValg)^n$ in the sense of geometric invariant theory. 
\begin{definition} \label{def:algdolbspace}
$\tDolalgparone$ is the algebraic symplectic reduction 
\begin{equation} \label{eq:algebraicsymplecticreduction}   \mu_\C^{-1}(0)_{\Log(\parone)-ss} / T \subset (\tfOValg)^n_{\Log(\parone)-ss} / T. \end{equation}
\end{definition}

The residual action of $D$ on $\tDolalgparone$ factors through $D/T =G$. It carries an algebraic moment map $\pi : \tDolalgparone \to \mathfrak{g}_{\C}^{\vee}$, obtained by restricting the map $(q_i)_{i=1}^n$ to the moment fiber and descending to the quotient. The character lattice $\fg^{\vee}_\Z$ also acts on $\tDolalgparone$, and this action is a covering space action on the subset $\{ |q_i| < 1 \}$. As in the basic case, we have an open embedding of holomorphic symplectic manifolds
\begin{equation} \label{eq:embedding}
    (\tDolparone)_{\epsilon} \subset \tDolalgparone.
\end{equation}
Since the left-hand side is a hyperk\"ahler reduction and the right-hand side is a GIT quotient, this requires some explanation. We first consider the analytic open subset 
\[ \{ |q_i| < 1 - \epsilon \}_{\Log(\parone)-ss} / T \]
of the right-hand side of \eqref{eq:algebraicsymplecticreduction}. We would like to identify the latter with the symplectic reduction 
\[ \{ |q_i| < 1 - \epsilon  \} /\!\!/_{\Log(\parone))} T_\R = \mu_\R^{-1}(\Log(\parone)) / T_\R \]
with respect to $\omega_I$. If $(\tfOValg)^n$ were compact, this would be a consequence of the Kempf-Ness theorem. The latter boils down to the fact that a $T$ orbit in $(\tfOValg)^n$ is semistable exactly if it intersects the moment fiber $\mu_\R^{-1}(\Log(\parone))$ in a $T_\R$ orbit. In our setting, this follows directly from the explicit description in \cite[Lemma 17]{groechenighypertoric} of the semistable locus of $(\tfOValg)^n$ as a union of toric subvarieties of $(\tfOValg)^n$. Thus the inclusion 
\[ \mu_\R^{-1}(\Log(\parone)) \subset \{ |q_i| < 1 - \epsilon  \} \] descends to an isomorphism of complex manifolds

\begin{equation} \label{eq:KNidentification}
\mu_\R^{-1}(\Log(\parone)) / T_\R = \{ |q_i| < 1- \epsilon  \}_{\Log(\parone)-ss} / T.
\end{equation} 
 Intersecting both sides of \eqref{eq:KNidentification} with the zero fiber of the complex moment map $\mu_\C : \tfOV^n \to \ft^{\vee}_\C$, we obtain the embedding \ref{eq:embedding}. 

To sum up, we have the cartesian square 

\begin{equation} 
\xymatrix{
 (\tDolparone)_{\epsilon} \ar[d]^{\momj + I\momk} \ar[r] & \tDolb^{\operatorname{alg}}_\parone \ar[d]^{\pi} \\
  \mathbb{B}\ar[r] &\fg_\C^{\vee},
}
\end{equation}
where $\mathbb{B}$ is an analytic open neighborhood of the origin in $\fg_\C^{\vee}$ given by $\{ |q_i| < 1- \epsilon  \}$.

\begin{lemma} \label{lem:scalingandlattice}
The diagonal subtorus $\scale \subset (\C^{\times}_w)^n$ preserves the $T$-semistable locus \[\mu_\C^{-1}(0)_{\Log(\parone)-ss} \subset \mu_\C^{-1}(0),\] and thus acts on the quotient $\tDolalgparone$. This action of $\scale$ scales both $\Omega_I^{\operatorname{alg}}$ and the holomorphic moment map $\pi$ with weight one. 
\end{lemma}

This dilating action also plays a crucial role in the prequel to this paper: see \cite[Section 4.3]{McBW}. \label{prop:scalingaction} As in the basic case, $\scale$ does not preserve $(\tDolparone)_{\epsilon} \subset \tDolalgparone.$ 

\begin{lemma}
For any $\tau \in \fg^{\vee}_\Z,$ we have $\tau^* \momi = \momi + \tau$ as functions $(\tDolparone)_{\epsilon} \to \mathfrak{g}^{\vee}_\R$.
\end{lemma}

Consider the surjection $D = (\C^{\times})^n \to G$, (the real part of) its derivative $\R^n \to \fg_\R$ and the dual embedding $\fg^{\vee}_\R \to \R^n$.

We write $\langle - , - \rangle $ for the pull back of the standard inner product on $\R^n$ to $\fg^{\vee}_\R$. It defines an isomorphism $\fg^{\vee}_\R \to \fg_\R$, denoted $\tau \to \hat{\tau}$, such that $\tau(\hat{\sigma}) = \langle \tau, \sigma \rangle.$

\begin{lemma} \label{lem:tauanddilations}
For $\tau \in \fg^{\vee}_\Z$ and $(z, w) \in G \times \C^{\times}_w$ we have
\begin{equation} \label{eq:tauanddilations}
\tau^* (z,w)^* = (\hat{\tau}(w)z, w)^* \tau^*.
\end{equation}
where on the right-hand side we view $\hat{\tau}$ as a cocharacter of $G$. 
\end{lemma}
\begin{proof}
The action of $\tau \in \fg^{\vee}_\Z$ on $\tDolalgparone$ is defined by the action of its pullback $\tau' \in \Hom(D, \C^{\times})$ on $(\tfOValg)^n$, restricted to the zero fiber of the moment map. Applying \eqref{eq:basictauanddilations} to this lift, we obtain the result.
\end{proof}

\subsection{Making \texorpdfstring{$(\Dolparone)_{\epsilon}$}{D beta epsilon} into a Liouville domain} \label{sec:generalliouvillefield}
As in the basic case, let $\hamvec_f$ denote the Hamiltonian vector field generated by $f$ with respect to the form $\omega_J$. Given $\tau \in \fg^{\vee}_\R$, pairing with the various moment maps defines three functions $\langle \tau, \momi \rangle, \langle \tau, \momj \rangle, \langle \tau, \momk \rangle : \tDolb \to \R$.  

\begin{definition}
Let $\trihamvec_{\tau} := \hamvec_{\langle \tau, \momj \rangle}$. 
\end{definition}
It is the Hamiltonian vector field generating the $S^1$-action associated with the cocharacter $\hat{\tau}$ of $G_\R$.

\begin{definition}
Let $\tX$ be the vector field on $(\tDolparone)_{\epsilon}$ generating the action of $\R^* \subset \C^{\times}_w$.
\end{definition}

By Lemma \ref{lem:scalingandlattice} we have
\begin{equation} \label{eq:generaldilatingproperty} \Lieder_{\tX} (\omega_J + I \omega_K) = \omega_J + I \omega_K  \end{equation}
and
\begin{equation} \label{eq:generaldilatesthebasetoo} \Lieder_{\tX} (\momj + I \momk) = \momj + I \momk. \end{equation}
From \eqref{eq:tauanddilations} we deduce
\begin{proposition} \label{prop:liouvilleoncover}
\begin{equation} \label{eq:tXdeckgeneral} \tau^* \tX = \tX + I \cdot \trihamvec_\tau. \end{equation}
\end{proposition}



As in the base case, we will add a certain Hamiltonian vector field to $\tX$ to obtain something $\Z$-invariant. 

Fix an orthonormal basis $\{ \sigma_a \}_{a =1}^{\rk G}$ of $\fg^{\vee}_\R$ with respect to our inner product $\langle - , - \rangle$. We can expand $\momi = \sum_{a=1}^{\rk G} \momi^a \sigma_a$ and likewise for $\momj, \momk$. 
\begin{definition} 
 Define
\[ H := \sum_{a = 1}^{\rk G} \momi^a \momk^a. \]
\end{definition}
Note that $H$ depends non-trivially on the choice of orthonormal basis.
\begin{definition}
\[ \liouvillevec := \tX +  \hamvec_H \]
\end{definition}
To show that this is invariant under deck transformations $\tau \in \fg^{\vee}_\Z$, we make the following calculations.

\begin{lemma}
\[ \tau^* H = H + \langle \tau, \momk\rangle. \]
\end{lemma}
\begin{proof}
\[ \tau^* H = \sum_{a = 1}^{\rk G} \tau^* \momi^a \momk^a = H + \sum_{a = 1}^{\rk G} \langle \tau, \sigma_a \rangle \momk^a = H + \langle \tau, \momk\rangle. \]
In the last step, we used the orthonormality of the basis $\sigma_a$.
\end{proof}

The proof of Lemma \ref{lem:hamvecofmomiandmomk} extends immediately to give 
\begin{equation} \hamvec_{\langle \tau, \momk \rangle } = - I \cdot \trihamvec_{\tau}. \end{equation}

\begin{corollary}
\[ \tau^*\hamvec_H = \hamvec_H - I \cdot \trihamvec_\tau. \]
\end{corollary}

\begin{corollary} \label{cor:descendstodolb}
The vector field $\liouvillevec$ is invariant under the action of $\fg^{\vee}_\Z$ on $(\tDolparone)_{\epsilon}$, and descends to a vector field on $(\Dolparone)_{\epsilon}$ satisfying
\[ \Lieder_\liouvillevec \omega_J = \omega_J. \]
\end{corollary}
We now generalize Lemma \ref{lem:basedilationproperty}, and show that the vector field $\liouvillevec$ dilates the base $\mathbb{B}$ of the hypertoric Hitchin map $(\Dolparone)_{\epsilon} \to \mathbb{B}$. More precisely, we will show that its flow increases the functions
\[ \langle \momk, \momk \rangle = \sum_{a = 1}^{\rk G} (\momk^a)^2 \mbox {   \&   } \langle \momj, \momj \rangle = \sum_{a = 1}^{\rk G} (\momj^a)^2.\]

\begin{lemma} \label{lem:dolbbasedilatingproperty}
We have \begin{equation} \liouvillevec(\langle \momj, \momj \rangle) = 2 \langle \momj, \momj \rangle  \end{equation}
and
\begin{equation} \liouvillevec( \langle \momk, \momk \rangle ) = 2( \langle \momk, \momk \rangle + g(\mathscr{R}, \mathscr{R})) \end{equation}
where $\mathscr{R} = \sum_{a=1}^{\rk G} \momk^a \trihamvec^a$.
\end{lemma}
\begin{proof}

By \eqref{eq:generaldilatesthebasetoo}, we have $\tX( \momj^a) = \momj^a$ and $\tX(\momk^a) = \momk^a$ for each $1 \leq a \leq \rk G$. Thus
\[ \tX( \langle \momk, \momk \rangle) = \tX (\sum_{a = 1}^{\rk G} (\momk^a)^2) = 2(\sum_{a = 1}^{\rk G} (\momk^a)^2) = 2 \langle \momk, \momk \rangle \] 
and likewise
\[ \tX( \langle \momj, \momj \rangle ) = 2\langle \momj, \momj \rangle. \]
We have $\hamvec_H(\momj) = -\trihamvec(H) = 0$, which proves the first equality. On the other hand, we can expand  
\begin{align*} \label{eq:poissonbracket}  
\hamvec_H(\langle \momk, \momk \rangle) &  = \{ H, \sum_{a = 1}^{\rk G} (\momk^a)^2 \} = \sum_{a, b = 1}^{\rk G}  \momk^b \{ \momi^b, (\momk^a)^2 \} \\ & = 2\sum_{a, b = 1}^{\rk G}  \momk^a \momk^b \{ \momi^b, \momk^a \} = 2\sum_{a, b = 1}^{\rk G} \momk^a \momk^b g(\trihamvec^a, \trihamvec^b) 
 \\ & = 2g(\mathscr{R}, \mathscr{R}).  
\end{align*}
This proves the second equality.
\end{proof}


\begin{lemma} \label{lem:generaldolbeaultLiouvilledomain}
The symplectic form $\omega_J$ and the vector field $\liouvillevec$ make $(\Dolparone)_{\epsilon}$ into a Liouville domain.
\end{lemma}
\begin{proof}
Corollary \ref{cor:descendstodolb} gives the dilating property $\liouvillevec \omega_J = \omega_J$. Lemma \ref{lem:dolbbasedilatingproperty} shows that it is outwards pointing along the boundary.
\end{proof}

\begin{lemma} \label{lem:generalskeleton}
The skeleton of $\liouvillevec$ is the zero fiber $\pi^{-1}(0) = (\momj + I\momk)^{-1}(0) \subset (\Dolparone)_{\epsilon}$.
\end{lemma}
\begin{proof}
The vector field $\mathscr{R}$ vanishes along $\{ \momk^a = 0 \}$. As in the proof of Lemma \ref{lem:basicskeleton}, we conclude that the flow preserves the zero fiber and sends everything else off to the boundary. 
\end{proof}
\subsection{The central fiber} 
\label{sec:centralfiber}



The key structure we have used in defining the Liouville flow is the action of $\scale$ on $\tDolalgparone$. This action fixes setwise the central fiber $\tDolfiber:=\pi^{-1}(0)$, while flowing all other fibers outward. It should now be clear why it was necessary to work on $\tDolparone^{\operatorname{alg}}$ when defining this action: first, different fibers of the complex-valued moment map on $(\Dolparone)_{\epsilon}$ are nonequivalent abelian varieties, so there can be no $\C^\times$ action identifying them with each other; hence, it is necessary to pass to the universal cover $\tDolparone$, where fibers become copies of $(\C^\times)^k.$ And second, the $\C^\times$ action does not preserve the small neighborhood $\mathbb{B},$ so it is necessary to pass from there to the algebraic variety $\tDolalgparone$.

We will be very interested in the central fiber $\tDolfiber.$ It also lives inside of $(\tDolparone)_{\epsilon},$ and we will denote by $\Dolfiber$ its image in $(\Dolparone)_{\epsilon}$, in accord with the case where $\Dolb=\fOV_{\epsilon}$ is the Tate curve and $\Dolfiber$ is the nodal curve defined in \eqref{eq:tatenodalfiber}.
The components of $\tDolfiber$ are smooth complex Lagrangians: in fact, they are toric varieties intersecting cleanly along toric subvarieties. 
%
\begin{example}
	Let $\Dolb_\epsilon =\fOV_{\epsilon}$. Then the space $\tDolfiber$ is an infinite chain of projective lines, meeting nodally at 0 and $\infty.$ 
\end{example}

We can encode the data of these toric varieties and their intersections into a periodic hyperplane arrangement inside of $\mathfrak{g}^\vee_\R,$ defined in terms of the exact sequence \eqref{eq:basictoriseq} and the parameter $\parone\in T_\R^\vee.$
Let $G^{\vee,\parone}_\R$ denote the preimage of $\parone$ under the composition $T^{\vee}_\R \leftarrow D^\vee$; this is an affine subtorus modeled on $G^{\vee}_\R$. The intersection of $G^{\vee,\parone}_\R$ with the coordinate subtori in $D^\vee_\R$ form a toric arrangement in $G^{\vee,\parone}_\R\cong G^\vee_\R.$ We can pull this back via the cover $\mathfrak{d}^{\vee}_\R \to D^{\vee}_\R$ to obtain a periodic hyperplane arrangement on $\mathfrak{g}^{\vee, \beta}_\R$, the translate of $\mathfrak{g}^{\vee}_\R$ by (any choice of) $\log \beta$. 


\begin{definition}\label{def:per-arrangement}
We denote by $\Bper_\parone$ the periodic hyperplane arrangement in $\mathfrak{g}_\R^{\vee, \beta}$ defined as above, and we write $\Btor_\parone$ for the toric arrangement which is its image in $G^{\vee,\parone}_\R.$
\end{definition}

In other words, the chambers of the periodic hyperplane arrangement $\Bper_\parone$ are defined by \[
\Delta_{\Bx}=\{\Ba\in \fg^{\vee,\parone}_\R \mid x_i < a_i <
x_i+1\},
\]  
where $\Bx \in \Z^n \cong \mathfrak{d}^\vee_\Z.$ is a character of $D$.
\begin{proposition}[{\cite[Proposition 4.11]{McBW}}]\label{prop:toricskel}
The irreducible components of the variety $\tDolfiber$ are the toric varieties $\mathfrak{X}_\Bx$ associated to the polytopes $\Delta_\Bx$, for $\Delta_\Bx\neq 0$. The images of $\mathfrak{X}_\Bx$ in $\fg^{\vee}_\R$ under the $G_\R$-moment map are the chambers of the arrangement $\Bper_\parone$. They intersect cleanly along toric subvarieties indexed by the intersections $\Delta_\Bx \cap \Delta_\By$.
\end{proposition}
\begin{lemma} \label{lem:weinsteinnhds}
Each component $\mathfrak{X}_{\Bx}$ admits an open neighborhood $U_\Bx$ in $(\tDolparone)_{\epsilon}$ isomorphic to an open neighborhood of $\mathfrak{X}_{\Bx}$ in its cotangent bundle. The intersection of $\mathfrak{X}_{\By}$ with this neighborhood is identified with the conormal to a toric stratum.
\end{lemma}
\begin{proof}
First consider the basic case $\widetilde{\fOV}_\epsilon$. The components of $\widetilde{\Dolfiber}$ are given by rational curves $\mathbb{P}^1_k$ labeled by an integer $k \in \Z$.
In this case, we can take $U_k$ to be a standard neighborhood of $\mathbb{P}^1_k.$
For a general Dolbeault space $\widetilde{\fOV}^n/\!\!/\!\!/_{(\parone,1)}T_\R,$, the components $\mathfrak{X}_{\Bx}$ of $\widetilde{\Dolfiber}$ are symplectic reductions of components $\prod_{i=1}^n (\mathbb{P}^1)_{k_i} \subset \widetilde{\fOV}^n$. The desired neighborhood is defined by $\widetilde{\Dolb}_\Bx := \prod_{i=1}^n U_{k_i} /\!\!/\!\!/_{(\parone,1)}T_\R$. 
\end{proof}

\begin{remark}If we were to relax our assumption on the unimodularity of the subtorus $T_\R\subset (S^1)^n$, the clean intersections statement would need to be interpreted in the sense that there is an orbifold chart around the intersection with the toric subvarieties corresponding to coordinate subspace, with the finite group acting by diagonal matrices. As we will see, our microlocal-sheaf computation applies to this case as well.
\end{remark}

\begin{remark}We can now simplify our Definition~\ref{def:paronegeneric} of genericity for $\parone\in\ft^\vee_\R$: the parameter $\parone$ is generic if and only if the hyperplane arrangement $\Bper_\parone$ is simple, in the sense that any $k$ hyperplanes always intersect in codimension $k$.
\end{remark}

\subsection{Liouville structure and skeleton of \texorpdfstring{$\mht_{(\parone, \tone)}$}{U(beta,1)}.} \label{sec:liouvillestruct}

We are now ready to use the Liouville domain $(\Dolparone)_{\epsilon}$ to define a Liouville structure on the multiplicative hypertoric variety $\mht_{(\parone,1)}.$ It will be useful to recall the following singular Lagrangian torus fibration on $\mht_{(\parone,\tone)}.$

\begin{definition}\label{defn:syzmap}
Let $\mathfrak{m} \colon \mht_{(\parone,1)} \to \fg^{\vee}_\C$ be the map $\mathfrak{m} := \log | \mmm_{\C^\times}| + i \mmm_\R$ (generalizing the map $\mmm_{I,\C}$ defined in Example~\ref{ex:torusfib}). 
\end{definition}

In Example \ref{ex:torusfib}, we describe how the map $\mmm_{I,\C}$ on $\TAo$ gives a fibration of $\TAo$ by 2-tori, with one singular fiber, the nodal torus $\mmm_{I,\C}^{-1}(0).$ It is straightforward to identify the preimage of a disk $\mmm_{I,\C}^{-1}(\D)$, as a smooth manifold with a singular torus fibration, with the Tate curve $\fOV\to \D,$ and in fact, such an identification exists for general multiplicative hypertoric varieties:

\begin{proposition} [{\cite[Theorem 9.6]{DMS19}}] \label{NAHT}
There exists an open $G_\R$-equivariant embedding \begin{equation}
\label{eq:NAHC} (\Dolparone)_{\epsilon} \hookrightarrow \mht_{(\parone,1)} \end{equation} intertwining $\mmm_{I,\C}$ and $\mathfrak{m}$. 
\end{proposition}

\begin{proposition}\label{prop:mht-skel}
The completion of $(\Dolparone)_{\epsilon}$ is diffeomorphic to $\mht_{(\parone,1)},$ equipped with a Liouville structure whose skeleton is the central fiber $\mathcal{L}:=\mathfrak{m}^{-1}(0)$ of the map $\mathfrak{m}$ described in Definition \ref{defn:syzmap}.
\end{proposition}
\begin{proof}
The space $(\Dolparone)_\epsilon$ is obtained from the Dolbeault space defined in \cite{DMS19} by imposing $|q| < 1 - \epsilon$. The retraction defined in Lemma 6.37 of \cite{DMS19}, for the star-convex set $|q| < 1 - \epsilon$, is a diffeomorphism preserving the central fiber. The diffeomorphism required here is obtained by composing this with the diffeomorphisms constructed in \cite[Theorem 9.6 and Corollary 9.10]{DMS19}. 
\end{proof}
This is the Liouville structure on $\mht_{(\parone,1)}$ which we will henceforth use in this paper. Recall from \cite{GPS2} that a Liouville manifold is called weakly Weinstein if its skeleton $\mathcal{L}$ is mostly Lagrangian, and each smooth point $p \in \mathcal{L}$ admits a generalized cocore. The latter is an exact Lagrangian cylindrical at infinity with respect to the Liouville flow, which intersects $\mathcal{L}$ transversely at $p$ and nowhere else. 

  \begin{lemma} \label{lem:weaklyweinstein}
  The Liouville manifold $\mhta$ is weakly Weinstein. 
  \end{lemma}
  \begin{proof}
      The skeleton is mostly Lagrangian because it is a finite union of smooth Lagrangian submanifolds. 

      The generalized cocore through a smooth point $x \in \mathcal{L}$ consists of a closed Lagrangian disk $D_x$ intersecting $\mathcal{L}$ transversely at $x$, extended to a cylindrical Lagrangian disk by applying the Liouville flow to the boundary $\partial D_x$. 
      
      Recall from Proposition \ref{lem:dolbbasedilatingproperty} that $\liouvillevec(F) \geq F$, where $F = \langle \momj, \momj \rangle + \langle \momk, \momk \rangle$. This ensures that the extension along the flow is closed.

      More explicitly, fix a component $\mathfrak{X} \subset \widetilde{\mathcal{L}}$. By construction, the pullback of the Liouville vector field decomposes as $\liouvillevec = \tX + \hamvec_H$ on $\tDolparone$. As in the proof of Proposition \ref{lem:dolbbasedilatingproperty}, we have $\liouvillevec(F) = 2F$ and $\hamvec_H(F) \geq 0$. 
      
      On the other hand, there is a cocharacter $\sigma : \C^{\times} \to G$ generating a Hamiltonian vector field $\hamvec_\sigma$, such that in a Weinstein neighborhood of $\mathfrak{X}$ we have $\tX = \liouvillevec_{\operatorname{std}} + \hamvec_{\sigma}$. Here $\liouvillevec_{\operatorname{std}}$ is the standard Liouville vector field on $T^*\mathfrak{X}$. Because $G$ preserves $\momj, \momk$, we have $\hamvec_\sigma(F)=0$.

     Consider a smooth function $\rho : \R_{\geq 0} \to \R_{\geq 0}$ such that $\rho(x)=1$ for $x < 1/4$ and $\rho(x) = 0$ for $x > 1/2$. Then $\liouvillevec_\rho := \liouvillevec - \rho(F)(\hamvec_{\sigma} + \hamvec_H)$ satisfies $\liouvillevec_\rho(F) \geq F$. It coincides with $\liouvillevec_{\operatorname{std}}$ near $\mathfrak{X}$ and with $\liouvillevec$ near $\partial^{\infty} (\Dolparone)_\epsilon$. The attracting cell of $x$ with respect to $\liouvillevec_\rho$ is the requisite generalized cocore.


      \end{proof}

\begin{remark}\label{rem:weak}
As mentioned in Warning \ref{warn:weakly}, we expect that the Liouville structure defined in Proposition \ref{prop:mht-skel} 
is homotopic to a Weinstein structure, arising as in \cite{Eli-Ciel} from the natural Stein structure on $\mht_{(\parone,1)}$. More precisely, there should be a sufficiently Weinstein homotopy (as defined in \cite{NS20}) between the underlying Liouville structures. This would ensure that the resulting Fukaya categories are equivalent.    
However, we do not prove this, and prefer to use the Liouville structure constructed above without perturbation.
\end{remark}
%

\subsection{Polarizations}\label{subsec:polarizations}

Before we can proceed to a discussion of categories associated to these geometries, we need one more geometrical input about the spaces $\mht,$ which is the existence of a stable polarization.

\begin{definition} \label{def:polarization}
    A {\em stable polarization} of a symplectic manifold $(X,\omega)$ is a section of the bundle of stabilized Lagrangian Grassmannians $\LGr(T_xX\oplus \C^N),$ for $N\gg 0.$
\end{definition}

As we will discuss in the next section, this stable polarization can be understood as providing the data necessary to define the Fukaya category, or as the data necessary to glue together local microlocal-sheaf calculations.

\begin{definition}
    A {\em stable holomorphic polarization} of a holomorphic symplectic manifold $(X,\Omega)$ is a section of the bundle of stabilized holomorphic Lagrangian Grassmannians $\LCGr(T_xX\oplus T^*\C^N),$ for $N\gg 0.$
\end{definition}
A stable holomorphic polarization for $\Omega$ defines a polarization in the sense of Definition \ref{def:polarization} for $\omega = \Re(\Omega).$

The desirable features of holomorphic Lagrangians can be understood through the map $\LCGr(T^*\C^N) = Sp(N)/U(N)\to U(2N)/O(2N) = \LGr(\C^{2N}).$ In particular, the domain of this map is simply-connected, so that a family of holomorphic Lagrangian subspaces in $T^*\C^N,$ considered as real Lagrangians in $\C^{2N},$ will never wrap the Maslov cycle in $\LGr(\C^{2N}).$ The absence of any ``Maslov anomaly'' will be ultimately responsible for the existence of a perverse t-structure on the categories we describe in the next section.

\begin{lemma} \label{lem:polarisationfromreduction}
Let $X$ be a hyperk\"ahler manifold with a hyperhamiltonian action of a torus $T_\R$. Suppose $X$ has a $T_\R$-equivariant holomorphic polarisation $P \subset TX$. Then any smooth hyperk\"ahler reduction $Y = \mu_{\mathbb{H}}^{-1}(\lambda) / T_\R$ of $X$ has a stable holomorphic polarisation.
\end{lemma}
\begin{proof}
The restriction of $TX$ to the moment fiber $\mu_{\mathbb{H}}^{-1}(\lambda)$ may be identified with $\pi^* TY \oplus T^*\ft$ where $\ft$ is the complexified Lie algebra of $T_\R$ and $\pi : \mu_{\mathbb{H}}^{-1}(\lambda) \to \mu_{\mathbb{H}}^{-1}(\lambda) / T_\R = Y$ is the quotient map. Thus the descent $P$ to $Y$ is a Lagrangian subbundle of $TY \oplus T^* \ft$.
\end{proof}
\begin{corollary}\label{cor:hol-pol}
The manifold $\Dolparone$ admits a stable holomorphic polarization.
\end{corollary}
\begin{proof}
We start by constructing an $S^1$-invariant holomorphic polarisation on the basic space $\fOV$. Consider the bundle of holomorphic Grassmanians $\LCGr_\fOV$ over $\fOV$. The action of $S^1$ on $\fOV$ induces an action on $\LCGr_\fOV$; we want to find an $S^1$-invariant section of this bundle, or equivalently a section of the quotient fibration 
$\LCGr_\fOV / S^1 \to \fOV / S^1 = \D \times S^1$.

The fiber away from $0 \times 1$ is $\mathbb{P}^1$, and the fiber over $0 \times 1$ is $\mathbb{P}^1 / S^1 = [0,1]$. The end points of this interval are the images of the two $S^1$-fixed points in $\LCGr_\fOV$.

A continuous section of the fibration $\LCGr_\fOV / S^1 \to \fOV / S^1 \cong \D \times S^1$, passing through $0 \in [0,1]$, may be constructed as follows. Cover $S^1$ by two closed intervals $A_1, A_2$ meeting at two points. Suppose $A_1$ contains $1$ in its interior. The preimage $U_2 \subset \LCGr_\fOV / S^1 $ of $A_2$ is a $\mathbb{P}^1$-bundle over $\D \times [0,1]$. Since the base is contractible, this bundle is trivial, and we fix a trivialisation. The part of $\fOV$ living over $A_1$ is topologically (the closure of) a neighborhood of $0$ in $T^*\C$, which carries the usual cotangent polarization. This defines a continuous section on the preimage $U_1 \subset \LCGr_\fOV / S^1$ of $A_1$.  Restricted to either component of the boundary, this defines a continuous section of a $\mathbb{P}^1$ bundle over $\D$. The space of such sections is path connected, since $\D$ is contractible and $\mathbb{P}^1$ is connected. We extend the polarization over $U_2$ by choosing a path between the two boundary sections.

By taking products, we obtain a $T_\R$-invariant polarisation on $\TAno$. This induces a stable polarisation on the reduction $\Dolparone = \fOV^n /\!\!/\!\!/ {T_\R}$, by Lemma \ref{lem:polarisationfromreduction}.
\end{proof}

We now recall the notion of orientation class of a polarization. We follow the discussion in \cite[Section 3]{cote2022perverse} and refer the reader to that paper for further details. 

Let $O$ and $U$ denote the stable orthogonal and unitary groups. We view $O$ as the subgroup of $U$ preserving a fixed Lagrangian submanifold. Suppose that $Y$ is a symplectic manifold, with stable tangent bundle classified by a map $Y \to B(U)$. A stable polarization is a lift of this map to 
\begin{equation} \label{eq:polarizationlift} Y \to B(O). \end{equation} 
This lift determines a nullhomotopy of the composite map 
\begin{equation} \label{eq:orientationseq} Y \to B(U) \to B(U/O) \to B^2(O) \end{equation}
induced by the tautological nullhomotopy of the composition
\[ B(O) \to B(U) \to B(U/O). \]
A second nullhomotopy of \eqref{eq:orientationseq} is induced by the natural nullhomotopy of
\[ B(U) \to B(U/O) \to B^2(O). \]
The difference of the two nullhomotopies defines a map
\begin{equation} \label{eq:differenceofnulls} Y \to B(O).\end{equation}
which is canonically identified with \eqref{eq:polarizationlift}.
The orientation class of the polarization is the element of $H^2(Y,\Z/2\Z)$ classified by the composition of \eqref{eq:differenceofnulls} with $w_2 : B(O) \to B^2(\Z/2\Z)$. It thus coincides with with the second Stiefel-Whitney class of any bundle representing the polarization. 

Returning to our setting, we write $\eta$ for the stable polarization on $\Dolparone$ constructed in Corollary \ref{cor:hol-pol}, and $w_2(\eta) \in H^2(\Dolparone, \Z/2\Z)$ for its orientation. On the other hand, each component $\mathfrak{X}_{\Bx}$ of $\tDolfiber$ is a smooth manifold, which maps to an immersed manifold in $\Dolparone$. We denote by $w_2(\mathfrak{X}_{\Bx}) \in H^2(\mathfrak{X}_{\Bx}, \Z/2\Z)$ the second Stiefel-Whitney class of its cotangent bundle.

\begin{lemma}
The pullback of $w_2(\eta)$ by the immersion $\mathfrak{X}_{\Bx} \to \Dolparone$ equals $w_2(\mathfrak{X}_{\Bx})$.  
\end{lemma}
\begin{proof}
We start by proving an equivariant version of the claim on $\tfOV$. The polarization of $\fOV$ constructed in Corollary \ref{cor:hol-pol} is a rank one $S^1$-invariant subbundle of the tangent bundle. It therefore defines $S^1$-equivariant complex line bundle on $\tfOV$. Because the orientation class is the second Stiefel-Whitney class of this bundle, it has a natural lift to an $S^1$-equivariant class. The latter pulls back to an element of $H^2_{S^1}(\mathbb{P}^1, \Z/2\Z)$ which we may write as
\[ \alpha_0 [0] + \alpha_\infty [\infty] \] 
where $\alpha_0, \alpha_\infty \in \Z/2\Z$ and $[0], [\infty]$ are the classes of the two $S^1$ fixed points $0, \infty \in \mathbb{P}^1$. Using equivariant localisation, we find that $\alpha_0, \alpha_\infty$ are the parities of the $S^1$-representations defined by the fibers of the polarisation over $0, \infty$. Thus $\alpha_0 = \alpha_\infty = 1$. The result coincides with the $S^1$-equivariant Stiefel-Whitney class $w_2(\mathbb{P}^1)$, which is the reduction modulo two of the perhaps more familiar first Chern class of $\mathbb{P}^1$. 

We can apply the same argument to the product case $\Dolparone = \fOV^n$. The polarization is now viewed as an $(S^1)^n$-equivariant subbundle of $T\Dolparone$, and as in the basic case its orientation class pulls back to the equivariant Stiefel-Whitney class $w_2((\mathbb{P}^1)^n)$. 

We now consider the case of a general $\Dolparone$, which we view as a hyperk\"ahler reduction of $\fOV^n$. Fix a component $\mathfrak{X}_{\Bx}$ of the central fiber. As in the proof of Lemma \ref{lem:weinsteinnhds}, we view $\mathfrak{X}_{\Bx}$ as the symplectic reduction of $\prod_{i=1}^n(\mathbb{P}^1)_{k_i}$ by a subtorus $T_\R \subset (S^1)^n$. 

The claim now follows from the fact that both $w_2(\eta)|_{\mathfrak{X}_{\Bx}}$ and $w_2(\mathfrak{X}_{\Bx})$ are obtained by descent from the corresponding equivariant classes on $(\mathbb{P}^1)^n$. For the orientation class, this is clear by construction. For the Stiefel-Whitney class, we argue as follows. Suppose more generally that a torus $T$ acts by Hamiltonian automorphisms on a symplectic manifold $Y$, and let $X := Y /\!/\!_{\lambda} T = \mu^{-1}(\lambda)/T$ be the reduction at a regular value of the moment map. Let $\pi : \mu^{-1}(\lambda) \to X$ be the projection. We have $\pi^{-1}TX \oplus N_{\mu^{-1}(\lambda)}Y \cong TY_{\mu^{-1}(\lambda)} / \mathfrak{t}$, where the right-hand side is the natural quotient by the Lie algebra of $T$. The differential $d \mu$ gives an equivariant trivialisation of the normal bundle $N_{\mu^{-1}(\lambda)}Y$. Since Stiefel-Whitney classes are invariant under sums with and quotients by trivial bundles, we find $\pi^{-1}(w(TX)) = w(\pi^{-1}TX) = w(TY|_{\mu^{-1}(\lambda)})$. The second equality holds at the level of equivariant classes, and thus $w(TX)$ equals the descent of the $T$-equivariant Stiefel Whitney class of $TY$. 
 \end{proof}

\section{Microlocal perverse sheaves and the Fukaya category}

\subsection{Microlocal sheaves}\label{subsec:microlocal}
Let $X$ be a manifold. We write $\Sh(X)$ for the dg category of complexes of (possibly infinite-rank) constructible sheaves on $X$: for every object $\cF\in\Sh(X),$ there exists a stratification $X=\bigsqcup X_\eta$ such that the cohomology of the restriction of $\cF$ to each $X_\eta$ is a local system, possibly of infinite rank. (We emphasize that this is in contrast to the usual definition of constructibility, which generally assumes finite-dimensionality of the stalks.)

For $\Lambda\subset T^*X$ a conical Lagrangian, we write $\Sh_\Lambda(X)$ for the full subcategory of sheaves on $X$ with singular support contained in $\Lambda.$ This category localizes to a conic sheaf $\mu \Sh_\Lambda$ of dg categories on the Lagrangian $\Lambda$ (hence also a sheaf $\overline{\mu \Sh}_{\partial \Lambda}$ on the Legendrian $\partial \Lambda$), the sheafification of a presheaf of categories whose stalks are described in \cite{KS94}. Each smooth point $(x,\xi)\in \Lambda$ determines a functor
\[
\Sh_\Lambda(X)\to \mathsf{Vect}_k^{\operatorname{dg}}
\]
defined by taking a sheaf $\cF$ to the microstalk of $\cF$ at $(x,\xi),$ which we denote by $\cF_{(x,\xi)}.$ For $\xi\neq 0,$ this is also known as the vanishing cycles of $\cF$ at $x$ in codirection $\xi.$

\begin{definition}
We denote by $\Shc_\Lambda(X)$ the category of compact objects in $\Sh_\Lambda(X).$
\end{definition}
As discussed in \cite{Nwms,GPS3}, the left adjoints of the restriction maps for the sheaf $\mu \Sh_\Lambda$ preserve compact objects, so that $\Shc_\Lambda(X),$ equipped with the functoriality coming from these left adjoints, localizes to a {\it co}sheaf of categories on $\Lambda.$
\begin{definition}\label{def:mush-c-cotangent}
    We denote by $\mu\Shc_\Lambda$ (resp. $\overline{\mu\Sh}^c_{\partial \Lambda}$) the cosheaf of dg categories obtained from $\mu\Sh_\Lambda$ (resp. $\overline{\mu\Sh}_{\partial \Lambda}$) by taking compact objects.
\end{definition}

As first observed in \cite{Nwms}, it is easy to describe a set of generators of the category $\Shc_\Lambda(X)$ (which, from the Fukaya-categorical perspective discussed below, will correspond to transversal Lagrangian disks through points of $\Lambda$). Write $\Lambda^\circ$ for the complement of the singular locus in $\Lambda$; it will be a disjoint union of smooth components $\Lambda^\circ=\bigsqcup_{i=1}^r\Lambda_i^\circ$.
\begin{proposition}[{\cite[Corollary 4.21]{GPS3}}]\label{prop:gens}
Let $(x_i,\xi_i)$ be a point in $\Lambda_i^{\circ}.$ Then the functor taking an object $\cF$ to its microstalk $\cF_{(x_i,\xi_i)}$ is corepresented by an object $\cP_i$ of $\Shc_\Lambda(X).$ The objects $\cP_1,\ldots,\cP_r$ generate $\Shc(\Lambda).$
\end{proposition}

In \cite{Nwms}, the cosheaf of categories $\mu\Shc_\Lambda$ was defined, as above, for $\Lambda$ a conical Lagrangian in a cotangent bundles; the definition was generalized in \cite{Sh-hprinciple,NS20} to the case where $\Lambda$ is a conical Lagrangian in a general (sufficiently) Weinstein manifold, equipped with a stable polarization (or more generally with Maslov data).

\begin{theorem}[\cite{Sh-hprinciple}]\label{thm:definition-of-ns-mush}
Let $W$ a weakly Weinstein manifold, $\sigma$ a stable polarization of $W$ in the sense of Definition \ref{def:polarization}, and $\Lambda\subset W$ a conic Lagrangian containing the skeleton $\LL_W$ of $W,$ so that $\Lambda = \LL_W\cup \operatorname{Cone}(\partial^\infty\Lambda).$ Then there is a cosheaf of dg categories $\mu\Sh^c_{\Lambda,\sigma}$ on $\Lambda$.
If $W$ is a cotangent bundle equipped with its natural polarization $\sigma_{\mathsf{fib}}$ by cotangent fiber directions, then there is an equivalence
\begin{equation}\label{eq:ks-is-ns-musheaves}
\mu\Sh^c_{\Lambda,\sigma_{\mathsf{fib}}}\cong \mu\Sh^c_\Lambda
\end{equation}
of this cosheaf with the cosheaf defined in \ref{def:mush-c-cotangent} above.
\end{theorem}



It is further established in \cite{NS20} that the global section category $\mu\Shc_{\Lambda,\sigma}(\Lambda)$ of this cosheaf is invariant under Weinstein homotopies. In fact, the methods of that paper, combined with the tools developed in \cite{GPS2}, are sufficient to prove something stronger:

\begin{theorem}[{\cite[Theorem 1.4]{GPS3}}]\label{thm:gps-main}
Let $W,\sigma,\Lambda$ as above, and write $-\Lambda=\LL_W\cup\operatorname{Cone}(-\partial^\infty\Lambda)$ for the conic Lagrangian obtained from $\Lambda$ by applying the involution negating cotangent fibers. Then there is an equivalence of categories
\begin{equation}
     \mu\Shc_{-\Lambda,\sigma}(-\Lambda)\cong \Fuk(W,\partial^\infty\Lambda)
\end{equation}
between the global sections of the above cosheaf (for $-\Lambda$) and the Fukaya category of the Liouville sector of $W$ stopped at $\partial^\infty\Lambda.$ (If $\partial^\infty\Lambda=\emptyset,$ so $\Lambda=\LL_W,$ then this is the usual Fukaya category of $W$.)
\end{theorem}

This suggests a method of computing the Fukaya category of a stably polarized weakly Weinstein manifold $W$: first, determine a skeleton $\LL_W$ for $W$ together with an open cover $\LL=\bigcup \LL_i$ such that the $\LL_i$ are conic Lagrangians in cotangent bundles. Second, study the cosheaves $\mu\Sh^c_{\LL_i}.$ Finally, the Fukaya category of $W$ will be computed as global sections of the cosheaf $\mu\Sh^c_{\LL},$ which can be glued together from its values on the $\LL_i,$ so it remains only to compute that colimit. The first step is what we have accomplished in the previous section; as we have seen, in our case, each local piece $\LL_i$ will be of the form of the union of a toric variety together with conormals to its toric strata.

As colimits of categories can be difficult to understand, we will simplify the above procedure via the following trick:
\begin{lemma}\label{lem:colim-to-lim}
Let $\cC = \varinjlim \cC_i$ be a colimit of cocomplete dg categories $\cC_i$ along continuous functors $F_i$. Then $\cC$ is equivalent to the limit $\varprojlim \cC_i$ obtained by replacing $F_i$ by their left adjoints.
\end{lemma}
We can thus compute the colimit described above as
\begin{equation}\label{eqn:co-lim}
\varinjlim \mu\Sh_{\LL_i}^c(\LL_i) = (\varinjlim \mu\Sh_{\LL_i}(\LL_i))^c
= (\varprojlim \mu\Sh_{\LL_i}(\LL_i))^c.
\end{equation}
If the limit in \eqref{eqn:co-lim} is the dg category $A\dgmod$ of modules over a dg algebra $A$, then the compact objects $(A\dgmod)^c$ will be the dg category $A\dgperf$ of perfect $A$-modules. 

%
%

Our two remaining difficulties are, first, to compute the individual categories involved in this limit, and second, to keep track of the restriction maps among them, which may depend sensitively on how the polarization $\sigma$ varies as we move over the skeleton. Each of these difficulties will be addressed in Section \ref{subsec:perv-calc}; in the remainder of this subsection, we establish the auxiliary results we will need there. In each case, the crucial fact will be
that our Lagrangian skeleton $\LL$ is actually holomorphic Lagrangian, and the polarization $\sigma$ is actually a holomorphic polarization.


In fact, the Lagrangian skeleton $\LL$ is covered by local pieces $\LL_i$ which are conormals to unions of complex subvarieties in complex manifolds $X_i$.
Recall that in this case, the sheaf category $\Sh_{\LL_i}(X_i)$ has a t-structure whose heart $\Sh_{\LL_i}(X_i)^{\heartsuit}$ is the abelian category $\Perv_{\LL_i}(X_i)$ of perverse sheaves on $X_i$ with microsupport along $\LL_i.$ (Under the equivalence of $\Sh_{\Lambda}(X)$ with the Nadler-Zaslow infinitesimal Fukaya category \cite{NZ} of $T^*X,$ the abelian category $\Perv_{\LL_i}(X_i)$ is known by \cite{jin2015} to contain all the holomorphic Lagrangian branes in $T^*X_i$; analogous results for the usual Fukaya category under the Ganatra-Pardon-Shende equivalence can be found in \cite{cote2022perverse}.)
In good situations, this category remembers all the information of the category $\Sh_{\LL_i}(X_i)$, in the sense that $\Sh_{\LL_i}(X_i)$ may be recovered as the dg derived category of its heart.

Like the category of constructible sheaves, the category of perverse sheaves microlocalizes to a sheaf of categories (in this case a sheaf of {\em abelian} categories) $\muPerv,$ as constructed in \cite{Waschkies}. The perverse t-structure is easy to describe from a microlocal point of view.
\begin{proposition}[{\cite[Theorem 10.3.12]{KS94}}]\label{prop:pervt}
Let $X_i,\LL_i$ as above. An object $\cF$ in $\Sh_{\LL_i}(X_i)$ is contained in the perverse heart of this category if and only if the microstalk $\cF_{(x,\xi)}$ at every smooth point of $\LL_i$ is cohomologically concentrated in degree 0.
\end{proposition}




In fact, as has been shown in \cite{cote2022perverse}, this t-structure exists for the category $\mu\Sh_{\Lambda,\sigma}(\Lambda)$ of microlocal sheaves on a general holomorphically polarized weakly Weinstein manifold with holomorphic Lagrangian skeleton. Outside the holomorphic setting, this is not the case for the usual t-structure on constructible sheaves, since after microlocalization, the microlocal stalk functor is defined only up to a shift in homological degree: a choice of homotopy between a given polarization $\sigma$ and the cotangent fiber polarization $\sigma_{\mathsf{fib}}$ can have arbitary winding number $k\in \Z$ around the Maslov cycle in $\LGr,$ shifting homological degree by $k$. This does not occur in the case of holomorphic polarizations, since the inclusion $\LCGr\to \LGr$ does not wrap the Maslov cycle:

\begin{lemma}[{\cite[Theorem 1.2]{cote2022perverse}}]\label{lem:mustalk-welldef}
Let $\Lambda$ be the Lagrangian skeleton of a weakly Weinstein manifold (or sector) equipped with holomorphic stable normal polarization $\sigma.$ At each smooth point $x$ of $\Lambda,$ there is a microlocal stalk functor
\[
\mu_x: \mu\Sh_{\Lambda,\sigma}(\Lambda)\to \mathsf{Vect}_k,
\]
well-defined up to (noncanonical) isomorphism.
If $\sigma=\sigma_{\mathsf{fib}}$ is the cotangent fiber polarization of a cotangent bundle, then the functor $\mu_x$ is isomorphic to the usual microlocal stalk functor.
\end{lemma}

Although the microstalk functor just defined is not canonical, it is canonical on objects, allowing one to define a pair of full subcategories $\mu\Sh_{\Lambda,\sigma}(\Lambda)^{\leq 0},\mu\Sh_{\Lambda,\sigma}(\Lambda)^{\geq 0}\subset \mu\Sh_{\Lambda,\sigma}(\Lambda)$ on the objects whose microstalks are concentrated in nonpositive (resp. nonnegative) degree.
\begin{theorem}[{\cite[Theorem 1.3]{cote2022perverse}}]
    Let $\Lambda$ be a holomorphic Lagrangian skeleton in a weakly Weinstein manifold (or sector) equipped with holomorphic stable polarization $\sigma.$  
    Then the categories
   $\mu\Sh_{\Lambda,\sigma}(\Lambda)^{\leq 0},\mu\Sh_{\Lambda,\sigma}(\Lambda)^{\geq 0}$ defined above
   determine a t-structure on the microsheaf category $\mu\Sh_{\Lambda,\sigma}(\Lambda).$
\end{theorem}

\begin{remark}
    The definitions and results about microlocal (perverse) sheaves discussed in this section are all proven first in the setting of Legendrians in a contact manifold and afterwards extended to conic Lagrangians in a Liouville manifold $(W,\lambda).$ A priori, these invariants may be sensitive to the choice of Liouville form $\lambda,$ although this is not reflected in our notation. Nevertheless, as we will explain in Remark \ref{rem:independence-of-liouville}, in practice the invariants we study will actually be independent of this choice, thus justifying our notation.
\end{remark}


\subsection{Perverse sheaves on toric varieties}\label{subsec:perv-calc}
Let $\mathfrak{X}_\Delta$ be a $k$-dimensional projective toric variety with momentum polytope $\Delta.$ We denote by $\Lambda\subset T^*\mathfrak{X}_\Delta$ the union of conormals to toric strata. We will be interested in the abelian category $\Perv_\Lambda(\mathfrak{X}_\Delta)$ of perverse sheaves on $\mathfrak{X}_\Delta$ constructible with respect to the toric stratification.

The fundamental calculation in this subject is the determination of this category in the case when $\mathfrak{X}_\Delta = \C^n,$ first accomplished in \cite{galligo1985}. We begin by recalling the $n=1$ case.

\begin{example}\label{ex:pbasic}
Suppose $\mathfrak{X}_\Delta=\C.$
The category $\Perv(\C,0)$ of perverse sheaves on $\C$ with singular support inside the union of the zero section and the conormal to 0 has a classical description, given by the linear-algebraic data of a pair of vector spaces and maps
\[
\operatorname{var}:\Phi\leftrightarrows\Psi:\operatorname{can},
\]
such that the compositions
$1_\Psi+\operatorname{var}\circ\operatorname{can}$ and $1_\Phi+\operatorname{can}\circ\operatorname{var}$ are invertible. 

As described for instance in \cite[\S 4]{Verdier} or \cite{Beilinson87}, the vector space $\Psi$ and $\Phi$ are the nearby and vanishing cycles, respectively, of an object in $\Perv(\C,0),$ and the invertible maps described above are their monodromies. In terms of the conic Lagrangian $\Lambda\subset T^*\C$ which is the union of conormals to toric strata (in other words, the union of the zero section with the conormal to the origin), the vector spaces $\Psi$ and $\Phi$ are the respective microstalks at the two components of the smooth locus $\Lambda^\circ = \C^\times\sqcup\C^\times$,
on which a perverse sheaf restricts (or microlocalizes) to a local system, with monodromy given as above.
\end{example}

The calculation for general $n$ is essentially a product of this calculation, which we will rephrase in terms of the moment-map geometry of $T^*\C^n.$

\begin{definition}\label{defn:toric-arrangement}
    Let $\Delta$ be the moment polytope of a toric variety, and consider the hyperplane arrangement given by extending all facets of $\Delta$ to hyerplanes, restricted to a small neighborhood of $\Delta$ so that no two hyperplanes intersect outside of $\partial\Delta.$ We will denote this arrangement by $\mathcal{H}_\Delta.$
\end{definition}

For $\Delta$ the moment polytope of $\C^n,$ this is just the hyperplane arrangement in $\R^n$ given by coordinate hyperplanes. The chambers of this hyperplane arrangement represent the moment polytopes of conormals of toric strata in $\C^n$: the upper-right orthant is the moment polytope of $\C^n$ itself; the lower-left orthant is the moment polytope of the Fourier dual $\C^n$ (the cotangent fiber at $0$); and in general, an orthant meeting the upper-right orthant at a face $F,$ representing a stratum $S\subset \C^n$ is the moment polytope of $T^*_F\C^n.$

Thus, each chamber $\delta$ in this arrangement corresponds to a toric variety (in this case, $\C^n$) whose dense torus we will denote by $T_\delta,$. Each facet connecting two chambers $\delta,\delta'$ represents a subvariety where the closures of $T_\delta$ and $T_{\delta'}$ meet in a subvariety with codimension 1 in each closure.  Let $T_{\delta\delta'}$ be the dense torus in this closure; this is a single $G$-orbit, and each point in it is stabilized by a 1-dimensional subtorus.  There are two parameterizations of this subtorus, corresponding to the normal vectors pointing into $\delta$ and $\delta'$ respectively.  We let 
$\gamma_F$ be the loop in $\pi_1(T_\delta)$ given by applying the $\delta$-parameterization to the positive parameterization of the unit circle in $\C^{\times}$, and similarly for the $\delta'$-parameterization in $\pi_1(T_\delta').$ Note that if we identify $T_{\delta}\cong T_{\delta'}$ using the $G$-action, these classes will be inverse in the fundamental group.

We will now associate a quiver with relations to this hyperplane arrangment.
\begin{definition}\label{def:perv-tor}
Let $\mathcal{H}$ be a locally finite hyperplane arrangement in $\ft^*$ with every $H$ hyperplane defined by an equation $\gamma_H=n_H$ for $\gamma_H\in \ft_{\Z}, n_H\in \R$, and let $U$ be its set of chambers. We define a quiver $Q_{\mathcal{H}}$ as follows.
\begin{enumerate}
    \item For every chamber $\delta$ in $U$, we associate a vertex $v_\delta.$ (We will also write $V_\delta$ for the vector space placed at vertex $V_\delta$ in a representation.)
    \item For every facet $F$ separating two chambers $\delta,\delta',$ there are a pair of opposite arrows
    \[
    u_{\delta'\delta}:V_{\delta'} \rightleftarrows V_\delta:u_{\delta\delta'}.
    \]
    \item The vector space $V_\delta$ at the vertex $v_\delta$ has an action of the  group algebra $\C[\pi_1(T)]$.
\end{enumerate}
For every facet $F$ separating two chambers $\delta,\delta',$
we will write $M_{F}$ for the endomorphism of $V_\delta$ given by
\[
M_{F}:=u_{\delta'\delta}u_{\delta\delta'}+1_{V_\delta}.
\]

The relations we impose on $Q_\Delta$ are the following:
\begin{enumerate}
    \item For $F$ a facet of $\delta$ whose affine span is a hyperplane $H$, the endomorphism $M_{F}\in\End(V_{\delta})$ is identified with the loop $\gamma_{H}\in \C[\pi_1(T)]$.  In particular, $M_F$ is invertible.
    \item If $F$ is a codimension-2 face along which $\delta$ and $\delta''$ meet, and we denote by $\delta_1',\delta_2'$ the two chambers which are adjacent (by a facet) to both $\delta$ and $\delta'',$ then 
    the following pairs of length-two paths in $Q_\delta$ agree:
    \begin{enumerate}
        \item The two paths from $V_{\delta''}$ to $V_\delta,$
        \item the two paths from $V_{\delta}$ to $V_{\delta''}$
        \item the two paths from $V_{\delta_1'}$ to $V_{\delta_2'}.$
    \end{enumerate}
\end{enumerate}
We will write $A_{\mathcal{H}}$ for the path algebra of the quiver $Q_\Delta$ with relations as above.
\end{definition}
In the case where $\mathcal{H}$ is the facets of the moment polytope $\Delta$ of $\C,$ it is clear that the category of $A_\Delta=A_{\mathcal{H}}$-modules is just a reformulation of the perverse sheaf category described in Example \ref{ex:pbasic}. But the above description has been engineered in order to serve equally well to describe perverse sheaves on the product variety $\C^n.$
\begin{theorem}[\cite{galligo1985}]\label{thm:galligo-etal}
The category $\Perv_\Lambda(\C^n)$ of perverse sheaves on $\C^n$ constructible with respect to its toric stratification is equivalent to the category $A_\Delta\mmod$ of $A_\Delta$-modules.
\end{theorem}

By using descent for categories of perverse sheaves, local copies of the calculation in Theorem \ref{thm:galligo-etal} can be glued together to produce an analogous description for a general toric variety.
\begin{theorem}[\cite{Dupont}]\label{thm:dupont}
Let $\mathfrak{X}_\Delta$ be a smooth toric variety with moment polytope $\Delta.$ Then the category $\Perv_\Lambda(\mathfrak{X}_\Delta)$ of perverse sheaves on $\mathfrak{X}_\Delta$ constructible with respect to its toric stratification is equivalent to the category $A_\Delta\mmod$ of modules over the algebra $A_\Delta.$
\end{theorem}
By construction, the correspondence is just as described in Example \ref{ex:pbasic}: given a perverse sheaf $\cF$ on $\mathfrak{X}_\Delta,$ the corresponding quiver representation will have at vertex $V_\delta$ the microstalk of $\cF$ at a point of the open orbit of $\mathfrak{X}_\delta,$ which we think of as a conormal, living inside $T^*\mathfrak{X}_\Delta,$ of a stratum of $\mathfrak{X}_\Delta.$

%

\begin{example}\label{ex:projectiveplane}
Let $\mathfrak{X}_{\Delta} = \mathbb{P}^2$. Its momentum polytope is the compact chamber $\delta_c$ of the arrangement depicted below. There are three chambers $\delta_1, \delta_2, \delta_3$ corresponding to the conormal bundles to the toric divisors in $\mathbb{P}^2$, and three chambers $\delta_{12}, \delta_{13}, \delta_{23}$ corresponding to the conormal fibers to the toric fixed points.    

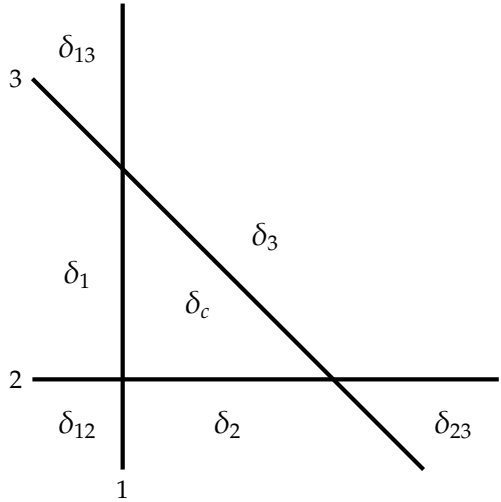
\begin{equation*}\label{eq:P2-chambers}
  \begin{tikzpicture}[very thick,scale=2,baseline]
    \draw [ultra thick] (-.2,-.8) -- node[below, at start,scale=.8]{$1$} (-.2,2.3);
    (1.8,2.5);
    \draw [ultra thick] (-.8,-.2) -- node[left, at start,scale=.8]{$2$} (2.3,-.2);
  \draw [ultra thick] (-.8,1.8) -- node[left, at start,scale=.8]{$3$} (1.8,-.8);
 \node at (0.75,0.75) {$\delta_3$};
\node at (0.5,-0.5) {$\delta_2$};
\node at (-0.5,0.5) {$\delta_1$};
\node at (-0.5,-0.5) {$\delta_{12}$};
\node at (2,-0.5) {$\delta_{23}$};
\node at (-0.5,2) {$\delta_{13}$};
\node at (0.3,0.3) {$\delta_{c}$};

  \end{tikzpicture}
\end{equation*}

The corresponding quiver has one vertex for each chamber, and a pair of arrows for each plane separating two adjacent chambers:
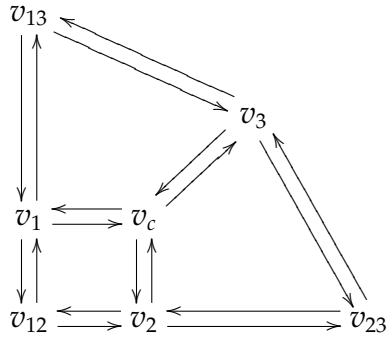
\[
\xymatrix{
v_{13}\ar@<-.5ex>[dd]\ar@<-.5ex>[rrd]&&&\\
&&v_3\ar@<-.5ex>[llu]\ar@<-.5ex>[ddr]\ar@<-.5ex>[ld]&\\
v_1\ar@<-.5ex>[uu]\ar@<-.5ex>[d]\ar@<-.5ex>[r]&v_c\ar@<-.5ex>[l]\ar@<-.5ex>[ur]\ar@<-.5ex>[d]&&\\
v_{12}\ar@<-.5ex>[u]\ar@<-.5ex>[r]&v_2\ar@<-.5ex>[u]\ar@<-.5ex>[rr]\ar@<-.5ex>[l]&&v_{23}\ar@<-.5ex>[ll]\ar@<-.5ex>[uul]
}
\]

\end{example}

Moreover, perverse sheaves do encapsulate the whole Fukaya category in this case, in the sense that the dg category $\Sh_\Lambda(\mathfrak{X}_\Delta)$ is equal to the dg derived category of $\Perv_\Lambda(\mathfrak{X}_\Delta):$

\begin{proposition}\label{prop:pgen}
$\Sh_\Lambda(\mathfrak{X}_\Delta)$ is equivalent to the dg category $A_\Delta\dgmod$ of modules over the algebra $A_\Delta.$
\end{proposition}
\begin{proof}
The dg category $A_\Delta\dgmod$ is the derived category of $\Perv_\Lambda(\mathfrak{X}_\Delta).$
As described in Proposition~\ref{prop:gens}, the category $\Sh_\Lambda(\mathfrak{X}_\Delta)$ is generated by the objects $\cP_i$ which corepresent microstalk functors at smooth points of $\Lambda$, so we need to prove that all of the objects $\cP_i$ are contained in the heart of the perverse t-structure on $\Sh_\Lambda(\mathfrak{X}_\Delta)$ -- i.e., that the objects corepresenting microstalk functors in $\Perv_\Lambda(\mathfrak{X}_\Delta)$ also corepresent microstalk functors in $\Sh_\Lambda(\mathfrak{X}_\Delta).$

First, consider the case where $\cP_i$ corresponds to the open torus orbit in $\mathfrak{X}_\Delta$.  Consider the free (infinite-rank) local system $\mathcal{L}$ on the open torus orbit $U_0\cong G$ with fiber $\C[\pi_1(G)]$ obtained by pushing forward the constant sheaf on the universal cover of $G$.  If $\iota \colon U_0\to \mathfrak{X}_\Delta$ is the inclusion, then we claim $\iota_!\mathcal{L}$ corepresents the functor of taking stalk at any point of $U_0$ (i.e., at any smooth point on the zero section $\mathfrak{X}_\Delta\subset \Lambda).$
Indeed, for any $\mathcal{F}\in \Sh_\Lambda(\mathfrak{X}_\Delta),$ we have
\begin{equation}\label{eq:zerosec-proj}
\Hom_{\Sh_\Lambda(\mathfrak{X}_\Delta)}(\iota_!\mathcal{L},\mathcal{F}) \cong
\Hom_{\Loc(U_0)}(\mathcal{L},\iota^!\mathcal{F}) \cong
\Hom_{\Loc(U_0)}(\mathcal{L},\iota^*\mathcal{F}),
\end{equation}
where the first equivalence is by adjunction and the second follows from the fact that $\iota$ is an open immersion. Since $\mathcal{L}$ is the universal local system on the torus $U_0,$ we conclude that \eqref{eq:zerosec-proj} is indeed equivalent to the (micro)stalk of $\mathcal{F}$ at a smooth point of the zero section.

Now consider the case of a general component of the Lagrangian $\Lambda$. Recall that the components of the smooth locus of $\Lambda$ are in bijection with toric subvarieties in $\mathfrak{X}_\Delta$. 
Given such a toric subvariety, choose a $T$-fixed point $x$ in its closure, and let $U\subset \mathfrak{X}_\Delta$ be the open subset given by all torus orbits which contain $x$ in their closure.  This subset may be identified with an orbifold quotient of $\C^d,$ with the various toric subvarieties corresponding to the coordinate subspaces in $\C^d$, and the components of $\Lambda^\circ$ given by the conormals to these.  In this context, we have already constructed the perverse sheaf corepresenting the functor of microstalk at a smooth point in the zero section, and the sheaves corepresenting microstalks along generic covectors to other toric subvarieties of $\C^d$ may be constructed from this one by Fourier transform.

To extend these perverse sheaves from $U$ to all of $\mathfrak{X}_\Delta,$ we push forward by $\iota_!$ under the inclusion $\iota \colon U\to \mathfrak{X}_\Delta$ as above.  Since Fourier transform and $\iota_!$ for an open inclusion are exact in the  perverse $t$-structure, the results of this pushforward remain perverse. Moreover, they continue to corepresent microstalk functors on $\Sh_\Lambda(\mathfrak{X}_\Delta)$: this may be deduced from the fact that for $\mathcal{F}\in \Sh_\Lambda(\mathfrak{X}_\Delta),$ we have
\[
\Hom_{\Sh_{\Lambda}(\mathfrak{X}_\Delta)}(\iota_!\mathcal{P},\mathcal{F})\cong \Hom_{\Sh_{\Lambda}(U)}(\mathcal{P},\iota^*\mathcal{F});
\]
by construction, the object $\mathcal{P}$ measures the desired microstalk in $\Sh_\Lambda(U),$ and the pullback $\iota^*$ does not affect the microstalk of $\mathcal{F}$ there.

In particular, this implies that the Hom spaces $\Hom_{\Sh_\Lambda(\mathfrak{X}_\Delta)}(\cP_i,\cP_j)$ are concentrated in degree 0, since this Hom space can be interpreted as the microstalk of $P_j$ at vertex $v_i.$
\end{proof}

Note that Theorem \ref{thm:dupont} remains true if we allow $\mathfrak{X}_\Delta$ to be an orbifold toric variety which is smooth when considered as a DM stack; we provide here a couple of examples to illustrate this point.

\begin{example}\label{ex:orbifold}
Consider the case $\mathfrak{X}_\Delta=\C/\{\pm 1\}$; even though the underlying variety is still isomorphic to $\C,$ the space $\mathfrak{X}$ carries a nontrivial orbifold structure.  We can find this as a Lagrangian inside the additive hypertoric orbifold $\C^2/\{\pm 1\}$, which one can interpret as the cotangent bundle of $\mathfrak{X}_\Delta$.  
 The category of perverse sheaves on $\mathfrak{X}_\Delta=\C/\{\pm 1\}$ is the category of $\{\pm 1\}$-equivariant perverse sheaves on $\C$.  As before, we can take $\Phi$ and $\Psi$ with the maps $\operatorname{var}$ and $\operatorname{can}$, but the operation of parallel transport from $1$ to $-1$ in $\C^{\times}$ followed by moving back to $1$ using the $\{\pm 1\}$-action gives an endomorphism of $\Psi$ which is the square root of the monodromy.  
\end{example}
In case the example $\C/\{\pm 1\}$ is confusing, the following example offers a possibly clearer illustration:
\begin{example}Let $\mathfrak{X}_\Delta =\C^2/(\Z/2).$ 
If we write 
\[
\C[\mathfrak{X}_\Delta]=\C[x,y]^{\Z/2}=\C[x^2,xy,y^2]=\C[a,b,c]/(b^2-ac),
\]
then we see that on the open stratum of $\mathfrak{X}_\Delta$, 
the torus acting is 
\[
T_\Delta:=\Spec\C[a^\pm,b^\pm,c^\pm]/(b^2-ac).
\]
In $Q_\Delta,$ the two pairs of maps at the vertex $V_\Delta$ corresponding to this stratum have compositions equal to $1-a$ and $1-c$, respectively. However, in contrast to the normal crossings case, $1+u_1v_1$ and $1+u_2v_2$ do not give all the invertible self-loops of $V_\Delta,$ but rather this vertex is equipped with an extra automorphism $b$ which is equal to a square root of $ac.$
\end{example}

By \eqref{eq:ks-is-ns-musheaves}, the above results have given a computation of the sheaf of categories $\mu\Sh_{\Lambda,\sigma_{\mathsf{fib}}}$ in $T^*\mathfrak{X}_\Delta.$ Such Lagrangians $\Lambda$ will form a cover of the skeleton $\LL$ which will study below, and hence we will be interested in the restriction to $\Lambda$ of $\mu\Sh_{\LL,\sigma}$ for a more general holomorphic polarization $\sigma.$ We now investigate establish invariance of the sheaf of categories $\mu\Sh$ under certain changes of polarization.

\begin{remark}\label{rem:independence-of-liouville}
    Hidden by our notation is the fact that for a conic Lagrangian $\Lambda$ in a Weinstein manifold $(W,\lambda)$, the sheaf of categories $\mu\Sh_{\Lambda,\sigma}$ depends not only on $\sigma$ but also on the Liouville form $\lambda$, since it is defined first for contact manifolds, and then extended to Liouville manifolds by embedding $W = W\times \{0\}\hookrightarrow \R\times W$ in the contact manifold $(\R\times W, dt+\lambda).$ However, in fact, the change in Liouville form will not affect the above calculations: if $\lambda,\lambda'$ are two Liouville forms on $T^*\C^n$ which vanish on the conic Lagrangian $\Lambda$, then, writing $\lambda-\lambda'=df,$ the family $(t,w)\mapsto (t+sf, w),$ for $0\leq s\leq 1,$ gives a contact isotopy, preserving the Legendrian $\{0\}\times \Lambda,$ on $\R\times W$ between the contact forms $dt+\lambda$ and $dt+\lambda',$ showing invariance of $\mu\Sh_{\Lambda,\sigma}$ under the change in Liouville form.
\end{remark}

Let $\mathfrak{X}_\Delta$ be a smooth toric variety and let $\Lambda\subset T^*\mathfrak{X}_\Delta$ be the union of conormals to toric strata.  Consider the cotangent fiber polarization 
$\sigma_{\mathsf{fib}}$ of $T^*\mathfrak{X}_\Delta.$ Now fix a stable holomorphic polarization
$\sigma$ 
with the property that the orientation class $w_2(\sigma)\in H^2(T^*\mathfrak{X}_\Delta, \Z/2)\simeq H^2(\mathfrak{X}_\Delta,\Z/2)$ discussed in Section \ref{subsec:polarizations} agrees with $w_2(\mathfrak{X}_\Delta)$.
\begin{proposition}\label{prop:canonical}
There is an equivalence of cosheaves of dg categories
\begin{equation}\label{eq:toric-pol-indep}
\mu\Sh_{\Lambda,\sigma}^c\cong \mu\Sh_{\Lambda,\sigma_{\mathsf{fib}}}^c
\end{equation}
respecting the microlocal perverse t-structures.
The space of such equivalences \eqref{eq:toric-pol-indep} is a torsor for the space of maps to $B\Z/2.$
\end{proposition}
\begin{proof}

As is shown in  \cite[Proposition 11.19]{NS20} (cf. \cite[\S 5.3]{GPS3} for relation to the traditional grading and orientation data), 
the data necessary for defining a microlocal-sheaf category (with coefficients in $\Z$-modules) within the symplectic manifold $T^*\frX_\Delta$ is a nulhomotopy of  the composition 
\begin{equation}\label{eq:maslov-map}
\xymatrix{
T^*\frX_\Delta\ar[r]& B(U/O)\simeq B^2(\Z\times BO)\ar[r]&B^2\Pic\Z \simeq B^2\Z\times B^3(\Z/2),
}
\end{equation}
where the first map in \eqref{eq:maslov-map} classifies the stable Lagrangian Grassmannian bundle of $T^*\frX_\Delta,$ and \cite[Theorem 1.2]{Jin-jhom} (following an earlier calculation for $\Z$-module coefficients in \cite[Theorem 11.5]{Guillermou})  identifies the second map with the double delooping of the $J$-homomorphism, followed by the change-of-coefficients map $\Pic(\mathbb{S})\to\Pic(\Z).$

The microsheaf categories appearing in \eqref{eq:toric-pol-indep} are defined by the respective trivializations of \eqref{eq:maslov-map} supplied by the polarizations $\sigma,\sigma_{\mathsf{fib}}$ (or rather their underlying real polarizations)
which by definition are nulhomotopies of the first map $T^*\frX_\Delta\to B(U/O)$ in \eqref{eq:maslov-map}.

The difference between these two nulhomotopies is therefore a map
\[
T^*\mathfrak{X}_\Delta\to \Omega B(U/O) = B(U/O),
\]
which induces (by composition with the looped right-hand map of \eqref{eq:maslov-map}) a map 
\begin{equation}\label{eq:measuring-difference}
   T^*\mathfrak{X}_\Delta\to B\Pic(\Z) \simeq B\Z\times B^2(\Z/2),  
\end{equation}
measuring the difference between the respective sheaves of categories defined by these polarizations.


We claim that the map \eqref{eq:measuring-difference} trivial. Indeed, the two nulhomotopies of the first factor in \eqref{eq:maslov-map} are the same since both (coming from {\em holomorphic} polarizations) are induced from the canonical nulhomotopy of the composition $Sp/U\to U/O\to B\Z,$ which shows that the first factor in \eqref{eq:measuring-difference} is trivial. On the other hand, by definition, the second factor measures the difference between $w_2(\sigma)$ and $w_2(\sigma_{\mathsf{fib}}) = w_2(\mathfrak{X}_\Delta),$ and our assumption on $\sigma$ guarantees that these are equivalent.
The choice of such an equivalence is a torsor for the space of maps from the variety $T^*\mathfrak{X}$ to $\Omega B^2(\Z/2) = B\Z/2.$
\end{proof}

\subsection{The global calculation}\label{sec:global}
We are now ready to glue together the local perverse-sheaf computations described above. We begin by describing what the result of this gluing will be.

Recall that the hyperplane arrangement $\Bper_\parone$ defined in \ref{def:per-arrangement} expresses $\R^n$ as an infinite union of $n$-dimensional polytopes.

\begin{definition}\label{def:abstract-quiver}
Let $\cQ_\parone$ be the quiver with relations associated to the hyperplane arrangement $\Bper_\parone$ by Definition~\ref{def:perv-tor}. 
\end{definition}

The category $\cQ_\parone\mmod$ can be reconstructed from subquivers associated to pieces of the hyperplane arrangement $\Bper_\parone.$ For each vertex $v\in \Bper_\parone,$ write $\cQ^v_\parone$ for the full subquiver of $\cQ_\parone$ (with the same relations) obtained by forgetting the vertices of $\cQ_\parone$ associated to any chambers of $\Bper_\parone$ not containing $v$ as a vertex. Similarly, for $F$ a higher-dimensional face (or chamber) in $\Bper_\parone,$ we write $\cQ_\parone^F$ for the subquiver whose vertices correspond to the chambers in $\Bper_\parone$ containing $F$.

Observe that for an inclusion $F\subset F'$ in $\Bper_\parone,$ we have an inclusion of quivers-with-relations $\cQ_\parone^{F'} \to \cQ_\parone^F$ and hence a restriction map on module categories $\cQ_\parone^F\mmod\to\cQ_\parone^{F'}\mmod,$ and that these functors commute for chains of inclusions.

\begin{definition}Let $\bP$ denote the  face poset of this polytopal decomposition of $\R^n$ given by $\Bper_\parone$. 
\end{definition}

\begin{lemma}\label{lem:quiver-cover}
There is a canonical equivalence
\begin{equation}\label{eq:quiver-map}
    \cQ_\parone\dgmod \xrightarrow{\sim} \varprojlim_{F\in \bP}\cQ^F_\parone\dgmod
\end{equation}
between the dg category of $\cQ_\parone$-modules and the homotopy limit of the dg categories of $\cQ^F_\parone$ modules.
\end{lemma}
We give a proof of this result in Appendix \ref{appendix:holim-proof}, written jointly with Laurent C\^ot\'e and Justin Hilburn.
The idea of the proof is relatively simple:  ``a representation of $\cQ_\parone$ is the same as a representation of its subquivers $\cQ^F_\parone,$ together with coherent equivalences among the restrictions to shared subquivers.'' However, establishing an equivalence of dg categories requires checking that the spaces of such equivalence data are contractible. 
This is obvious if the hyperplane arrangement $\Bper_\parone$ is 1-dimensional; it is still true in higher dimensions, but the bookkeeping is more tedious, which is why we have banished it to an appendix.


We will now produce this same limit of categories in a description of microlocal sheaves on our Lagrangian skeleton $\LL.$ 
As above, to simplify notation we will work on the universal cover $\widetilde{\LL},$ a holomorphic Lagrangian which is the union of toric varieties whose moment polytopes are the chambers in $\Bper_\parone.$

\begin{proposition}\label{prop:fuk-is-lim}
Let $\widetilde{\sigma}$ be the holomorphic polarization on $\widetilde{\mht}_{(\parone,1)}$ pulled back from the polarization $\sigma$ constructed on $\mht_{(\parone,1)}$ in Corollary \ref{cor:hol-pol}.
Then
the global microlocal sheaf category $\mu\Sh_{\widetilde{\LL},\widetilde{\sigma}}(\widetilde{\LL})$ 
is equivalent to the limit $\varprojlim_{F\in\bP}\cQ_\parone^F\dgmod.$
\end{proposition} 
\begin{proof}
  The Lagrangian $\tilLL$ has an open cover indexed by the face poset $\bP$ of the hyperplane arrangement $\Bper_\parone$: to a face $F$ we write $\tilLL_F\subset \tilLL$ for the open subset of $\tilLL$ given by the union of toric orbits in $\tilLL$ whose moment polytopes contain $F$ in their closure. The basic open sets in this cover are the opens $\tilLL_v$ associated to the vertices $v$, and their intersections are given by
  \[
  \tilLL_{v_1}\cap\cdots\cap \tilLL_{v_k} = \tilLL_{F}
  \]
  if $F$ is the minimal face containing all the $v_i$ (with empty intersection if no such face exists).
  
  The category $\mu\Sh_{\widetilde{\LL},\widetilde{\sigma}}(\LL)$ is the global sections of a sheaf of categories $\mu\Sh_{\widetilde{\LL},\sigma},$ so we may compute it as the limit of its sections over the open cover $\{\tilLL_F\}_{F\in\bP}$ which we just described.
 
 Now observe that for any vertex for a vertex $v\in \bP,$ the Lagrangian $\tilLL_v$ is equivalent to the union of conormals to toric strata in $\C^n,$ and for any face $F$ containing $v$, the Lagrangian $\tilLL_F$ is the sub-Lagrangian obtained by deleting some coordinate axes in $\C^n$ (and their conormals).
 

%
For a chamber $C$ with vertices $v_1,\ldots, v_k,$ write
\[
\tilLL_{\overline{C}}:=\tilLL_{v_1}\cup \cdots \cup \tilLL_{v_k}
\]
for the natural open neighborhood of the closure of $\tilLL_C.$ This open subset is naturally identified with the union of conormals to toric strata in $T^*\frX_{\overline{C}},$ where $\frX_{\overline{C}}$ is the toric component of the skeleton $\LL$ corresponding to chamber $C$, and
Proposition \ref{prop:canonical} gives an equivalence
between $\mu\Sh_{\tilLL,\widetilde{\sigma}}|_{\tilLL_{\overline{C}}}$ and microsheaves computed in the cotangent fiber polarization of $T^*\mathfrak{X}_{\overline{C}}.$

For each face $F$ of $\overline{C},$ we thus have an equivalence
\[
\mu\Sh_{\tilLL,\widetilde{\sigma}}(\tilLL_F)\simeq \cQ_\parone^F\dgmod,
\]
such that these equivalences coherently intertwine the restriction diagrams
\begin{equation}\label{eq:vertex-to-chamber}
\mu\Sh_{\tilLL,\widetilde{\sigma}}(\tilLL_{F})\to \mu\Sh_{\tilLL,\widetilde{\sigma}}(\tilLL_{F'}),
\end{equation}
for subfaces $F'\subset\overline{F},$ with the maps
\begin{equation}\label{eq:mu-restr}
\cQ_\parone^F\dgmod\to \cQ_\parone^{F'}\dgmod    
\end{equation}
given by forgetting the vertices of $\cQ_\parone^F$ corresponding to chambers whose closures do not contain the face $F'$.



Since the equivalence provided by Proposition \ref{prop:canonical} is canonical up to a map to $B\Z/2,$ and the open subsets $\tilLL_{\overline{C}}$ and their intersections are simply-connected, the above identifications are compatible globally over the skeleton.
We conclude that the $\bP$-diagram of categories $\mu\Sh_{\tilLL,\widetilde{\sigma}}(\tilLL_F)$ and restriction maps is equivalent to the $\bP$-indexed diagram of categories $\cQ_\parone^F$ and ``forget vertex'' maps, from which the result follows.
\end{proof}

Combining the equivalences of Lemma \ref{lem:quiver-cover} and Proposition \ref{prop:fuk-is-lim}, we conclude:
\begin{corollary}
    There is an equivalence of categories
    \begin{equation}\label{eq:equivalence-for-periodic-musheaves}
    \mu\Sh_{\tilLL,\widetilde{\sigma}}(\tilLL) \cong \cQ_\parone\dgmod.
    \end{equation}
\end{corollary}

We have been dealing so far with the universal cover $\widetilde{\mht}_{(\parone,\tone)}$,  but it is easy to deduce from the above result a similar one about $\mhta$ itself.
Let $\overline{\cQ}_\parone$ denote the quiver with relations associated to the toroidal hyperplane arrangement obtained from $\Bper_\parone$ by quotienting by the translational $\mathfrak{g}^\vee_\Z=\pi_1(\mht)$ symmetry.

\begin{corollary}\label{cor:musheaves-on-L}
There is an equivalence of categories
\[
\mu\Sh_{\LL,\sigma}(\LL)\cong \overline{\cQ}_\parone\dgmod.
\]
\end{corollary}

\begin{proof}
    As the polarization $\widetilde{\sigma}$ on $\tilLL$ is pulled back from $\sigma$ along the $\mathfrak{g}^\vee_\Z$-cover $\tilLL\to \LL,$ the category $\mu\Sh_{\tilLL,\widetilde{\sigma}}(\tilLL)$ has a $\mathfrak{g}^\vee_\Z$ action, and $\mu\Sh_{\LL,\sigma}(\LL)$ is the invariant category for this action. Under the equivalence \eqref{eq:equivalence-for-periodic-musheaves}, this $\mathfrak{g}^\vee_\Z$ action corresponds to the action on $\cQ_\parone\dgmod$ induced by $\mathfrak{g}^\vee_\Z$ actions on the quiver $\cQ_\parone.$ This is a free action, with invariant category $(\cQ_\parone\dgmod)^{\mathfrak{g}^\vee_\Z}\cong \overline{\cQ}_\parone\dgmod.$
\end{proof}

\begin{corollary}\label{cor:Fuk-Q}
There is an equivalence of dg categories
    \[
    \Fuk(\mht_{(\parone,\tone)}) \cong \overline{\cQ}_\parone\dgperf.
    \]
\end{corollary}
\begin{proof}
In Proposition \ref{prop:toricskel}, we realized $\LL$ as a Lagrangian skeleton for $\mht_{(\parone,1)}.$ From Theorem \ref{thm:gps-main} and Lemma \ref{lem:colim-to-lim}, we conclude that the Fukaya category is equivalent to compact objects in the category of microlocal sheaves on the Lagrangian skeleton $\LL$ (where both Fukaya category and microsheaves are understood with respect to polarization data $\sigma$). The result now follows from Corollary \ref{cor:musheaves-on-L} by passing to compact objects.
\end{proof}

This is one point where we can make a connection to the constructions of the first part of this paper \cite{McBW}.  We constructed an algebra $H^{\parone}_\Z$ in \cite[\S 3.9]{McBW}, which is a quotient of the path algebra of the same quiver $\overline{\cQ}_\parone$, but with the Laurent polynomial ring $\C[\pi_1(G)]$  replaced by the polynomial ring $\operatorname{Sym}^{\bullet}(\mathfrak{g})$ and with different relations \cite[(3.10a–3.10c)]{McBW}.  While this algebra is different from $A_{\parone}$, we can compare our results by showing they coincide after a completion.  

\begin{definition}\label{def:hats}
    Let $\widehat{A}_{\parone}$ be the completion of this  algebra by the 2-sided ideal generated by $g-1$ for all $g\in \C[\pi_1(T)]$.  Let $\widehat{H}^{\parone}_\Z$ be the completion of this algebra by the 2-sided ideal generated by $\ft$.
\end{definition}

\begin{proposition}\label{prop:completed isomorphism}
    The algebras $\widehat{A}_{\parone}$ and $\widehat{H}^{\parone}_\C$ are isomorphic.
\end{proposition}
\begin{proof}
    First, note that we have an isomorphism $\varphi$ between the completions $\operatorname{Sym}^{\bullet}(\mathfrak{g})$ and $\widehat{\C[\pi_1(G)]}$.  The former is generated by elements $\mathsf{s}_i$ which are the image of the usual basis in the Lie algebra $\mathfrak{d}$ of diagonal matrices, and latter by loops $\gamma_i$  which are images of coordinate subtori of dimension 1.  This isomorphism sends: \[\varphi(\mathsf{s}_i)= \log \gamma_i=(\gamma_i-1)-\frac{(\gamma_i-1)^2}{2}+\]
    using the Taylor expansion of $\log(x)$ at $x=1$.  Indeed, this matches the linear relations in $\mathfrak{g}$ as a quotient of $\mathfrak{D}$ to the multiplicative relations of $\pi_1(G)$ as a quotient of $\pi_1(D)$.  Note that $\frac{\log \gamma_i}{\gamma_i-1}$ is an invertible element of the completion, given by the Taylor expansion of the function $\frac{\log x}{x-1}=1-\frac{x}{2}+\frac{x^2}{3!}-\cdots $ at $x=1$, or more generally any real power $\left(\frac{\log \gamma_i}{\gamma_i-1}\right)^p$ for $p\in \R$, using the binomial expansion.  

    Thus, now we need to match the length 1 paths $c_{\delta,\delta'}$ in $\widehat{H}^{\parone}_\C$ to the paths $u_{\delta,\delta'}$ in $\widehat{A}_{\parone}$.  Compare the relation $c_{\delta,\delta'}c_{\delta',\delta}=\mathsf{s}_i$ from \cite[(3.10a)]{McBW} with the relation $u_{\delta,\delta'}u_{\delta',\delta}=\gamma_i-1$ of Definition \ref{def:perv-tor}(1). These don't match perfectly, but this is easy to fix. We let
    \[\varphi(c_{\delta,\delta'})=u_{\delta,\delta'}\left(\frac{\log \gamma_i}{\gamma_i-1}\right)^{1/2}.\] The relations are compatible, since 
    \[\varphi(c_{\delta,\delta'})\varphi(c_{\delta',\delta})=u_{\delta,\delta'}u_{\delta',\delta}\frac{\log \gamma_i}{\gamma_i-1}=\log \gamma_i=\varphi(\mathsf{s}_i)=\varphi(c_{\delta,\delta'}c_{\delta',\delta})\]
   The relations \cite[(3.10b--c)]{McBW} exactly match Definition \ref{def:perv-tor}(2).

   Thus $\varphi$ defines a homomorphism, and we can easily write down its inverse:
   \[\varphi^{-1}(\gamma_i)=\exp(\mathsf{s}_i)\qquad \varphi^{-1}(u_{\delta,\delta'})=c_{\delta,\delta'}\left(\frac{\exp(\mathsf{s}_i)-1}{\mathsf{s}_i}\right)^{1/2}.\qedhere\]
\end{proof}


\begin{example}
We spell out this construction when the underlying sequence of tori is $1 \to \C^{\times} \to \C^{\times}$. Thus $\mht_{(\parone,\tone)} = \mht_{(\tone,\partwo)} = (T^*\C)^{\circ}$, whereas the Dolbeault hypertoric variety is the Tate curve. (Since $T$ is trivial, $\parone=\partwo=\tone.$) As previously explained, $\tilLL$ is an infinite chain of rational curves, with $0$ of one link intersecting nodally with $\infty$ of the next. Thus $\cQ$ has vertices indexed by $n \in \Z$, each carrying a vector space $V_n$ on which the group algebra $\C[\pi_1(\C^{\times})] = \C[\Z]$ acts with generator $\gamma_n$. Between neighboring vertices, we have a pair of arrows  
$u_{n, n+1}:V_{n} \rightleftarrows V_{n+1}:v_{n,n+1}$, and an equality $1 + v_{n,n+1}u_{n,n+1} = \gamma_n$. 
This is equivalent to the category of modules over the multiplicative preprojective algebra of the infinite quiver 
\[
\ldots \to \bullet \to \bullet\to \bullet\to \cdots.
\]
Any such module may be viewed as the global sections of a $\C^{\times}$-equivariant coherent sheaf on $(T^*\C)^{\circ}$.

The simple modules $\cS_{n}$ over $\cQ$ are defined by a copy of $\C$ placed at a single vertex $n$, with all maps set to zero. These correspond to $\C^{\times}$-equivariant skyscraper sheaves at $x=y=0$. 

On the other hand, the projective object $\cP_n$ has $V_i = \C[\gamma_i]$. The maps $v_{i,i+1}$, $i < n$ and $u_{i,i+1}, i \geq n$ are given by the natural isomorphism taking $\gamma_i$ to $\gamma_{i+1}$. The maps in the reverse direction are fixed by the quiver relations. $\cP_n$ corresponds on the B-side to a line bundle whose global sections are the free graded $\C[x,y]\left[\frac{1}{1+yx}\right]$-module generated by $x^n$ for $n>0$ or $y^n$ for $n<0$.
\end{example}
\begin{example}
Now consider the sequence  $\C^{\times} \to (\C^{\times})^2 \to \C^{\times}$, where the first map is the diagonal embedding. In this case the $\C^{\times}$-valued moment map on $\mht_{\tone,\partwo)}$ expresses it as a $\C^{\times}$-fibration over $\C^{\times} \setminus 1$, whereas the fiber over $1$ is the ``TIE fighter'' given by a $\mathbb{P}^1$ nodally intersecting two copies of $\mathbb{A}^1$ at $0$ and $\infty$. On the other hand, $\mht_{(\parone,\tone)}$ is affine; its moment fibration has two singular fibers, each given by a union of two nodally intersecting copies of $\mathbb{A}^1$. The unwrapped skeleton $\tilLL$ is as before, but the lattice action on it now shifts the chain by two links rather than one, so that $\LL$ is a copy of two spheres meeting each other transversally at two points.

Repeating the calculation as above, we find that the category of microlocal sheaves along $\tilLL$ is the category of modules over the multiplicative preprojective algebra associated to the $\hat{A}_1$ quiver. Recall that the McKay correspondence of \cite{KV00} identifies modules over the additive $\hat{A}_1$ preprojective algebra with the category of coherent sheaves on the stack $\C^2/\Z/2$ (which we should think of as a noncommutative resolution of its singular coarse moduli space); similar arguments identify modules over the multiplicative preprojective algebra with coherent sheaves on the stack/noncommutative resolution $\TAo/\Z/2,$ which is $\mht_{(\tone,\partwo)}.$

More details of the calculation of this multiplicative preprojective algebra in the setting of Fukaya categories and microlocal perverse sheaves can be found in \cite{etgu2017} and \cite{BK16}, respectively.
 \end{example}

\section{Tilting bundles and coherent sheaves}\hfill\label{sec:Bside}

We will now calculate the B-model category associated to the mirror of $\mht_{(\parone,\tone)}$: this is the dg category $\Cohdg(\mht_{(\tone,\parone)})$ of coherent sheaves on the multiplicative hypertoric variety $\mht_{(\tone,\parone)}.$  As discussed before, we factor $\parone\in T^\vee$ as a product   $\Bfield\cdot \exp(\GIT)$. $\GIT$ will play the r\^ole of  GIT parameter, whereas $\Bfield$ will play the r\^ole of B-field, indexing a noncommutative resolution -- or, in the case where $\GIT$ is generic, indexing an Azumaya algebra on the resolution specified by $\GIT.$

\subsection{Noncommutative resolutions and mirror symmetry}\label{subsec:bfield}
Recall the notion of a {\bf noncommutative crepant resolution} (NCCR) of an affine Gorenstein variety $X=\Spec R$, originally defined in \cite{vdB04}: this is an algebra $A=\End_R(M),$ for some reflexive $R$-module $M$, such that $A$ is a Cohen-Macaulay R-module and the global dimension of $A$ is equal to $\dim X.$  This notion generalizes to a non-affine scheme in an obvious way: we consider a coherent sheaf of algebras $\mathscr{A}$ with module $\mathscr{M}$ such that the restriction to any affine open set is a NCCR.

In this paper, we will explicitly construct a NCCR for each generic choice of B-field $\Bfield$.  For simplicity, we will do this first in the case where the parameter is of the form $(\tone,\Bfield)$.  In this case, the underlying variety $\mht_{(\tone,\tone)}$ is affine and highly singular. 

Recall that for a choice of  $\logB\in \ft^\vee_\R$ (from which we can produce an element $\Bfield=\exp(\logB)\in T^\vee_\R),$ we have already described a quiver with relations $\overline{\cQ}_\Bfield$.  This has only finitely many nodes, so we can think of its path algebra $A_\Bfield=A(\overline{\cQ}_\Bfield)$ as an algebra with unit given by the sum of idempotents $e_p$ for the different chambers $p$ of top dimension in the quotient arrangement $\Btor$.
\begin{theorem}\hfill\label{th:nc-res}
\begin{enumerate}
    \item We have an isomorphism of algebras $\C[\mht_{(\tone,\tone)}]\cong e_pA_\Bfield e_p$ for any chamber $p$.
    \item The algebra $A$ acting on the module $M=A_\Bfield e_p$ is a noncommutative crepant resolution of $\mht_{(\tone,\tone)}$.  
\end{enumerate}
\end{theorem}
This theorem is effectively the main theorem of our paper, since it identifies the Fukaya category of $\mht_{(\beta,\tone)}$ with the category of ``coherent sheaves'' on this noncommutative resolution. However, we are not yet ready to prove it.  In fact, rather than showing directly that this algebra $A$ has the desired properties, we will show that it is derived equivalent to a usual commutative resolution of singularities obtained by varying the parameter $\GIT$. 
 
 More precisely, assume  that we have a (usual commutative) crepant resolution of singularities $\pi\colon Y\to X$.  A {\bf D-equivalence} between $Y$ and a noncommutative resolution $A$ is an equivalence of dg-categories $\Cohdg(Y)\cong A\dgmod$.  

A standard way for these to arise is through a {\bf tilting generator}, a locally free sheaf $\mathcal{T}$ on a scheme $Y$ such that  $\Ext^\bullet(\mathcal{T},-)$ induces an equivalence of derived categories between $\End(\mathcal{T})\dgmod$ and $\Cohdg(Y)$; note that as a consequence of this definition, we have have that $\Ext^i(\mathcal{T},\mathcal{T})=0$ for $i>0$ and $\langle\mathcal{T}\rangle=\Cohdg(Y)$, which is the definition given in \cite[Def. 1.1]{KalDEQ}.  
The following is a corollary of \cite[Lem. 3.2.9 \& Prop. 3.2.10]{van2004three}:
\begin{lemma}\label{lem:tilt-nccr}
Suppose $\mathcal{T}$ is a tilting generator on $Y$ such that the structure sheaf $\mathcal{O}_Y$ is a summand of $\mathcal{T}$, and let $M=\Gamma(Y;\mathcal{T}).$ Then $A=\End_{\Coh(Y)}(\mathcal{T})\cong \End_R(M)$ is a noncommutative crepant resolution of singularities, canonically D-equivalent to $Y$.  
\end{lemma}
We will prove Theorem \ref{th:nc-res} using this lemma, by constructing an appropriate tilting generator (see Theorem \ref{th:tiltingiso}).

\subsection{Comparison of multiplicative and additive varieties}
As mentioned before, a key component of the proof of  Theorem \ref{th:nc-res} is that multiplicative hypertoric varieties  ``locally look like additive hypertoric varieties.'' More precisely, additive and multiplicative hypertoric varieties have complex moment maps with respective targets $\mathfrak{g}^\vee$ and $G^\vee,$ and we will show that formal neighborhoods of the fibers of these maps agree.
  
  We identify $\Z^n$ with the cocharacter lattice of $D$ by the map sending $\Ba\in \Z^n$ to the cocharacter $s\mapsto \operatorname{diag}(s^{a_1},\dots, s^{a_n})\colon\mathbb{G}_m\to D$.
  Under this identification, the standard characters $\ep_i$ of $D^\vee$ are sent to the unit vectors of $\Z$.  We can also consider these as functions on $D^\vee$, and for a given $h\in D^\vee$, we denote the values of these by $h_i$. 
   
 Assume now that $h\in D^\vee$ is contained in the subtorus $G^{\vee}$. The fact that $h$ lies in $G^\vee$ is reflected by the relation $\prod_ih_i^{a_i}=1$ for $\Ba\in\ft_\Z\subset \Z^n$.
  
  Consider the formal neighborhood $U$ of $h$ inside $G^\vee$, and let $\widehat{\mht}_h$ be the preimage of this neighborhood inside $\mht$.  In particular, any power series in the functions $\sz_i\sw_i+1-h_i$ is well-defined as a function on this completion.  
  
  Let $\aht=\aht_{(\tone,\GIT)}$ be the additive hypertoric variety associated to the data of the exact sequence \eqref{eq:basictoriseq} and $(\tone,\GIT)\in \mathfrak{t}^\vee_\mathbb{H}$; in order to avoid confusion, we will use $\sx_i,\sy_i,$ rather than $\sz_i,\sw_i,$ as variables for the additive hypertoric variety.  Let $\log(h)\in \fg^\vee$ be a choice of preimage of $h$ under the  exponential map, such that $\log(h_i)=0$ if $h_i = 1$; implicitly, this fixes choices of $\log(h_i)$ for all $i$, such that $\sum a_i\log h_i=0$ for all $\Ba\in \ft_\Z$.

  Let $\widehat{\aht}_{\log(h)}$ be the base change of $\aht$
  to a formal neighborhood of $\log(h)\in \fg^\vee$.  Let $\ml_\chi$ be the line bundle on $\mht$ induced by a character $\chi:T\to \mathbb{G}_m$ and $\al_\chi$ be the correspond line bundle on $\aht$. 
  \begin{theorem}\label{thm:complete-iso}
      The formal neighborhoods $\widehat{\mht}_h$ and $\widehat{\aht}_{\log(h)}$ are isomorphic, by an isomorphism identifying the character line bundles $\ml_\chi$ and  $\al_{\chi}$.
  \end{theorem}
  \begin{proof}
  Recall that the projective coordinate rings of $\aht$ and $\mht$ are $\C[T^*\C^n][t]^T/\langle \mu_\C =0 \rangle$ and $\C[\TAno][t]^T/\langle\mmm_{\C^\times}=0\rangle,$ where $\mu_\C,\mmm_{\C^\times}$ denote the complex moment maps for the $T$-actions, and $t$ is an additional variable of degree $1$ with $T$-weight $-\GIT.$ We will produce isomorphisms between certain completions of these coordinate rings.
  
  Write $\widehat{\TAno}_h$ for the completion of $\TAno$ with respect to the the ideal generated by $\sz_i\sw_i+1-h_i$.
  
      Note that if $h_i\neq 1$, then we have that $\sz_i\sw_i$ is itself invertible, and 
      \begin{align*}
         \log(1+\sz_i\sw_i)&=\log(h_i+(\sz_i\sw_i+1-h_i))\\
         &=\log(h_i)+\frac{\sz_i\sw_i+1-h_i}{h_i}-\frac{(\sz_i\sw_i+1-h_i)^2}{2h_i^2}+\frac{(\sz_i\sw_i+1-h_i)^3}{3h_i^3}-\cdots
      \end{align*} is an invertible function on $\widehat{\TAno}_{h}$.  If $h_i=1$, then 
      \[\log(1+\sz_i\sw_i)=\sz_i\sw_i-\frac{\sz_i^2\sw_i^2}{2}
      +\frac{\sz_i^3\sw_i^3}{3}-\cdots\]
      is still well-defined, but no longer invertible.  However, the quotient  \[\gamma_i:=\frac{\log(1+\sz_i\sw_i)}{\sz_i\sw_i}=1-\frac{\sz_i\sw_i}{2}
      +\frac{\sz_i^2\sw_i^2}{3}-\cdots\] is invertible.
      
      Let $\widehat{T^*\ACn}_{\log{h}}$ be the completion of $T^*\ACn$, now with variables $\sx_i$ and $\sy_i,$ at $\sx_i\sy_i-\log(h_i)$.
      Let $\sx_i,\sy_i$ map to $\C[\widehat{\TAno}_h]$ by 
      \[\sx_i\mapsto \sz_i\qquad \sy_i\mapsto 
      \displaystyle\sw_i\gamma_i. 
     \]
      This map is an isomorphism, and its inverse is given by 
      \[\sz_i\mapsto \sx_i,\qquad \sw_i\mapsto 
      \displaystyle \sy_i\delta_i, \]
      where we define
      \[\delta_i:=\frac{h_i e^{\sx_i\sy_i - \log (h_i)}-1}{\sx_i\sy_i}=
      \frac{h_i-1}{\sx_i \sy_i} + \frac{h_i}{
      \sx_i \sy_i} \sum_{k=1}^\infty \frac{(\sx_i\sy_i - 
      \log(h_i))^k}{k!}.\]
      Furthermore, under these maps, we have 
      \[\sum_{i=1}^na_i\sx_i\sy_i=\sum_{i=1}^na_i\log(1+\sz_i\sw_i)=\log\left(\prod_{i=1}^n(1+\sz_i\sw_i)^{a_i}\right),\]
      so that these isomorphisms intertwine the additive and multiplicative moment maps.  Furthermore, since $\gamma_i$ and $\delta_i$ are $D$-invariant, this is an equivariant isomorphism.
      
      Thus, we obtain an isomorphism of projective coordinate rings for $\widehat{\mht}_h$ and $\widehat{\aht}_{\log(h)}$, which moreover must match the modules corresponding to the line bundles $\ml_\chi$ and $\al_\chi$.
  \end{proof}

\subsection{Construction of the tilting bundle}\label{sec:construction-tilting}
In view of Lemma \ref{lem:tilt-nccr}, we can complete the proof of Theorem \ref{th:nc-res} by constructing a tilting generator on $\mht_{(\tone,\partwo)}$ for $\partwo$ generic.  
This has already been carried out in the additive case, implicitly in \cite{Stadnik} and explicitly in \cite[Theorem 3.36]{McBW}.   This is accomplished using a now-standard technique in geometric representation theory, \emph{quantization in characteristic p}. 

Roughly speaking, one notices that a cotangent bundle, as for instance $T^*\mathbb{A}^n,$ has a noncommutative deformation given by the sheaf of differential operators on the base; and moreover, in characteristic $p$, this sheaf $\mathcal{D}$ is Azumaya over its center, which is actually (a Frobenius twist of) the original cotangent bundle. If the  parameter $\Bfield$ is $p$-torsion, it may then be used as a noncommutative moment map parameter for a quantum Hamiltonian reduction, after which one obtains an Azumaya algebra $\mathscr{A}$ on the Hamiltonian reduction $\aht$ of $T^*\mathbb{A}^n,$ as desired. We would like write $\mathscr{A}$ as the endomorphism sheaf of a vector bundle, i.e. to show that it is split.  Based on some homological properties of the Azumaya algebra {\it if} a splitting bundle were to exist, it would be necessarily tilting by \cite[Lem. 2.7]{KalDEQ}, and for generic $\Bfield$, it would be a tilting generator by \cite[Prop. 4.2]{KalDEQ}.  
No such splitting bundle exists, but the restriction of $\mathscr{A}$ to a formal neighborhood of the fiber $\mu^{-1}(0)$ does split---it is isomorphic to the endomorphisms of a $\mathbb{G}_m$-equivariant vector bundle $\ta$ on $\aht$, restricted to this formal neighborhood. One can conclude that the vector bundle $\ta$ is a tilting generator by the same arguments as above.  


For the multiplicative case, this method of quantization has an analogue:  instead of deforming to the sheaf of differential operators, one deforms to a sheaf of $q$-\emph{difference} operators. This was accomplished for multiplicative hypertoric varieties in \cite{Cooney16,Ganev18}, and for general multiplicative quiver varieties in \cite{GJS19}.  This allows us to define an Azumaya algebra on $\mht$ as before, and one could easily prove the same infinitesimal splitting property, but there is no contracting $\C^{\times}$-action that allows us to construct a tilting generator via the same automatic process.

Disappointing as this fact is, we can still use it as inspiration to {\it guess} a tilting generator.  In fact, this is not truly a guess; the bundle in question arises naturally when we view $\C[\mht]$ as a multiplicative  Coulomb branch (the Coulomb branch $\mathcal{C}_4$ in the terminology of \cite{teleman2018role}),
and the other nodes in $\overline{\cQ}$ as spectrally-flowed line operators, \emph{i.e.}, as objects of the {\bf extended BFN category} introduced in \cite[\S 3]{WebSD}.

However one chooses to motivate the definition of this tilting generator, it is easy to describe: it is a direct sum of equivariant line bundles $\ml_\chi$, which we enumerate below.
 
Choose a lift $\logB=\log(\Bfield)\in \ft^\vee_\R$.  
Let $\fg^{\vee,\logB}_\R$ be the preimage of $\logB$ under the projection $\fd^\vee_\R\to \ft^\vee_\R$.  Let
 \[ \Delta_{\Bx}=\{\Ba\in \fg^{\vee,\logB}_\R \mid x_i < a_i <
x_i+1\},
\qquad \Lambda(\logB)= \{\Bx\in \mathfrak{d}^\vee_\Z\mid \chamber_\Bx\neq \emptyset\}.\]
These are the chambers and their parametrizing set for a periodic hyperplane arrangement in $\fg^{\vee,\logB}_\R,$ which we denote $\textgoth{B}^{\operatorname{per}}_{{\logB}}$. We write $\overline{\Lambda}(\logB)$ for the quotient of $\Lambda(\logB)$ by the action of $\mathfrak{g}^{\vee}_\Z$. This is a finite set, parametrizing the chambers of the toric hyperplane arrangement $\textgoth{B}^{\operatorname{tor}}_{{\logB}} := \textgoth{B}^{\operatorname{per}}_{{\logB}} / \mathfrak{g}^{\vee}_\Z$.  Note that if we choose $\logB$ differently, then the result will change by an element of $\ft^\vee_\Z$.  

We continue to assume that $\logB$ is generic, so that there is a
neighborhood $U$ of $\logB$ in $\R\otimes \ft^\vee$ such that for
all $\logB'\in U$, we have $\Lambda(\logB)=\Lambda(\logB')$. In
particular, the hyperplanes in
$\textgoth{B}^{\operatorname{per}}_{\logB}$ intersect
generically. 

Consider the following vector bundles, defined on $\mht$ and $\aht,$ respectively:
\[\tm_{\logB}:=\bigoplus_{\chi\in \overline{\Lambda}(\logB)}\ml_{\chi}, \qquad \ta_{\logB}:=\bigoplus_{\chi\in \overline{\Lambda}(\logB)}\al_{\chi}.\]
Note that if we choose a different value of $\logB$ while leaving $\Bfield$ unchanged, the effect will be to tensor $\tm_{\logB}$ with the line bundle corresponding to the element of $\ft^\vee_\Z\cong \operatorname{Pic}(\mht)$ we change our branch of $\log$ by.  In particular, we can always choose $\logB$ so that $\mathbf{0}\in \Lambda(\logB)$, that is, the trivial bundle $\cO_{\mht}$ is a summand of $\tm_{\logB}$.
\begin{theorem}[\mbox{\cite[Prop. 3.36]{McBW}}]
  The vector bundle $\ta_{\logB}$ is a tilting generator of $\aht$, and so $\End(\ta_{\logB})$ gives an NCCR of $\aht_{(0,\tone)}$.  
\end{theorem}
This motivates the corresponding multiplicative result:
 \begin{theorem} \label{thm:tilting}
 The bundle $\tm_{\logB}$ is a tilting generator for the category $\Coh(\mht_{\tone,\GIT})$.
 \end{theorem}
 \begin{proof}
 First, let us verify that $H^i\big(\mht_{\tone,\GIT};\mathcal{E}nd(\tm_{\logB})\big)=0$ for $i>0$.  We compute this by first pushing forward $\mathcal{E}nd(\tm_{\logB})$ by the map $\mmm_{\C^\times}$.  Since tensoring with the completion of $\C[G^\vee]$ at $\gamma$ is exact, flat base change shows that the completion of $\mathbb{R}^i(\mmm_{\C^\times})_*(\mathcal{E}nd(\tm_{\logB}))$ is the same as the $i$th cohomology sheaf of the pullback of $\mathcal{E}nd(\tm_{\logB})$ to a formal neighborhood of the fiber.  That is, it suffices to prove this fact on the completion $\widehat{\mht}_h$ for all $h\in G^\vee$.  By Theorem \ref{thm:complete-iso}, we can transport this bundle to $\widehat{\aht}_{\log(h)}$ and obtain the completions of the line bundles $\al_{\chi}$ for $\chi\in \bar{\Lambda}$.  By \cite[Prop. 3.36]{McBW}, the corresponding sheaf $\ta_{\logB}$ on $\aht$ is a tilting generator, and so we have the desired cohomology vanishing.
 
 Now, we only need to check that this vector bundle generates the derived category.  Here we will use techniques from \cite{KalDEQ}.  In order to translate between that paper and this one, note that we will use freely throughout the Morita equivalence between coherent sheaves and right $\mathcal{E}nd(\tm_{\logB})$-modules induced by $M\mapsto M\otimes(\tm_{\logB})^\vee$, so any place where Kaledin uses sections, we should apply sections after using this functor to turn a coherent sheaf into a right $\mathcal{E}nd(\tm_{\logB})$-module.  
 
This will, of course, be the case if the functor \begin{equation}\label{eq:push-pull-1} 
 	M\mapsto \mathbb{R}\Gamma(\mht_{\tone,\GIT}; M\otimes(\tm_{\logB})^\vee)\overset{L}\otimes_{\End(\tm_{\logB})}\tm_{\logB}
 \end{equation}  of taking derived sections and then localizing is an equivalence, since this immediately shows that if $\mathbb{R}\Gamma(\mht_{\tone,\GIT}; M\otimes(\tm_{\logB})^\vee)=0$, then $M=0$.  Consider the sheaf of algebras on $G^\vee$ given by $E_{G^{\vee}}=(\mmm_{\C^\times})_*(\mathcal{E}nd(\tm_{\logB}))$.  Since $G^\vee$ is affine, we have that 
 \begin{equation}\label{eq:push-pull-2} 
 	\mathbb{R}\Gamma(\mht_{\tone,\GIT}; M\otimes(\tm_{\logB})^\vee)\overset{L}\otimes_{\End(\tm_{\logB})}\tm_{\logB}\cong (\mathbb{R}\mmm_{\C^\times})_* M\otimes(\tm_{\logB})^\vee \overset{L}\otimes_{E_{G^{\vee}}}\tm_{\logB}.
 \end{equation} 
 
 This functor \eqref{eq:push-pull-1} is given by Fourier-Mukai transform with a kernel on $\mht_{\tone,\GIT}\times \mht_{\tone,\GIT}$ which Kaledin denotes $\mathcal{M}^\bullet$ in \cite[Lem. 3.1]{KalDEQ}; this is the localization of $\End(\tm_{\logB})$ considered as a bimodule over itself. 
 By \eqref{eq:push-pull-2}, this kernel is supported on $\mht_{\tone,\GIT}\times_{G^\vee} \mht_{\tone,\GIT}$. 
 
 The counit of the adjunction induces a map $\mathcal{M}^\bullet\mapsto \mathcal{O}_{\Delta}$ to the structure sheaf of the diagonal $\Delta \subset \mht_{\tone,\GIT}\times \mht_{\tone,\GIT}$, whose cone we denote $\mathcal{K}^{\bullet}$.  As observed by Kaledin in \cite[Def. 3.1]{KalDEQ}, $\mathcal{K}^{\bullet}=0$ if and only if $\tm_{\logB}$ is a tilting generator.  Using flatness again, the functors
 \[M\mapsto  (\mathbb{R}\mmm_{\C^\times})_* M\otimes(\tm_{\logB})^\vee\qquad N\mapsto N\overset{L}\otimes_{E_{G^{\vee}}}\tm_{\logB}\]
 commute with restriction to a formal neighborhood of $h\in G^\vee$ and to $\widehat{\mht}_h$.  This shows that the restriction of $\mathcal{K}^\bullet$ to the formal neighborhood $\widehat{\mht}_h\times \widehat{\mht}_h \subset \mht_{\tone,\GIT}\times \mht_{\tone,\GIT}$ is the same sheaf defined with respect to restriction of $\mathcal{E}nd(\tm_{\logB})$ to $\widehat{\mht}_h$, i.e. the sheaf that measures the failure of this restriction to be a tilting generator.  
 
 As discussed above, we have already used  Theorem \ref{thm:complete-iso} and \cite[Prop. 3.36]{McBW} to argue that this restriction is a tilting generator, so this shows that the restriction of $\mathcal{K}^\bullet$ to the formal neighborhood of any fiber of $\mht_{\tone,\GIT}\times_{G^\vee} \mht_{\tone,\GIT}\to G^\vee$ is trivial.  This shows that $\mathcal{K}^\bullet$ itself is trivial and completes the proof. 
 \end{proof}
 \begin{corollary} \label{cor:tilteq}
     The map $\mathcal{F} \to \RHom(\tm_{\logB}, \mathcal{F})$ defines an equivalence of categories $\Cohdg(\mht) \cong \End(\tm_{\logB})\dgmod$.
 \end{corollary}
 
 Together with Lemma \ref{lem:tilt-nccr}, this shows that if we choose $\logB$ so that $\mathbf{0}\in \Lambda(\logB)$, then the algebra $\End(\tm_{\logB})$ is a noncommutative crepant resolution of singularities.  

 \subsection{Computation of the endomorphism algebra} 
%
 We are now ready to compute the endomorphism algebra for the tilting bundle defined above.
Recall from Section \ref{sec:global} that we defined quivers with relations $\cQ_\Bfield,\overline{\cQ}_\Bfield,$ and $A_{\Bfield}$ is the path algebra of the latter quiver.
\begin{theorem}\label{th:tiltingiso}
There is an isomorphism
\begin{equation}\label{eq:tiltingiso}
    A_{\Bfield}\cong \End(\tm_{\Bfield})
\end{equation} sending the length zero path $e_{\Bx}$ to the projection to the corresponding line bundle.  This induces an equivalence of categories $\Cohdg(\mathfrak{U})\cong \overline{\cQ}_{\Bfield}\operatorname{-dgmod}$.
\end{theorem}
Note that this completes the proof of Theorem \ref{th:nc-res}.
\begin{proof}
We define the map \eqref{eq:tiltingiso} as follows: the map $\C[G^\vee]\to \End(\tm_\Bfield)$ is given by the composition
\[
\C[G^\vee]\to \C[\mht] \to \End(\tm_{\logB}),
\]
where the first map is the pullback and the second is the action of functions as endomorphisms of any coherent sheaf. And for $\Bx,\By\in\Lambda$ such that
$\By=\Bx+\ep_i$, we send
\[c_{\By,\Bx}\mapsto \sz_i,\qquad c_{\Bx,\By}\mapsto \sw_i.\]
We can easily check that this is a homomorphism: the relations (2)  from Definition~\ref{def:perv-tor} are just commutativity, and the relations (1) are the moment map condition.  

Given any $\Bx,\By\in \Lambda$, there is a unique element $c_{\Bx,\By}$ defined as the product along any minimal length path between these vertices.  This maps to the unique minimal monomial in $\C[\mmm^{-1}(G^\vee)]$ whose weight under the action of $D$ is the difference $\Bx-\By$, and so $c_{\Bx,\By}\cdot \C[G^\vee]$ maps surjectively to all elements of that weight in $\C[\mmm^{-1}(G^\vee)]$.  Since we have assumed that $T$ contains no coordinate subtori, every cocharacter into $T$ has non-trivial weight on at least 2 coordinates, and by the symplectic property, this implies that the Kirwan-Ness stratum for this cocharater has codimension $\geq 2$.  Thus, $\C[\mmm^{-1}(G^\vee)]$ surjects to functions on the stable locus.

This shows that $c_{\Bx,\By}\cdot \C[G^\vee]$ maps surjectively to the elements of $\Hom(\ml_{\Bx},\ml_{\By})$ of the correct $T$-weight. Ranging over all $\By$ with the same image in $\bar{\Lambda}$, we obtain all homomorphisms $\ml_{\Bx}\to \ml_{\By}$, so the desired map is surjective.  On the other hand, our relations allow us to shorten any path to an element of $c_{\Bx,\By}\cdot \C[G^\vee]$.  
\end{proof}

Combining Corollary \ref{cor:Fuk-Q} and Theorem \ref{th:tiltingiso}, we deduce our main theorem. 

\begin{theorem}\label{thm:main-thm}
Let $\parone \in T^\vee$ be generic. Then there is an equivalence of dg categories
\begin{equation}\label{eq:hms}
\Fuk(\mht_{(\parone,\tone)})\cong \Cohdg(\mht_{(\tone,\parone)}).
\end{equation}
For $\parone\in T^\vee_\R,$ this equivalence is induced from an equivalence of abelian categories, the respective hearts of the perverse t-structure and the noncommutative resolution t-structure.
\end{theorem}

\subsection{SYZ mirror symmetry}\label{subsec:SYZ2}
In order to make contact with the theory of SYZ mirror symmetry, we make the following definition:
\begin{definition}\label{def:torusfib-object}
    An object of $\Fuk(\mhta)$ is a {\em torus fiber object} if, under the equivalence $\Fuk(\mhta)\cong\cQ_\parone\dgperf,$ it corresponds to an quiver representation with dimension vector $(1,\ldots,1).$
\end{definition}

To see that this definition is reasonable, recall that the Lagrangian skeleton $\LL$ of $\mhta$ is equivalent to the central fiber of the holomorphic integrable system on the Dolbeault space described in \ref{subsec:dolbeault}, and that smooth fibers of this integrable system degenerate to the central fiber with multiplicity one along each toric component. (If such a smooth torus fiber $\mathcal{T}$ is an exact Lagrangian, then an actual embedding of the moduli of local systems on $\mathcal{T}$ into the Fukaya category $\Fuk(\mhta)$ can be accomplished using the nearby cycles functors studied in \cite{NS20}. A similar embedding for non-exact torus fibers is provided by \cite{Shende-nonexact}, at the cost of introducing a Novikov variable.)

SYZ mirror symmetry, as discussed in Section~\ref{sec:syz-appendix}, takes the form of a statement relating torus fiber objects in a Fukaya category to points of a dual variety. Using Definition \ref{def:torusfib-object}, a statement of this form can be obtained as a corollary of the mirror symmetry isomorphism in Theorem \ref{thm:main-thm}. Recall that there is an equivalence
\begin{equation}\label{eq:tilting2}
    \Cohdg(\mht_{(\tone,\parone)})\cong \Cohdg(\mht_{(\tone,\GIT)})
\end{equation}
between the ``noncommutative resolution'' category $\Cohdg(\mht_{(\tone,\parone)})$ and the category of coherent sheaves on an actual resolution $\mht_{(\tone,\GIT)},$ given by a tilting bundle $\tm$ on the resolution.

\begin{corollary}\label{cor:skyscraper-torus}
    Let $\cO_p\in \Coh(\mht_{(\tone,\GIT)})$ be a skyscraper sheaf at a point of the resolved variety $\mht_{(\tone,\GIT)}.$ Under the composition of equivalences \eqref{eq:tilting2} and \eqref{eq:hms}, the object $\cO_p$ corresponds to a torus-like object in the Fukaya category $\Fuk(\mhta).$
\end{corollary}
\begin{proof}
The equivalence \eqref{eq:tilting2} sends $\cO_p$ to the natural $\End(\tm)$-module given by $\R\Hom(\tm,\cO_p).$ As $\cO_p$ is a skyscraper sheaf and $\tm$ is a direct sum of line bundles, this module has dimension 1 at each summand of $\tm,$ as desired.
\end{proof}

\begin{remark}
The most robust statement of SYZ mirror symmetry would actually construct $\mht_{(\tone,\parone)}$ as a moduli space of $(1,\ldots,1)$-dimensional representations of the quiver $\cQ_\parone.$ Corollary \ref{cor:skyscraper-torus} shows that this is roughly correct: from examples, it appears that the map $\mht_{(\tone,\parone)}$ to the moduli stack of torus fiber objects is birational.  However, this map is definitely not an isomorphism, and in particular, certain points in this moduli space will be in the image of  $\mht_{(\tone,\parone)}$ for some $\parone$ and not for others; furthermore, some points, such as the trivial representation, will not be in this image for any value of $\parone$.  
At the moment, we know no principled explanation for how this image depends on the GIT parameter $\GIT$; in particular, it does not seem to have a description in terms of the usual slope-stability for quiver representations.
\end{remark}


\section{Monodromy functors from perverse schobers}\label{sec:mon}
As mentioned in the introduction, a mirror symmetry theorem over the nonsingular parameter space would include not only the equivalences \eqref{eq:hms},
but also equivalences among these as we circumnavigate walls in the parameter space.
These equivalences would thus fit into an action of $\pi_1(T^\vee_{\operatorname{gen}})$ by ``monodromy in a local system of categories,'' compatible with the equivalences \eqref{eq:hms}. 

In fact, the categories we have computed should live not just over the nonsingular parameter space but over the whole parameter space $T^\vee$; the resulting family of categories will fail to be a local system precisely along the real subtori of $T^\vee$ which are the walls in a toric hyperplane arrangement $\Dtor.$ The notion of such a ``perverse sheaf of categories'' exists already in mirror symmetry, thanks to work of Kapranov and Schechtman \cite{KSschobers,KShyperplane,BKS}, where it goes by the name of {\bf perverse schober.}

One advantage of the description from \cite{KShyperplane,BKS} is that it presents the category of perverse sheaves on a complex vector space stratified with respect to a real hyperplane arrangement only in terms of the behavior of these sheaves over subsets of the \emph{real locus} inside this complex vector space. For instance, the monodromy of such a perverse sheaf can be recovered as a composition of canonical and variation maps into and out of the singular locus. This should be understood as a generalization of the geometrical fact that monodromy of nearby cycles can be deduced from a sufficiently good understanding of their relation with vanishing cycles.

Perverse schobers give a presentation of these facts at the level of categorical invariants: in short, all the data of perverse sheaf of categories can be encapsulated in functors between nearby and vanishing categories. Using this insight, we present here a construction of a ``perverse schober of Fukaya categories,'' and we compare it to the B-side perverse schober defined using wall-crossing functors in\cite{WebcohI}.
As expected, we prove that equivalence of categories \eqref{eq:hms} extends to an equivalence of perverse schobers. 

\subsection{Perverse schobers stratified by hyperplane arrangements}
The B-side schober in this case fits into a pattern familiar in representation theory, namely the action of wall-crossing functors, as explained in \cite[Section 4]{WebcohI}. We review that discussion here and then specialize to the hypertoric case. 
	
We need from \cite{BKS} the notion of a perverse schober on a complex vector space stratified with respect to a real hyperplane arrangement. (Technically, we are interested in perverse sheaves on a complex torus stratified by a real toric hyperplane arrangement, but as usual we choose to simplify notation by passing to the universal cover and working $\mathfrak{t}^\vee_\Z$-equivariantly.) The precise version of the definition we use is based on the procedure which associates to a perverse sheaf its sections along the star of each stratum in a real hyperplane arrangement. (Recall from \cite[Section 5]{KSschobers} that a schober on a base $B$ should associate categories to \emph{Lagrangian} subsets of $B$.)

	\begin{example}[``Double-cut realization'']
		Consider the case when the underlying real vector space is one-dimensional, stratified by a single hyperplane (point) at the origin. The Kapranov-Schechtman description presents $\Perv(\C,0)$ in the following way: to a perverse sheaf $E$, we associate one vector space for each stratum in the real hyperplane arrangement. To the two open strata, we associate the sections $E_{\pm}:=\Gamma_{\R_{\pm}}(E)$ of $E$ on the positive or negative real loci (which are equivalent to the stalks of $E$ at a generic positive and negative real point), and to the closed stratum, we associate the sections $E_0:=\Gamma_\R(E)$ along the whole real line (which we think of as the star of the closed stratum). For each inclusion of stratum closures, we get a pair of maps $\Gamma_{\R_\pm}(E)\leftrightarrows\Gamma_R(E)$, and the whole perverse sheaf is determined by these maps.

		Similarly, a perverse schober $\scrE$ in this setting is determined by three dg categories, which we write $\scrE_\pm$ and $\scrE_0,$ along with pairs of adjoint functors $\scrE_\pm\leftrightarrows\scrE_0$ satisfying some conditions, which we detail below.
	\end{example}

	The case of a higher-dimensional hyperplane arrangement is the obvious generalization of this example: the data of a perverse schober includes a category $\scrE_A$ associated to the star of each stratum $A$ in the real vector space, along with adjoint functors for incidences of strata. In the following definition (the modification in \cite{WebcohI} of the definition from \cite{BKS}), we detail one way of encapsulating this data, which we will find convenient for our purposes.

\begin{definition}\label{def:schober}
	Let $\ft^\vee=\ft^\vee_\R\otimes\C$ be the complexification of a real vector space, and $\Dper\subset \ft^\vee_\R$ a periodic hyperplane arrangement. A {\bf perverse schober on $\ft^\vee,$ stratified by $\Dper$,} consists of the following data:
	\begin{enumerate}
		\item A dg category $\scrE_A$ for each face $A$ in $\Dper$.
		\item A pair of adjoint functors $\delta_{C'C}:\scrE_C \leftrightarrows \scrE_{C'}:\gamma_{CC'}$ for any incidence of faces $C'\leq C.$ The left adjoint $\delta_{C'C}$ is the {\bf specialization}, and its right adjoint $\gamma_{CC'}$ the {\bf generalization}, functor for the incidence $C'\leq C$. If $\bar{A}\cap \bar{C}\neq 0,$ and $B$ is the unique open face in the intersection, these functors compose to give a {\bf transition functor} $\phi_{AC}:=\gamma_{AB}\delta_{BC}:\scrE_C\to \scrE_A.$
		\item Isomorphisms $\gamma_{CC'}\gamma_{C'C''}\cong \gamma_{CC''}$ for consecutive incidences $C''\leq C'\leq C,$ with associativities for longer incidences.
		\item Isomorphisms $\phi_{AB}\phi_{BC}\cong \phi_{AC}$ for $(A,B,C)$ a collinear family of faces, with associativities for longer collinear families. Hence for any two faces $A$ and $B$, we can define $\phi_{AB}:=\phi_{AA_1}\phi_{A_1A_2}\cdots\phi_{A_nB}$ for $A_1,\ldots,A_n$ the ordered list of faces passed by a generic line segment from $A$ to $B$; the previous isomorphism data guarantees that these composite transition functors are well-defined.
	\end{enumerate}
	These data are required to satisfy the following further conditions:
	\begin{enumerate}\renewcommand{\theenumi}{\roman{enumi}}
		\item If $C'\leq C,$ the unit natural transformation $\gamma_{C'C}\delta_{CC'}\to 1_{\scrE_{C'}}$ is an isomorphism. For $B$ any face in the intersection $\bar{A}\cap \bar{C}$, this gives a natural isomorphism $\phi_{AB}\cong \gamma_{AB}\delta_{BC}.$
		\item If $A$ and $B$ have the same dimension, span the same subspace, and are adjacent across a face of codimension 1 in $A$ and $B$, then $\phi_{AB}$ is an equivalence.
	\end{enumerate}
\end{definition}
In our setting, we will associate to each stratum $C$ an algebra $A_C$, and to each pair of strata a bimodule ${}_C T_{C'}$ so that 
\begin{align*}
    \scrE_C & = A_C \dgperf \\
    \phi_{CC'} & = {}_C T_{C'} \Lotimes-.
\end{align*}
 In particular, each top-dimensional stratum $C$ of our stratification will consist of generic parameters $\parone$ in the sense of Defintion \ref{def:paronegeneric}, and we will set $\scrE_C = A_{\parone} \dgperf$, which we have seen describes both coherent sheaves on $\mht_{1, \parone}$ and the Fukaya category of $\mht_{(\parone, 1)}$.

\begin{remark}
The above definition may seem involved; luckily, we will not need here to work with it in any serious way: the verification that the B-side family we discuss satisfies this definition was already performed in \cite{WebcohI}, and we will prove that the A-side data we construct is isomorphic to that B-side schober.
\end{remark}





\subsection{The B-side schober}
As promised, we define a pair of perverse schobers over the parameter space $\ft^{\vee}_\R$ which encapsulate the whole family of B-side and A-side categories, respectively, as the parameter $\parone\in\ft^\vee_\R$ varies. We begin here with the B-side schober.

\begin{definition}
Let $\Dper$ be the periodic hyperplane arrangement in $\ft^\vee_\R$ whose hyperplanes are the loci in which the parameter $\parone\in \ft^\vee_\R$ is nongeneric, in the sense of Definition~\ref{def:paronegeneric}.
\end{definition}

In other words, the walls of $\Dper$ are defined by the equations $\langle \sigma, \parone\rangle = n,$ for $n\in\Z,$ and $\sigma$ a signed circuit (as defined in Definition~\ref{def:circuits}).
The resulting stratification of $\ft^{\vee}_\R$ is the stratification by topological type of $\Bper_\parone$.

\begin{definition}
     Let $\scrH$ be the periodic coordinate hyperplane arrangement in $\fd^{\vee}_\R = \R^n$. We write $\widetilde{\scrQ}$ for the quiver with relations associated to the arrangement $\scrH$ as in Definition \ref{def:abstract-quiver}. 
\end{definition}

The quiver $\widetilde{\scrQ}$ has a natural geometric interpretation.  The trivial line bundle on $\mmm_{\C^\times}^{-1}(1)$ carries a different equivariant structure for each character of $D$, or equivalently for each chamber of $\scrH$.  One can easily modify the proof of Theorem \ref{th:tiltingiso} to show that this quiver with relations controls the $D$-equivariant homomorphisms between these line bundles.  Of course, each such homomorphism maps to a homomorphism between the associated bundles on $\mathfrak{U}$. However, it seems unlikely to always be a surjective map; the tilting property of $\tm$ guaranteed that this was not an issue in Theorem \ref{th:tiltingiso}, but a general line bundle could have higher cohomology that interferes with the surjectivity.

Fix a stratum $C$ in $\Dper$, and write $U_C$ for the preimage of the star of $C$ under the projection map $\fd^\vee_\R\to \ft^\vee_\R$. The intersection $U_C\cap \scrH$ picks out a subset of the chambers and faces of $\scrH$; we can understand this as a hyperplane arrangement in $U_C,$ which we will denote by $\scrH_C.$ 

\begin{definition} \label{def:prequotientschobquiver}
We write $\widetilde{\scrQ}_C$ for the subquiver with relations of $\widetilde{\scrQ}$ associated to the inclusion $\scrH_C\subset \scrH$. Thus $\widetilde{\scrQ}_C$ contains only those vertices whose associated chambers intersect $U_C$.

\end{definition}

\begin{definition} \label{def:quotientschobquiver}
We write $\bqsch_C$ for the quotient of this quiver by the action of the lattice $\fg^\vee_\Z$. 
We let $\widetilde{A}_C$ be the path algebra of this graph modulo the relations already discussed. 

The map $T \to D$ induces an embeding $\C[\pi_1 (T)] \to \C[\pi_1 (D)]$. We have a central embedding $\C[\pi_1 (D)] \to \widetilde{A}_C$. Write $A_C$ for the quotient of $\widetilde{A}_C$ by the ideal generated by $\gamma - 1$ for all $\gamma \in \pi_1(T)$.
\end{definition}
Of course, the algebra $A_C$ contains a central copy of $\C[\pi_1(G)]$, which we identify with the function ring of $G^\vee$.  

If we interpret the quiver $\widetilde{\scrQ}$ as the Ext algebra of a set of $D$-equivariant line bundles, then $\bqsch_C$ can be understood as the Ext algebra of the same line bundles, deequivariantized from $D$ to $T$.

The same quiver with additive moment map relations instead of multiplicative ones controls the schober for coherent sheaves on an additive quiver variety described in \cite{WebcohI} and gives the algebra denoted $A_C$ in \cite[\S 4.5]{WebcohI}, which we will denote $A_C^{(+)}$ here to avoid confusion.  The proof of Theorem \ref{thm:complete-iso} is easily extended to show that:
\begin{lemma}\label{lem:quiver-m-a}
	The completion of $A_C$ at $h\in G^\vee$ is isomorphic to the completion of $A_C^{(+)}$ at $\log h$, sending the idempotent $e_x$  for $x\in \bqsch_C$ to the same idempotent in $A_C^{(+)}$.  
\end{lemma}


If $C$ is a top-dimensional stratum (so that its star is again $C$), the arrangement $U_C$ is the product of $C$ with $\Bper_{\parone}$ for any $\parone \in C$. 
This gives a natural isomorphism $A_C \cong A_\parone$.


More generally, if we have two faces $C_+$ and $C_-$ whose closures intersect, then there is a face $C$ such that $\bar{C}=\bar{C}_+\cap \bar{C}_-$. In this case, $\qsch_C$ contains $\qsch_{C_\pm}$ as subgraphs, so $A_C$ contains idempotents $e_\pm$ such that $e_{\pm}A_Ce_{\pm}=A_{C_\pm}$.  
\begin{definition}
    For a pair $(C_+,C_-)$ of adjacent faces whose closures meet at $C$, the {\bf wall-crossing bimodules}
    are given by
    \[{}_{C_+} T_{C_-}:=e_{+}A_Ce_{-},\qquad {}_{C_-} T_{C_+}:=e_{-}A_Ce_{+}.\]
    For more general pairs of faces $(C,C')$, we pick a generic path between them that crosses a minimal number of hyperplanes, passing consecutively through faces $C',C_1,\dots,C_n,C,$
    so that $C_i$ and $C_{i+1}$ are adjacent. The corresponding wall-crossing bimodule is defined by
    \begin{equation}
        {}_{C'} T_{C}:={}_{C'} T_{C_1}\otimes_{A_{C_1}}{}_{C_1} T_{C_2}\otimes_{A_{C_2}}\cdots \otimes_{A_{C_n}}{}_{C_{n}} T_{C}.\label{eq:WC-def}
    \end{equation}
\end{definition}
Lemma \ref{lem:quiver-m-a} shows that after completion at a point $h\in G^\vee$, the bimodule ${}_{C'} T_{C}$ with action transported by the isomorphisms to $A_{C'}^{(+)},A_C^{(+)}$ is isomorphic to the additive version of this wall-crossing bimodule ${}_{C'} T^{(+)}_{C}$, compatibly with the multiplication maps ${}_{C''} T_{C'}\overset{L}\otimes {}_{C'} T_{C}\to {}_{C''} T_{C}$.  In particular, this shows that 
\begin{lemma}\label{lem:bim-m-a}
\hfill
	\begin{enumerate}
		\item The multiplication map ${}_{C''} T_{C'}\overset{L}\otimes {}_{C'} T_{C}\to {}_{C''} T_{C}$ is a quasi-isomorphism if and only if the same is true for the additive version of these bimodules ${}_{C''} T^{(+)}_{C'}\overset{L}\otimes {}_{C'} T^{(+)}_{C}\overset{\sim}\to {}_{C''} T^{(+)}_{C}$.
		\item The tensor product functor ${}_{C'} T_{C}\overset{L}\otimes_{A_{C}}-\colon D^+(A_C\mmod)\to D^+(A_{C'}\mmod)$ is an equivalence if and only if the same is true for the additive version ${}_{C'} T^{(+)}_{C}\overset{L}\otimes_{A^{(+)}_{C}}-\colon D^+(A^{(+)}_C\mmod)\to D^+(A^{(+)}_{C'}\mmod)$.
	\end{enumerate}
\end{lemma}
\begin{proof}
	Both of these are questions whether maps of complexes are isomorphisms, which is suffices to check after completion at all points in $G^\vee$.  This is manifest in the first case.  In the second, this is because ${}_{C'} T_{C}\overset{L}\otimes_{A_{C}}-$ is an equivalence if and only if the unit and counit of its adjunction with $ \mathbb{R}\Hom_{A_{C'}}({}_{C'} T_{C},-)$ are isomorphisms, i.e. if the natural maps
	\[  A_{C}\to \mathbb{R}\Hom_{A_{C'}}({}_{C'} T_{C},{}_{C'} T_{C})\qquad {}_{C'} T_{C}\overset{L}\otimes_{A_{C}}\mathbb{R}\Hom_{A_{C'}}({}_{C'} T_{C},A_{C'})\to A_{C'}\]
	are both isomorphisms.  This completes the proof.
\end{proof}

Since \cite[Lem. 4.6]{WebcohI} (transfered to the multiplicative setting through the completion) guarantees that higher Tors vanish, the defining equation \eqref{eq:WC-def} of ${}_{C'}T_C$ is actually a non-derived tensor product.

That the procedure used to define ${}_{C'}T_C$ gives a well-defined wall-crossing functor is a consequence of the schober property, proved in Theorem~\ref{thm:b-schober} below. 
For our purposes, it is more useful to think of the algebras $A_C$ as the basic objects, and the wall-crossing functors as manifestations of them, since they encode the full schober structure.

\begin{definition}
	The {\bf B-side schober}  assigns the dg-category $A_C\dgmod$ to a face $C$ and the functor ${}_{C'} T_{C}\Lotimes-$ as the transition functor $\phi_{C'C}$.  
\end{definition}

This name is justified by the following theorem:
\begin{theorem}\label{thm:b-schober}
	The above assignment  defines a perverse schober in the sense of Definition~\ref{def:schober}.
\end{theorem}
\begin{proof}
Note that the additive analogue of this structure is shown to be a perverse schober in \cite[Theorem 4.15]{WebcohI}.  This is done by constructing an equivalence of finite dimensional $A_C$-modules with modules over certain quantizations of hypertoric varieties in characteristic $p$.  We point this since the proof of the schober conditions \cite[Theorem 4.11]{WebcohI} uses this language.  However, the proof just uses properties of wall-crossing functors which transport through all equivalences in the picture, exactly those appearing in Lemma \ref{lem:bim-m-a}.

We first need to establish the required isomorphisms of functors for conditions of (3-4).  As shown in the proof of  \cite[Theorem 4.11]{WebcohI} (points (1) and (3); note that the schober conditions are numbered differently here), these isomorphisms are induced by multiplication maps as in Lemma \ref{lem:bim-m-a}(1).  The same is true of the isomorphism needed in (i) (which corresponds to point (2) in the proof of \cite[Theorem 4.11]{WebcohI}).

Finally, point (ii) (which matches point (4) in the proof of \cite[Theorem 4.11]{WebcohI}) simply requires a wall-crossing functor to be an equivalence of categories, which we can transport from the additive case to the multiplicative by Lemma \ref{lem:bim-m-a}(2).
\end{proof}


\subsection{The A-side schober}\label{subsec:a-schober}

We now propose a construction of a schober starting from a family of weakly Weinstein manifolds; we do not show that this construction satisfies the schober axioms in general, but we check that in the case of interest here, we do obtain a schober, which is moreover equivalent to the B-side schober defined above.


Constructible-sheaf methods for the construction of A-side perverse schobers have appeared before in \cite{Nad-LG,Don-Ku,Nad-wall}, though always in the 1-dimensional case (the first two in the double-cut realization on $\C$, the last in a slightly more involved arrangement). We will be guided in particular by the geometry of \cite{Nad-LG}. 
%
%
%

Thus, let $f:\sfamq\to \ft^\vee$ be a symplectic fibration of weakly Weinstein manifolds over a complexified vector space $\ft^\vee=\ft^\vee_\R\otimes\C,$ whose set of critical values is a real hyperplane arrangement $\Dper$ in $\ft^\vee_\R\subset\ft^\vee.$ We assume that the stratification induced by this arrangement lifts to a stratification of the critical locus by weakly Weinstein submanifolds. 
\begin{definition}
	A {\bf schober skeleton} for the family $\sfamq$ is a Lagrangian $\bsL\subset \sfamq$ with the following properties:
	\begin{itemize}
		\item For any point $\parr\in \ft^\vee_\R$ not lying in a hyperplane, the fiber $\bsL_\parr:=\bsL\cap \sfam_\parr$ of $\scrL$ is a skeleton of the fiber $\sfam_\parr$.
		\item For any exit path $\parr'\to\parr$ in $\ft^\vee_\R$ with respect to the stratification of $\ft^\vee_\R,$ the fiber $\bsL_{\parr'}$ is the symplectic parallel transport of $\bsL_{\parr}$ along the path. 
	\end{itemize}
\end{definition}

For $C$ a stratum in $\ft^\vee_\R,$ we write $\sfamq_C$ for the preimage $f^{-1}(\text{star}(C))$ of the star of $C$, and $\bsL_C:=\bsL\cap \sfamq_C$ for the part of $\bsL$ lying over the star of $C.$  We expect that the following procedure will produce the A-side perverse schober associated to the family $\sfamq$:
\begin{definition} \label{def:aside-schober}
    To a schober skeleton $\bsL$ as above, we associate the following set of data:
    \begin{itemize}
        \item To a stratum $C$ of $\ft^\vee_\R,$ we assign the category $\scrD_C=\mu\Shc_{\bsL}(\bsL_C)$ of sections of the ``microlocal sheaves'' cosheaf $\mu\Shc_{\bsL}$ over $\bsL_C.$
        \item For an inclusion $C'\leq C$ of strata, we define restriction functors $\delta_{C'C}:\scrD_C\to\scrD_{C'},$ by functoriality of the cosheaf $\mu\Shc_{\bsL}$ along the inclusion $\bsL_{C}\to\bsL_{C'}.$
    \end{itemize}
\end{definition}

Note that the generalization functors are fixed as the adjoints of the restrictions, so the above data is sufficient to characterize a schober. However, we do not check that the resulting structure does in fact satisfy the requirements in the definition of a perverse schober.

Now we are ready to almost ready to begin considering the situation of interest to us: the family of affine multiplicative hypertoric varieties indexed by complex moment map parameter $\parone.$ This space is obtained from $\TAno$ by performing a complex Hamiltonian reduction with 0 real moment map parameter and unfixed complex moment map parameter; in other words, it is the GIT quotient $\TAno/\!\!/ T.$ This space has a residual complex moment map to $T^\vee.$ Since we would like a schober on $\ft^\vee$ rather than $T^\vee,$ we must pull back this space along the universal cover $\ft^\vee\to T^\vee$; we denote the resulting pullback by $\sfamq.$

As desired, this space carries a map $\sfamq\to \ft^\vee,$ with fiber $\sfamq_\parone$ over $\parone\in \ft^\vee$ given by the multiplicative hypertoric variety $\mht_{(\parone, 1)}$. The critical locus of this map is the hyperplane arrangement $\Dper$; the fibers become progressively more singular as one moves to higher-codimension faces of the arrangement.

Now we would like to apply Definition~\ref{def:aside-schober}, with schober skeleton given fiberwise by the usual skeleton $\LL$ on smooth fibers, and its degenerations on singular fibers. However, this doesn't literally make sense as stated, since the total space of the family $\sfamq$ is not smooth.

Nevertheless, $\sfamq$ is the Hamiltonian reduction of a smooth space. Write $\sfam$ for the pullback of $\TAno$ along the universal cover $\ft^\vee\to T^\vee$: now, by definition, $\sfamq$ is obtained from $\sfam$ as a Hamiltonian reduction by the compact torus $T_\R.$ Using the heuristic that the geometry of $\sfamq$ is the $T$-equivariant geometry of $\sfam,$ we can  associate a schober to $\sfamq$ by starting with a schober associated to $\sfam$ and passing to $T$-invariants.

So suppose that $\scrL$ is a schober skeleton for $\sfam$ which is invariant under the action of the torus $T_\R$ on $\sfam.$ (In our case, we have such a Lagrangian, given by the preimage of $\bsL$ under the $T_\R$-quotient.)

Then for any $T_\R$ invariant open subset $U$, the category $\mu\Sh^c_\scrL(U)$ is graded by the character lattice of $T$, and we can pass to the $T$-invariant category $\mu\Sh^c_\scrL(U)^T$ by picking out the zero-weight part. In other words, $\mu\Sh^c_\scrL(U)$ is the category of modules over an algebra $A$ with a natural map from the group algebra $\C[\pi_1(T)],$ and we define $\mu\Sh^c_{\scrL}(U)^T$ as modules over the algebra obtained from $A$ by imposing the relation $\gamma=1$ for all monodromies $\gamma.$

We therefore define the conjectural schober associated to the family $\sfamq$ as follows:

\begin{definition} \label{def:eqvt-aside-schober}
    Suppose that $\scrL$ is a schober skeleton for $\sfam$ which moreover is $T_\R$-equivariant. Then we associate to $\scrL$ the following set of data:
    \begin{itemize}
        \item To a stratum $C$ of $\ft^\vee_\R,$ the we assign the category $\scrD_C=\mu\Shc_{\scrL}(\scrL_C)^T.$
        \item For an inclusion $C'\leq C$ of strata, we define restriction functors $\delta_{C'C}:\scrD_C\to\scrD_{C'},$ by functoriality of the cosheaf $(\mu\Shc_{\scrL})^T$ along the inclusion $\scrL_{C}\to\scrL_{C'}.$
    \end{itemize}
\end{definition}
As for Definition~\ref{def:aside-schober}, we will not show that this defines a schober in general. However, for family $\sfam$ and equivariant schober skeleton $\scrL$ as above, we can check that this does define a schober by computing it explicitly:
\begin{theorem}\label{thm:schobers}
    For $\sfam$ and $\scrL$ as above, the data defined by Definition~\ref{def:eqvt-aside-schober} are isomorphic to the data of the B-side perverse schober from Definition~\ref{def:quotientschobquiver}. (In particular, Definition~\ref{def:eqvt-aside-schober} does in fact define a perverse schober in this case.)
\end{theorem}
\begin{proof}
First of all, to simplify the situation, we pass from $\sfam$ to its universal cover, which we can do relative to the map $\sfam\to \ft^\vee$, and work equivariantly with respect to the group of deck transformations. Thus, let $\sfam_{sc}$ be the pullback of $\TAno$ along the map $\fd^{\vee} \to D^{\vee}$ and $\scrL_{sc}$ the pullback of its skeleton.
These both carry a free fiberwise action of $\pi_1(D/T) = \fg^{\vee}_\Z,$ with respective quotients $\sfam$ and $\scrL$. 

The periodic hyperplane arrangement $\scrH$ in $\mathfrak{d}^{\vee}_\R$ which encodes the structure of the skeleton $\scrL_{sc}$ is particularly simple: the arrangement $\scrH$ is given by all integer translates of the coordinate hyperplanes. In other words, the skeleton $\scrL_{sc}$ is a $\Z^n$ worth of copies of $(\mathbb{P}^1)^n,$ glued together along their boundaries.

Let $C$ be a stratum in $\ft^{\vee}_\R$. Recall that we write $U_C$ for  the preimage of the star of $C$ under the projection map $\fd^\vee_\R\to \ft^\vee_\R$, and $\scrH_C$ for the hyperplane arrangement in $U_C$ defined by the intersection $U_C \cap \scrH$. This hyperplane arrangement is the moment-map X-ray for the Lagrangian $\scrL_{sc, C}.$ 
Thus, for each $C,$ the category $\mu\Sh^c_{\scrL_{sc}}(\scrL_{sc, C})$ is the dg category of modules over the quiver with relations $\scrQ_C$ appearing in Definition \ref{def:quotientschobquiver}. Passing to $\fg_\Z^\vee$-invariants, we conclude that the category $\mu\Sh^c_{\scrL}(\scrL_{C})$ is the dg category of modules over the quotient of this quiver by $\fg^{\vee}_\Z$.


Finally, for each $C$, we need to impose we need to impose $T$-equivariance on the category $\scrQ_C/\fg^{\vee}_\Z\dgmod.$ As in the discussion preceding Definition~\ref{def:eqvt-aside-schober}, this is accomplished by imposing the relations $\gamma = 1$ for $\gamma \in \pi_1(T)$. Thus, we have an equivalence
\[\scrD_C \cong A_C \dgmod\]
between the category $\scrD_C$ which the equivariant A-side schober assigns to $C$ and the category $A_C\dgmod$ assigned to $C$ by the B-side schober of Definition~\ref{def:quotientschobquiver}.
And in either case, for an inclusion $C'\leq C$ of strata, the specialization functor $\delta_{CC'}$ is obtained from the inclusion of quivers $\scrQ_C\hookrightarrow \scrQ_{C'}$ corresponding to the inclusion $\scrH_C\subset \scrH_{C'}$ of hyperplane arrangements. Since these specialization functors determine all the data of the perverse schober, we conclude that the two schobers are equivalent.
%
\end{proof}

\appendix
\renewcommand{\theequation}{\thesection.\arabic{equation}}
\section{Symplectic forms on Hamiltonian reductions}\label{sec:kirwan}
We recall here how a symplectic form is determined by a choice of Hamiltonian reduction parameter. The material here is mostly standard: see for instance \cite{guillemin-sternberg} for equivariant symplectic forms and \cite{kirwan} for the Kirwan map.

\begin{lemma}
Let $(X,\omega)$ be a symplectic manifold with a Hamiltonian action of a group $G$. Then the moment map $\mu:X\to \mathfrak{g}^\vee$ determines a lift of $\omega$ to an equivariant 2-form $\omega_\mu\in \Omega^2_G(X).$
\end{lemma}
\begin{proof}
Recall that the Cartan model for equivariant de Rham cohomology $\Omega^\bullet_G(X)$ presents it as the complex $(\Sym\mathfrak{g}^\vee\otimes \Omega^\bullet(X))^G,$ which we may understand as the space of equivariant polynomial functions $f:\mathfrak{g}\to \Omega^\bullet(X),$ equipped with differential
\[
(df)(\xi) =  d_{dR}(f(\xi))-\iota_\xi f(\xi).
\]
We define the equivariant symplectic form $\omega_\mu$ lifting $\omega$ by
\[
\omega_\mu(\xi):= \omega + \langle \mu,\xi\rangle.
\]
This equivariant form is closed, since
\begin{align*}
    (d\omega_\mu)(\xi) &= d_{dR}(\omega_\mu(\xi)) - \iota_\xi \omega_\mu(X)\\
    &= d_{dR}\omega + d_{dR}\langle \mu,\xi\rangle - \iota_\xi \omega\\
    &= 0
\end{align*}
because $\omega$ is closed and the other two terms cancel by the definition of moment map.
\end{proof}

Using the equivariant lift $\omega_\mu$ of $\omega,$ it is easy to describe the symplectic form $\omega_\alpha$ on a Hamiltonian reduction $X/\!\!/_\alpha G.$ Recall that the {\em Kirwan map} is the map
\begin{equation}\label{eq:kirwan}
\Omega^\bullet_G(X) \to \Omega^\bullet(X/\!\!/_{\!\alpha} G)
\end{equation}
given by pullback along the inclusion $X/\!\!/_\alpha G = \mu^{-1}(\alpha)/G \hookrightarrow X/G.$ (Alternatively, using the Kempf-Ness theorem to present the Hamiltonian reduction as a GIT quotient, the Kirwan map is the pullback along the inclusion of the $\alpha$-semistable locus into the stack $X/G$.) We begin with the case where the Hamiltonian reduction parameter $\alpha$ is zero:
\begin{lemma}
The symplectic form on $X/\!\!/ G$ is the image of $\omega_\mu$ under the Kirwan map \eqref{eq:kirwan}.
\end{lemma}
\begin{proof}
This follows directly from the definition of the symplectic form on a Hamiltonian reduction.
\end{proof}
One may therefore also obtain the symplectic form on reductions at non-zero moment-map parameter $\alpha\in (\mathfrak{g}^\vee)^G$ by shifting the symplectic form $\omega_\mu.$
\begin{corollary}
    The symplectic form on $X/\!\!/_\alpha G$ is the image of $\omega_{\mu-\alpha}$ under the Kirwan map \eqref{eq:kirwan}.
\end{corollary}

    We now explain how a complexified moment-map parameter can be used to specify a complexified symplectic form on a Hamiltonian reduction.

\begin{definition}
    Let $(M,\omega)$ be a symplectic manifold, and let $B\in H^2(M;\R).$ Then we say that $B$ is a {\bf B-field} for $M$, and that $\omega+iB$ is a {\bf complexified symplectic form}.
\end{definition}

We would like to construct a B-field on the Hamiltonian reduction of a manifold $X$ by a compact torus $G=T_\R$. Let $\Phi:H^2_{T_\R}(\text{pt}) \to H^2(X/\!\!/_\alpha T_\R)$ denote the precomposition of the Kirwan map \eqref{eq:kirwan} with the pullback $H^2_{T_\R}(\text{pt})\to H^2_{T_\R}(X),$ and note that the domain of $\Phi$ may be identified with $\mathfrak{t}^\vee.$

\begin{definition}\label{def:bfield-reduction}
    Let $T_\R$ be a torus with a Hamiltonian action on a K\"ahler manifold $X$,
 and let $\alpha\in T^\vee$ be an element of the complexified dual torus, which we write as $\alpha=\gamma\cdot \exp(\delta),$ for $\delta\in \mathfrak{t}_\R^\vee$ and $\gamma \in T_\R^\vee.$ Then we write
 \[
    X/\!\!/_\alpha T_\R := (X/\!\!/_\delta T_\R, \omega_\delta + i\Phi(\widetilde{\gamma}))
 \]
 for the symplectic manifold $(X/\!\!/_\delta T_\R,\omega_\delta)$ equipped with B-field $\Phi(\widetilde{\gamma}),$ where $\widetilde{\gamma}\in \mathfrak{t}^\vee_\R$ is any lift of $\gamma$ along the covering map $ \mathfrak{t}^\vee_\R\to T^\vee_\R.$
 \end{definition}
 As the notation suggests, what is fundamental for us is not $\widetilde{\gamma}$ but rather the class $\gamma\in T_\R^\vee = H^2_{T_\R}(\text{pt}; \R/\Z),$ since the B-field contributes to Floer-theory calculations through its exponential, which is $\Z$-periodic.

\section{SYZ duality}\label{sec:syz-appendix}

From the perspective of SYZ mirror symmetry, the most important feature of the multiplicative hypertoric varieties $\mht_{(\parone,\partwo)}$ is the presence of a special Lagrangian torus fibration, whose properties we describe here. We begin as usual with the basic case.

\begin{example}[{\cite[Section 5.1]{Auroux07}}]\label{ex:torusfib}
Consider the quasi-hyperhamiltonian moment map $\mmm_\Hbb:(T^*\C)^\circ \to \C^\times\times \R,$ which we can present more invariantly as a map with codomain $\R\times S^1\times \R,$ given by
\[ 
(\sz,\sw)\longmapsto
(\log|\sz\sw+1|,\Arg(\sz\sw+1),|\sz|^2-|\sw|^2).
\]
Consider the map
\[\mmm_{I,\C}:(T^*\C)^\circ\to \R^2, \qquad (\sz,\sw)\mapsto (\log|\sz\sw+1|,|\sz|^2-|\sw|^2)\] obtained by composing $\mmm$ with the projection onto the two copies of $\R.$ The fiber of this map over any nonzero point is a special (with respect to holomorphic volume form $\Omega_K$) Lagrangian 2-torus, and the fiber over the central point is a nodal curve. 
\end{example}

The paper \cite{Auroux07} describes a mirror construction beginning from the symplectic manifold $\TAo$ and producing a moduli space of Lagrangian tori. This construction was pursued systematically in a more general context in \cite{AAK}, and we will recall here the main theorem of that paper as it relates to the spaces $\TAo$ and $\TAno.$ The notion of SYZ duality considered in that paper was as follows:

\begin{definition}[{\cite[Definition 1.2]{AAK}}]\label{defn:syz}
    For $X$ a K\"ahler manifold equipped with a Lagrangian torus fibration $\pi:X\to B,$ we say that $Y$ is {\em SYZ mirror} to $X$ if $Y$ is a completion of a moduli space of unobstructed torus-like objects of $X$ supported on SYZ fibers.
\end{definition}

The work \cite{AAK} contains several different ``flavors'' of SYZ mirror theorems; we are interested in the one concerning conic bundles with a prescribed degeneracy locus. Let
\[
X^0=\{(\mathbf{x},\mathbf{y},\mathbf{z})\in (\C^\times)^n\times \C^d\times \C^d\mid f_i(\mathbf{x})=\sy_i\sz_i\}
\]
be the iterated conic bundle over $\C^n$ degenerating where $f_i(\mathbf{x})=0.$

Its mirror will be described as the complement of a hypersurface in a toric variety. Let $\varphi_i$ be the tropicalization of the function $f_i,$ and define a polytope
\[
\Delta_Y = \{(\mathbf{\xi}, \eta_1,\ldots,\eta_d)\in \R^n\times \R^d\mid \eta_i\geq \varphi_i(\mathbf{\xi}).\}
\]
Let $Y$ be the toric variety with moment polytope $\Delta_Y$, and for $1\leq i\leq d,$ let $v_i:Y\to \C$ be the toric monomial with weight $(\mathbf{0},0,\ldots,1,\ldots,0),$ whose only nonzero entry is the $(n+i)$th entry. Define the function $w_i:Y\to \C$ by $w_i=v_i-1,$ and finally let
\[
Y^0 = Y\setminus \left\{\prod_{i=1}^d w_i =0\right\}
\]
be the subset of $Y$ where all the $w_i$ are nonzero.

\begin{theorem}[{\cite[Theorem 11.1]{AAK}}]\label{thm:aak-syz}
    $Y^0$ is SYZ mirror to the iterated conic bundle $X^0$.
\end{theorem}

We will be interested in a special case of this theorem:
\begin{corollary}\label{cor:tan-syz}
    The space $\TAno$ is SYZ mirror to itself.
\end{corollary}
\begin{proof}
  First, consider the space $X^0$ as above, in the case where $d=n$ and the function $f_i(\mathbf{x})$ is defined by $f_i(\mathbf{x})=\sx_i+1.$ In this case, the space $X^0$ is evidently equal to $\TAno,$ so we need only to show that $Y^0$ is also isomorphic to $\TAno.$
  
  The tropicalization $\varphi_i$ of $f_i$ is given by $\varphi_i(\mathbf{\xi})=\max(\xi_i,0).$ The resulting moment polytope $\Delta_Y$ is abstractly isomorphic to the positive orthant in $\R^{2n},$ so that the toric variety $Y$ will be abstractly isomorphic to $\C^{2n} \cong T^*\C^n.$ However, because the embedding of the polytope $\Delta$ is skewed, the toric monomial $v_i$ will be $\sz_i\sw_i$ (in the natural coordinates on $T^*\C^n$), and therefore the complement $Y^0$ will be the locus in $\C^n$ where $\sz_i\sw_i-1\neq 0$ for all $i$.
\end{proof}

Now recall that the multiplicative hypertoric variety $\mht_{(\parone,\partwo)}$ is obtained from the variety $\TAno$ in two steps: (1) Imposing the complex group-valued $T$-moment map equation $\mmm_{\C^\times}=\parone,$ and (2) Performing a Hamiltonian reduction by the real torus $T_\R$ at character $\partwo.$ To understand the relation between these operations, note that the Hamiltonian $T_\R$ action occurs within fibers of the Lagrangian torus fibration on $\TAno$ (which is a product of the fibration described in \ref{ex:torusfib}).

\begin{definition}\label{def:mirrorfib}
    Let $\pi:X\to B$ be a Lagrangian torus fibration (with singular fibers), and suppose that $X$ admits a Hamiltonian action by a torus $T_\R$ which moreover acts within fibers of $\pi.$ This condition ensures that each fiber of $\pi$ is contained in some level set of the moment map $\mu_\R$ for the torus action. Therefore, to a point in the SYZ mirror $X^\vee$ to X (in the sense of Definition \ref{defn:syz}), representing a torus fiber of $\pi$ (equipped with some local system) contained in some level set $\{\mu_\R = \delta\},$ we can associate the value $\exp(\delta)\cdot \gamma,$ where $\gamma\in T_\R^\vee$ is the holonomy of the local system along the $T_\R$-orbit, and assembling these together, we obtain a function $f:X^\vee\to T^\vee$ from the SYZ mirror of $X$ to the complexified dual torus of $T_\R.$ We say that this map is the {\bf mirror fibration} to the Hamiltonian $T_\R$ action on the SYZ fibration $X\to B.$
\end{definition}

This definition does capture essential features of our situation: the proof of \cite[Theorem 11.1]{AAK} proves more than \ref{thm:aak-syz}, but also establishes the following fact:
\begin{proposition}
Under the SYZ duality of Corollary \ref{cor:tan-syz}, the map $\TAno\to (\C^\times)^n$ given by $(\sx_i\sw_i-1)_{i=1}^n$ is the mirror fibration to the Hamiltonian $T^n$ action on $\TAno.$
\end{proposition}

Definition \ref{def:mirrorfib} is meant to suggest the following conjecture about SYZ mirror symmetry (whose ``HMS'' counterpart is \cite[Conjecture 4.2]{Teleman14}):
\begin{conjecture}\label{conj:syz-duality}
    Let $\pi:X\to B, T_\R,$ and $f:X^\vee\to T^\vee$ as in Definition \ref{def:mirrorfib}. Then a Hamiltonian reduction $X/\!\!/_\alpha T_\R$ is SYZ mirror (in the sense of Definition \ref{defn:syz}) to a fiber $f^{-1}(\alpha')$ of $f.$
\end{conjecture}

\begin{remark}
Another way to conceptualize the above conjecture is through the following observation: the Hamiltonian reduction $X/\!\!/_\alpha T_\R$ imposes an equation on the SYZ base and imposes a quotient on SYZ fibers, whereas the equation $\{f = \alpha'\}$ imposes an equation on the SYZ base and imposes an equation on the SYZ fibers. In other words, both operations restrict to a $\text{dim}(T_\R)$-codimension subset of the SYZ base, and on torus fibers they implement the dual operations of passage to subsets or quotients.
\end{remark}

Note that the parameter $\alpha'$ lives in $T^\vee,$ and (using the constructions of Appendix \ref{sec:kirwan} to specify a complexified symplectic form), the parameter $\alpha$ can also be taken to live in $T^\vee.$ Therefore it is natural to guess that one may take $\alpha=\alpha'$ in Conjecture \ref{conj:syz-duality}. Unfortunately, this conjecture is too na\"ive, as it fails to account for corrections coming from disks with boundary on SYZ fibers; see \cite[Section 5]{AAK} for a counterexample to this na\"ive conjecture.)

In our setting, we expect that such instanton corrections do not occur. Therefore, applying Conjecture \ref{conj:syz-duality} twice, with $\alpha = \alpha',$ we conjecture the following:
\begin{conjecture}
    The multiplicative hypertoric varieties $\mht_{(\parone,\partwo)}$ and $\mht_{(\partwo,\parone)}$ are SYZ mirror in the sense of Definition \ref{defn:syz}.
\end{conjecture}

The paper \cite{LZhypertoric} proves a form of SYZ mirror symmetry for hypertoric varieties. Those authors begin with an additive hypertoric variety $\aht$ and then compute, using the techniques of \cite{AAK}, that the complement of a divisor $\divisor$ in $\aht$ is SYZ mirror to a multiplicative hypertoric variety $\mht$.
On the other hand, the above considerations suggest that a multiplicative hypertoric variety $\mht$ should be SYZ mirror to another multiplicative hypertoric variety. The following conjecture is therefore natural:
\begin{conjecture}\label{conj:LZ}
    Let $(\aht, \divisor)$ be the additive hypertoric variety and divisor which are constructed in \cite{LZhypertoric} with the property that the complement $\aht\setminus \divisor$ is a mirror to a multiplicative hypertoric variety. Then $\aht\setminus \divisor$ is itself a multiplicative hypertoric variety.
\end{conjecture}




\input{lemma4-22}

\bibliography{./gen}
\bibliographystyle{amsalpha}

\end{document}

%% file: lemma4-22.tex
\section{Proof of Lemma \ref{lem:quiver-cover}}
\label{appendix:holim-proof}
  \section*{by Laurent C\^ot\'e, Benjamin Gammage, Justin Hilburn, Michael McBreen, and Ben Webster}

This appendix is devoted to a proof of Lemma \ref{lem:quiver-cover}. In fact, we will formulate and prove the slightly more general Lemma \ref{lem:quiver-covergeneralA}, in which $\textgoth{A}_\parone$ is replaced by a general arrangement.   

\begin{remark}
    We will avoid mentioning microlocal sheaves and hypertoric spaces in the argument below. However, the reader may wish to keep the following picture in mind.

    The skeleton $\widetilde{\cL}$ carries a sheaf $\mu Sh$ of dg categories, as well as a sheaf $\muPerv$ of abelian categories. The ultimate objective is to show that $\mu Sh(\widetilde{\cL})$ is obtained from $\muPerv(\widetilde{\cL})$ by taking the dg category of complexes. To do so, one may push forward both sheaves along the torus map $\mu_\R : \widetilde{\cL} \to \mathfrak{g}^{\vee}_\R$. The question then reduces to a comparison of sheaves on a vector space stratified by a hyperplane arrangement. This appendix carries out such a comparison. 
\end{remark}
\subsection{Categories from arrangements}
We fix a real (affine) hyperplane arrangement $\textgoth{A} \subset V$. We allow $\textgoth{A}$ to have infinitely many hyperplanes, but we ask that any compact subset of $V$ intersect only finitely many of these.

This gives a stratification of $V$ whose strata are the (closed) faces $F$ of the various chambers of $\textgoth{A}$. We write $\partial F \subset F$ for the union of all smaller dimensional subfaces and $F^{\star} = F \setminus \partial F$.

We associate two arrangements to a face $F$:
\begin{enumerate}
    \item $\textgoth{A}^F$ is obtained by deleting all the hyperplanes which do not contain $F$.
    \item  Let $V_F \subset V$ be the affine subspace spanned by $F$. Then the image of $\textgoth{A}^F$ in $V/V_F$ defines a central arrangement $\textgoth{A}_F$. 
\end{enumerate}    
Roughly speaking, $\textgoth{A}^F$ describes the arrangement in a neighborhood of a point in $F^{\star}$, whereas $\textgoth{A}_F$ describes a transversal slice to this point.

There are natural bijections between the chambers of $\textgoth{A}$ containing $F$, the chambers of $\textgoth{A}^F$, and the chambers of $\textgoth{A}_F$. We will sometimes abuse notation and use the same symbol for a chamber in all three of these arrangements. 

Definition \ref{def:abstract-quiver} associates a quiver with relations $\cQ$ to the hyperplane arrangement $\textgoth{A}$. The quiver $\cQ^F$ of $\textgoth{A}^F$ is given by the full subquiver of $\cQ$ obtained from the chambers containing $F$. It is naturally identified with the quiver of $\textgoth{A}_F$. 

If $F \subset F'$ is an inclusion of faces, then $\cQ^{F'}$ is a full subquiver of $\cQ^F$, inducing a restriction isomorphism $\cQ^F\dgmod\to \cQ^{F'}\dgmod.$ 

Denote by $\bP$ the face poset of the hyperplane arrangement. We are interested in the homotopy limit of dg categories $\varprojlim_{F\in \bP}\cQ^{F'} \dgmod.$ An object of this limit consists of a collection $(M_F)_{F\in \bP}$ of representations of the various subquivers $M_F,$ together with some gluing data.

If representations of $\cQ^F$ and $\cQ^{F'}$ are both restrictions of the same representation of the full quiver $\cQ,$ then their further restrictions to any shared face of $F,F'$ are canonically identified, so that we have a functor
\begin{equation}\label{eq:limit-functor}
\Phi:\cQ\dgmod\to \varprojlim_{F\in\bP}\cQ^{F}\dgmod.
\end{equation}

The goal of this appendix is to prove the following.
\begin{lemma}\label{lem:quiver-covergeneralA}
The functor $\Phi$ defines a canonical equivalence
between the dg category of $\cQ$-modules and the homotopy limit of the dg categories of $\cQ^F$-modules. 
\end{lemma}

Our strategy is to reduce this claim to a calculation involving concrete objects on both sides. Given a chamber $\Delta$, restriction $M \to M_\Delta$ defines a functor $\cQ \dgmod \to \C[\pi_1 T]\dgmod$, corepresented by an object $P_\Delta$. 

\begin{lemma}\label{lem:image-is-cocore}
    The functor on $\varprojlim_{F\in\bP}\cQ^{F} \dgmod$ that takes an object $(M_F)$ to $M_\Delta$ is corepresented by $\Phi(P_\Delta)$.
\end{lemma}
Before we move on to the proof, we note that it implies our main result:
\begin{proof}[Proof of Lemma \ref{lem:quiver-covergeneralA} given Lemma \ref{lem:image-is-cocore}]
    Lemma \ref{lem:image-is-cocore} proves that $\Phi$ is essentially surjective, since the objects corepresenting $(M_F)\mapsto M_\Delta$ jointly generate the target, and it also implies that $\Phi$ induces equivalences $\Hom(P_\Delta,P_{\Delta'})\simeq \Hom(\Phi(P_\Delta),\Phi(P_\Delta)).$ Since the domain of $\Phi$ is generated by the $P_\Delta,$ we are done.
\end{proof}

To prove Lemma \ref{lem:image-is-cocore}, it will be helpful to give a more geometric interpretation of our set-up. The abelian category of $\cQ^F$-modules is the stalk along $F$ of a sheaf of abelian categories $\cP(-)$ on the vector space $V$ equipped with its usual topology, locally constant along the arrangement $\textgoth{A}$. We have $\cP(V) = \cQ\mmod$. Consider the associated presheaf $DG(\cP(-))$ of dg-categories and its sheafification $\cP_{\operatorname{dg}}(-)$. Then the right-hand side of \eqref{eq:limit-functor} is simply $\cP_{\operatorname{dg}}(V)$, and the lemma claims that the natural map is an isomorphism. 

Consider the constant sheaf of rings $\cR$ on $V$ with fiber $\C[\pi_1T]$. 
If $M,N$ are two objects of $\cP_{\operatorname{dg}}(V)$, then $\Hom(M,N)$ is the hypercohomology of a complex $\cH(M,N) \in D(\cR)$, locally constant with respect to the face stratification of $\textgoth{A}$ and characterized by
\begin{equation} \label{eq:microhomsheaf} \cH(M,N)_F = \Hom_{\cQ^F\dgmod}(M_F, N_F). \end{equation} 
with the natural restriction maps. Here $\cH(M,N)_F$ denotes the stalk of $\cH$ along the interior of the face $F$. 

Now, fix an object $M$ of $\varprojlim\cQ^F\dgmod$ and a chamber $\Delta$, and let $\cH := \cH(\Phi(P_\Delta), M)$. 
We have \begin{equation} \label{eq:whatisH} \cH_{\Delta} = M_\Delta. \end{equation} 
Therefore, Lemma \ref{lem:image-is-cocore} is equivalent to proving that $H^*(V,\cH) = \cH_{\Delta}$.  We will show this below in Corollary~\ref{cor:image-is-cocore2},
whose proof will occupy the rest of the appendix.

\subsection{A partial order on faces}
We start by giving an explicit description of the objects $P_\Delta$. 
\begin{lemma}\label{lem:explicit-projective}
  For each chamber $\delta$ of $\cQ$, we have a natural identification $(P_\Delta)_{\delta} = \C[\pi_1 T]$. Given adjacent chambers $\delta,\delta'$ separated by a hyperplane $H$, the map 
  \[ (P_\Delta)_{\delta} \to (P_\Delta)_{\delta'} \]
  equals the identity if $\delta$ and $\Delta$ are on the same side of $H$ and equals $\gamma_H - \id$ otherwise.
\end{lemma}
\begin{proof}
    This is essentially the calculation which we have already performed in the proof of Theorem \ref{th:tiltingiso}; it follows from the fact that the relations of $\cQ$ allow any path between two vertices in the quiver to be shortened to a ``taut path'' in the terminology of \cite{GDKD}--one which passes through no hyperplane more than once--at the cost of picking up some factors of $(1-\gamma_H)$ when we cancel out a back-and-forth trip across a hyperplane. As a result, the algebra 
    $\cQ$
    has a $\cQ$-bimodule decomposition
    \[
    \cQ\simeq \bigoplus_{\Delta,\Delta'}e_{\Delta'} \cQ e_{\Delta} \simeq \bigoplus e_{\Delta'} \C[\pi T]e_{\Delta},
    \]
    with the projective $P_{\Delta}$ given by
    \[
    P_{\Delta}= \cQ e_{\Delta} = \bigoplus_{\Delta'} e_{\Delta'} \C[\pi_1 T]e_\Delta.\qedhere
    \]
\end{proof}
Lemma \ref{lem:explicit-projective} motivates the following definition. Given a pair of faces $F,F'$ such that $F' \subset F$, we write $F' <_\Delta F$ if there exists a collection of hyperplanes $H_j$ in $\textgoth{A}$, none of which contain $F$, such that:
\begin{enumerate}
    \item  $F' = F \cap H_1 \cap \cdots  \cap H_l$.
    \item The face $F$ and the chamber $\Delta$ are on the same side of each $H_j$.
\end{enumerate}  
When the choice of $\Delta$ is clear, we simply write $F' < F$. This defines a partial order on faces.

\begin{lemma} \label{lem:partialorders} \leavevmode
\begin{enumerate} 
    \item Every face $F$ is contained in a unique chamber $\Delta(F)$ such that $F \leq \Delta(F)$. \label{item1}
    \item If $F \leq F'$, then $\Delta(F) = \Delta(F')$. \label{item2}
    \item The partial order on $\textgoth{A}^F$ defined by $\Delta(F)$ equals the partial order obtained by restriction from $\textgoth{A}$.\label{item3}
\end{enumerate}
\end{lemma}
\begin{proof}
We may suppose $F$ is not a chamber. Consider the set of hyperplanes containing $F$. Each such hyperplane divides $V$ into two halfspaces. The chambers containing $F$ are given by picking one of these half-spaces for each such hyperplane. The only chamber $\delta$ satisfying $F < \delta$ is obtained by choosing in each case the half-space containing $\Delta$. This proves the claim \eqref{item1}. Now suppose $F \leq F'$. Then we have $F' \leq \Delta(F')$ and thus $F \leq \Delta(F')$. Therefore, the claim \eqref{item2} follows from \eqref{item1}. Finally, the claim \eqref{item3} follows from the fact that $\Delta(F)$ and $\Delta$ lie on the same side of every hyperplane containing $F$. 
\end{proof}
Consider the image of $P_{\Delta}$ under the  restriction functor $\cQ\dgmod \to \cQ^F\dgmod$. By Lemma \ref{lem:explicit-projective}, it is determined by the restriction of the partial order on $\textgoth{A}$ to $\textgoth{A}^F$. By Lemma \ref{lem:partialorders}, this equals the partial order defined by $\Delta(F)$. The image must therefore be the projective $P_{\Delta(F)}.$ 
    This is the algebraic statement at the heart of our proof, but we need to restate it in sheaf-theoretic terms to carry the argument through, as we will do in the next section. 

 
\subsection{Receding sheaves}

Consider a general complex $\cG \in D(\cR)$, locally constant with respect to $\textgoth{A}$.\footnote{From now on, one may take $\cR$ to be an arbitrary constant sheaf of rings.}

\begin{definition}
We say $\cG$ {\em recedes} from $\Delta$ if, whenever $F' < F$, the generalization map $\cG_{F'} \to \cG_{F}$ is an isomorphism.
\end{definition}

\begin{lemma} \label{lem:Hrecedes}
 The sheaf $\cH$ recedes from $\Delta$.
 \end{lemma}
\begin{proof}
Combine \eqref{eq:whatisH} and Part (2) of Lemma \ref{lem:partialorders}.
\end{proof}


Let $ss(\cG) \subset T^* V$ be the singular support. It is a closed, conical Lagrangian subset contained in the union $\Lambda$ of the conormal bundles $T^*_{F^{\star}} V$, where $F$ ranges over the faces of $\textgoth{A}$. It follows from \cite[8.3.13]{KS94}, that $ss(\cG)$ is in fact a union of subsets of the following forms:
\begin{enumerate}
\item $F^{\star}$ where $F$ is a chamber. 
\item The closure of either component of $T^*_{F^{\star}} V \setminus F^{\star}$, where $F$ is face of codimension one. 
\item $T^*_{F^{\star}} V$ where $F$ is a face of codimension $\geq 2$.
\end{enumerate}

Given a face $F$, the chamber $\Delta(F)$ defines a closed cone in the normal bundle $T_{F^{\star}} V$. Let $\Delta(F)^{\circ} \subset T^*_{F^{\star}} V$ be the polar dual of this cone. This consists of covectors which pair non-positively with elements of the cone. Consider the union \begin{equation} \label{eq:unionoverfaces}
    \Delta(\Lambda) = \bigcup_{F} \Delta(F)^{\circ} \subset \Lambda.
\end{equation}
\begin{lemma}
\label{lem:microsupports}
Suppose that $\cG$ recedes from $\Delta$.  Then $ss(\cG)$ lies in the closure of $\Delta(\Lambda)$. 
\end{lemma}
\begin{proof}
Fix $(x,p) \notin \Delta(\Lambda)$. We must show $(x,p) \notin ss(\cG)$. We have $x \in F^{\star}$ for some face of $\textgoth{A}$. We may assume that $(x,p)$ does not lie in the closure of $T^*_{(F')^{\star}} V$ for any $F' \neq F$. 

First, we consider the specialization of $\cG$ to $F^{\star}$. This is a conical sheaf on the normal bundle $T_{F^{\star}} V$, which is defined in \cite[4.2]{KS94}. We denote it $\nu_{F} \cG \in D^b(T_{F^\star} V)$. 

Because $\cG$ is smooth with respect to a polyhedral stratification, its specialization may be characterized as follows. 
Fix a tubular neighborhood $U$ of $F^{\star} \subset T_{F^{\star}}V$. There exists a neighborhood $U'$ of $F^{\star} \subset V$ and an isomorphism $U \cong U'$ restricting to the identity on $F^{\star}$. Then $\nu_F \cG$ is the unique conical sheaf whose restriction to $U$ coincides with the restriction of $\cG$ to $U'$.

We can be more explicit. There is a natural identification $T_{F^{\star}} V \cong V/V_F \times F^{\star}$, where the latter carries the fiberwise arrangement $\textgoth{A}_F \times F^{\star}$. The specialization $\nu_F \cG$ is pulled back from a conical sheaf $\nu_F \cG^{'}$ on $V/V_F$, smooth with respect to $\textgoth{A}_F$. The stalks and generalization maps of $\nu_F \cG^{'}$ coincide with those of $\cG$, after identifying the faces of $\textgoth{A}_F$ with the faces of $\textgoth{A}$ containing $F$.




 To show that $(x,p) \notin ss(\cG)$, we must show that the stalk of the Fourier-Sato transform of $\nu_F \cG^{'}$ vanishes at $(x,q)$ for all $q$ in a neighborhood of $p$. The stalk is given by  
\[ H_c^*(K_q, \nu_F \cG^{'}) \]
where $K_q = \{ y \in V/V_F : \langle q, y \rangle \geq 0 \}$. Since $\nu_F \cG^{'}$ is conical, its compactly supported cohomology equals $H_o^*(K_q, \nu_F \cG^{'})$, the cohomology with support at the origin $o \in V/V_F$. 

If $(x,p) \in T^*_{F^{\star}}V$ but $(x,p) \notin \Delta(F)^{\circ}$, then for all $q$ in a sufficiently small neighborhood, $K_q \cap \Delta(F)$ has non-empty interior. In this case, the isomorphism $(\nu_F \cG^{'})_F \to (\nu_F \cG^{'})_{\Delta(F)}$ remains an isomorphism after restricting to $K_q$. This shows $H_o^*(K_q, \nu_F \cG^{'}) = 0$.  

\end{proof}

\begin{remark}
The singular support of a sheaf is a union of conical Lagrangians. On the other hand, $\Delta(F)^{\circ}$ contains such a Lagrangian only when $F$ is a chamber or a facet. It follows that we can restrict the union in \eqref{eq:unionoverfaces} to chambers and facets. We will not use this below.
\end{remark}
\begin{lemma} \label{cor:restisisom}
Suppose that $\cG$ recedes from $\Delta$. The restriction map $H^*(V, \cG) \to \cG_{\Delta}$ is an isomorphism.
\end{lemma}
\begin{proof}
Fix an inner product on $V$ and a point $x_0$ in the interior of $\Delta$. Let $\phi = |x-x_0|^2$.    

We compute $H^*(V,\cG)$ by taking the limit along the increasing sequence of opens $V_r = \phi^{-1}[0,r)$. Given a hyperplane $H$ in $\textgoth{A}$, let $v$ be the normal vector pointing away from $\Delta$. Then $(d\phi(x), v)$ is positive for any $x \in H$. It follows that $d\phi(x) \notin ss(\cG)$ for all $x \neq x_0$.

By the microlocal Morse lemma \cite[Corollary 5.4.19]{KS94}, the sequence of cohomology groups $H^*(V_r, \cG)$ is constant and thus equal to its initial term $\cG_{x_0} = \cG_{\Delta}$.
\end{proof}
By Lemma~\ref{lem:Hrecedes}, the sheaf $\cH$ recedes from $\Delta$, so we then find that:
\begin{corollary} \label{cor:image-is-cocore2} 
$  H^*(V,\cH) = \cH_{\Delta}. $
\end{corollary}
This completes the proof of Lemma~\ref{lem:image-is-cocore} and thus of Lemma~\ref{lem:quiver-covergeneralA}.